\documentclass[letterpaper]{aims} % Use the aims.cls file to compile your paper
\usepackage{amsmath}
\usepackage{paralist}
\usepackage[misc]{ifsym}
\usepackage{epsfig}
\usepackage{epstopdf}
\usepackage[colorlinks=true]{hyperref}
\hypersetup{urlcolor=blue, citecolor=red}
\allowdisplaybreaks

\def\currentvolume{X}
 \def\currentissue{X}
  \def\currentyear{200X}
   \def\currentmonth{XX}
    \def\ppages{X-XX}
     \def\DOI{\enspace }%{10.3934/xx.xxxxxxx}

\usepackage{amssymb}
\usepackage{mathtools}
\usepackage[nameinlink]{cleveref}
\usepackage{csquotes}
\usepackage{bm}
\usepackage{enumitem}
\usepackage{units}
\usepackage{accents}
\usepackage{stmaryrd}
\usepackage{scalerel}
\usepackage{appendix}
\usepackage{xcolor}

\usepackage[backend=biber, style=numeric, maxnames=5, giveninits=true]{biblatex}
\bibliography{literature}

\newbibmacro{string+doi}[1]{%
  \iffieldundef{doi}{ \iffieldundef{url}{#1}{\href{\thefield{url}}{#1}}}{\href{http://dx.doi.org/\thefield{doi}}{#1}}}
\DeclareFieldFormat{title}{\usebibmacro{string+doi}{\mkbibemph{#1}}}
\DeclareFieldFormat[article,inbook,incollection,inproceedings,patent,thesis,unpublished]{title}{\usebibmacro{string+doi}{\mkbibquote{#1\isdot}}}

\renewbibmacro*{volume+number+eid}{%
  \printfield{volume}%
  \setunit*{\addnbthinspace}% 
  \printfield{number}%
  \setunit{\addcomma\space}%
  \printfield{eid}}
\DeclareFieldFormat[article]{number}{\mkbibparens{#1}}

\renewcommand{\rho}{\varrho}
\renewcommand{\epsilon}{\varepsilon}
\renewcommand{\phi}{\varphi}
\renewcommand{\theta}{\vartheta}

\newcommand*{\matr}[1]{\mathbf{#1}}
\newcommand*{\vct}[1]{{\bm{#1}}}

\newcommand{\D}{{\operatorname{d}}}
\newcommand{\rmf}{{\mathrm{f}}}
\newcommand{\rmT}{{\mathrm{T}}}
\newcommand*{\trp}{\mathrm{t}}
\newcommand*{\dimq}[1]{{\tilde{#1}}}
\newcommand*{\refq}[1]{{#1}^\star}

\newcommand{\partialDer}[2]{\frac{\partial #1}{\partial #2}}

\newcommand{\norm}[1]{{\left\Vert #1 \right\Vert}}
\newcommand{\normsize}[2]{{#1\Vert #2 #1\Vert}}
\newcommand{\abs}[1]{{\left\vert #1 \right\vert}}
\newcommand{\abssize}[2]{{#1\vert #2 #1\vert}}
\newcommand*{\jump}[1]{\left\llbracket #1 \right\rrbracket}
\newcommand*{\jumpsize}[2]{{#1 \llbracket #2 #1\rrbracket}}
\newcommand{\normiii}[1]{{\left\vert\kern-0.25ex\left\vert\kern-0.25ex\left\vert #1 
    \right\vert\kern-0.25ex\right\vert\kern-0.25ex\right\vert}}
\newcommand{\normsizeiii}[2]{{#1\vert\kern-0.25ex #1\vert\kern-0.25ex #1\vert #2 
    #1\vert\kern-0.25ex #1\vert\kern-0.25ex #1\vert}}

\makeatletter
\newcommand{\hathat}[1]{% 
\begingroup%
  \let\macc@kerna\z@%
  \let\macc@kernb\z@%
  \let\macc@nucleus\@empty%
  \hat{\raisebox{.3ex}{\vphantom{\ensuremath{#1}}}\smash{\hat{#1}}}%
\endgroup%
}
\makeatother

\newcommand{\inlineeqnum}{\refstepcounter{equation}~~\mbox{(\theequation)}}

\DeclareMathOperator*{\argmin}{arg\,min}

\newcommand*{\ubar}[1]{\underaccent{\bar}{#1}}
\newcommand{\invpsi}{\ubar{\vct{\psi}}}
\newcommand{\invT}{\ubar{\vct{T}}}
\newcommand{\invt}{\ubar{t}}
\newcommand{\invY}{\ubar{\mathcal{Y}}}
\newcommand{\invR}{\ubar{\mathcal{R}}}

\newcommand*{\nablaGamma}{\nabla_{\!\Gamma}}
\newcommand{\nablaN}{\nabla_{\!\vct{N}}}

\newcommand*{\smallpar}{{\scaleobj{.8}{\parallel}}}
\newcommand*{\smallperp}{{\scaleobj{.8}{\perp}}}
\newcommand*{\smallo}{{\scalebox{.9}{$\scriptstyle\mathcal{O}$}}}

\newcommand\StepSubequations{
  \stepcounter{parentequation}
  \gdef\theparentequation{\thesection.\arabic{parentequation}}
  \setcounter{equation}{0}
}

\numberwithin{equation}{section}

\newtheorem{theorem}{Theorem}[section]

\newtheorem{lemma}[theorem]{Lemma}
\newtheorem{proposition}[theorem]{Proposition}

\theoremstyle{definition}
\newtheorem{definition}[theorem]{Definition}
\newtheorem{remark}[theorem]{Remark}

\title[Derivation of Discrete Fracture Models]
{Rigorous Derivation of Discrete Fracture Models for Darcy Flow in the Limit of Vanishing Aperture} %

\author[Maximilian Hörl and Christian Rohde]{}

\subjclass{Primary: 76S05, 58J90; Secondary: 35Q35, 35B40.}
\keywords{Fractured porous media, discrete fracture model, weak compactness, vanishing aperture, Darcy flow}

\thanks{$^*$Corresponding author: Maximilian Hörl}

\begin{document}
\maketitle

\centerline{\scshape
Maximilian Hörl$^{{\href{mailto:maximilian.hoerl@mathematik.uni-stuttgart.de}{\textrm{\Letter}}}*1}$
and Christian Rohde$^{{\href{mailto:christian.rohde@mathematik.uni-stuttgart.de}{\textrm{\Letter}}}1}$}

\medskip

{\footnotesize
 \centerline{$^1$Institute of Applied Analysis and Numerical  Simulation, University of Stuttgart, Germany}
} 

\bigskip

% The name of the handling editor will be entered by AIMS production staff.
% "Communicated by Handling Editor" is not needed for special issue.
% \centerline{(Communicated by Handling Editor)}

%%%%%%%%%%%%%%%%%%%%%%%%%%%%%%%%%%%%%%%%%%%%%%%%%%%%%%%
%             5. ABSTRACT
%%%%%%%%%%%%%%%%%%%%%%%%%%%%%%%%%%%%%%%%%%%%%%%%%%%%%%%

\begin{abstract}
We consider single-phase flow in a fractured porous medium governed by Darcy's law with spatially varying hydraulic conductivity matrices in both bulk and fractures. 
The width-to-length ratio of a fracture is of the order of a small parameter~$\epsilon$ and the ratio~$\refq{K_\mathrm{f}} / \refq{K_\mathrm{b}}$ of the characteristic hydraulic conductivities in the fracture and bulk domains is assumed to scale with~$\epsilon^\alpha$ for a parameter~$\alpha \in \mathbb{R}$. 
The fracture geometry is parameterized by aperture functions on a submanifold of codimension one.
Given a fracture, we derive the limit models as~$\epsilon \rightarrow 0$.
Depending on the value of~$\alpha$, we obtain five different limit models as $\epsilon \rightarrow 0$, for which we present rigorous convergence results. 
\end{abstract}

%%%%%%%%%%%%%%%%%%%%%%%%%%%%%%%%%%%%%%%%%%%%%%%%%%%%%%
%                   6. BODY
%%%%%%%%%%%%%%%%%%%%%%%%%%%%%%%%%%%%%%%%%%%%%%%%%%%%%%

\section{Introduction}
Porous media with fractures or other thin heterogeneities, such as membranes, occur in a wide range of applications in nature and industry including carbon sequestration, groundwater flow, geothermal engineering, oil recovery, and biomedicine.
Fractures are characterized by an extreme geometry with a small aperture but a significantly larger longitudinal extent, typically by several orders of magnitude.
Therefore, it is often computationally unfeasible to represent fractures explicitly in full-dimensional numerical methods, especially in the case of fracture networks, as this results in thin equi-dimensional domains that require a high resolution.
However, the presence of fractures can have a crucial impact on the flow profile in a porous medium with the fractures acting either as major conduits or as barriers. 
Moreover, in order to obtain accurate predictions for the flow profile, generally, one also has to take into account the geometry of fractures, i.e., curvature and spatially varying aperture~\cite{burbulla23,wang23}.

In the following paragraph, we provide a brief overview on modeling approaches for flow in fractured porous media with a focus on discrete fracture models. 
For details on modeling and discretization strategies, we refer to the review article~\cite{berre19} and the references therein.
Conceptually, one can distinguish between models with an
explicit representation of fractures and models that represent fractures implicitly by an effective continuum. 
For the latter category, there is a distinction between equivalent porous medium models~\cite{liu16,oda85}, where fractures are modeled by modifying the permeability of the underlying porous medium, and multi-continuum models~\cite{arbogast90,barenblatt60}, where the fractured porous medium is represented by multiple superimposed interacting continua---in the simplest case by a fracture continuum and a matrix continuum.
In contrast, discrete fracture models represent fractures explicitly as interfaces of codimension one within a porous medium.
In comparison with implicit models, there is an increase in geometrical complexity but no upscaled description based on effective quantities.
Besides, there are also hybrid approaches for fracture networks, where only dominant fractures are represented explicitly~\cite{chen22,karvounis16}.
The most popular method for the derivation of discrete fracture models is vertical averaging~\cite{ahmed17,brenner18,bukac17,burbulla22,lesinigo11,martin05,rybak20,starnoni21}, 
where the governing equations inside the fracture are integrated in normal direction to obtain an interfacial description based on averaged quantities. Using this approach, the resulting mixed-dimensional model is typically closed by making formal assumptions on the flow profile inside the fracture. 
Most commonly, averaged discrete fracture models are based on the conception of a planar fracture geometry with constant aperture. 
However, there are also works that consider curved fractures and fractures with spatially varying aperture~\cite{burbulla23,paranamana21}.
Moreover, there are papers that take a mathematically more rigorous approach for the derivation of discrete fracture models by applying weak compactness arguments to prove (weak) convergence towards a mixed-dimensional model in the limit of vanishing aperture~\cite{armiti22,huy74,list20,morales10,morales17,morales12,sanchez74}. 
This is also the approach that we follow here. 
Further, we mention~\cite{kumar20,melnyk24}, where asymptotic expansions are employed to obtain limit models for vanishing aperture.
Besides, error estimates for discrete fracture models are obtained in~\cite{brezina16,gander21}.
In particular, in~\cite{gander21}, an asymptotic expansion based on a Fourier transform is used to obtain the reduced model.
Further, the authors in~\cite{boon21} have developed a mixed-dimensional functional analysis, which is utilized in~\cite{boon23} to obtain a poromechanical discrete fracture model using a \enquote{top-down} approach.
In addition, we also mention phase-field models~\cite{mikelic15}, which are convenient to track the propagation of fractures and can be combined with discrete fracture models~\cite{burbulla23b}.

In this paper, we consider single-phase fluid flow in a porous medium with an isolated fracture. 
Here, the term fracture refers to a thin heterogeneity inside the bulk porous medium which may itself be described as another porous medium with a distinctly different permeability, e.g., a debris- or sediment-filled crack inside a porous rock. 
We assume that the flow is governed by Darcy's law in both bulk and fracture. 
Further, we introduce the characteristic width-to-length ratio~$\epsilon > 0$ of the fracture as a scaling parameter.
Given that the ratio~$\refq{K_\rmf} / \refq{K_\mathrm{b}} $ of characteristic hydraulic conductivities in the fracture and bulk domain scales with~$\epsilon^\alpha$, we obtain five different limit models as $\epsilon \rightarrow 0$ depending on the value of the parameter~$\alpha \in \mathbb{R}$.
As the mathematical structure of the limit models is different in each case and reaches from a simple boundary condition to a PDE on the interfacial limit fracture, the different cases require different analytical approaches.
Aside from delicate weak compactness arguments, the convergence proofs rely on tailored parameterizations and a  novel coordinate transformation with controllable behavior with respect to the scaling parameter~$\epsilon$. 
Besides, we show the wellposedness of the limit models and strong convergence.

For simple geometries and constant hydraulic conductivities, the limit of vanishing aperture~$\epsilon \rightarrow 0$ is discussed for similar systems in \cite{list20} for the case $\alpha < 1$, in \cite{huy74,morales10} for the case $\alpha = -1$, and in 
\cite{sanchez74} for the case $\alpha = 1$.
Our approach is related to the approach in~\cite{list20}, where Richards equation is considered. 
However, while their focus is on dealing with the nonlinearity and time-dependency of unsaturated flow, our focus is on the derivation of limit models for general fracture geometries and spatially varying tensor-valued hydraulic conductivities for the the whole range of parameters~$\alpha \in \mathbb{R}$.
This aspect is not considered in~\cite{list20}.

The structure of this paper is as follows.  
In \Cref{sec:sec2}, we define the full-dimensional model problem of Darcy flow in a porous medium with an isolated fracture and introduce the characteristic width-to-length ratio~$\epsilon$ of the fracture as a scaling parameter.
\Cref{sec:sec3} deals with the derivation of a-priori estimates for the family of full-dimensional solutions parameterized by~$\epsilon > 0$.
Further, in \Cref{sec:sec4}, depending on the choice of parameters, we identify the limit models as $\epsilon \rightarrow 0$ and provide rigorous proofs of convergence.
A short summary of the geometric background is given in Appendix~\ref{sec:secA}.

\section{Full-Dimensional Model and Geometry} \label{sec:sec2}
First, in \Cref{sec:sec21}, we define the geometric setting and introduce the full-dimensional model problem of single-phase Darcy flow in a porous medium with an isolated fracture in dimensional form. 
Then, in \Cref{sec:sec22}, dimensional quantities are rescaled by characteristic reference quantities to obtain a non-dimensional problem. 
\Cref{sec:sec23} discusses the dependence of the domains and parameters on the width-to-length ratio~$\epsilon$ of the fracture, which is introduced as a scaling parameter. 
Further, given an atlas for the surface that represents the fracture in the limit~$\epsilon \rightarrow 0$, \Cref{sec:sec24} introduces suitable local parameterizations for the bulk and fracture domains, which, in \Cref{sec:sec25}, allow us to transform the weak formulation of the non-dimensional problem from \Cref{sec:sec22} into a problem with $\epsilon$-independent domains.  

\subsection{Full-Dimensional Model in Dimensional Form} \label{sec:sec21}
In the following, dimensional quantities are denoted with a tilde to distinguish them from the non-dimensional quantities that are introduced in \Cref{sec:sec22}. 
Constant reference quantities are marked by a star. 

Let $n\in\mathbb{N}$ with $n\ge 2$ and let $\dimq{G} \subset \mathbb{R}^n$ be a bounded domain with  $\partial \dimq{G} \in \mathcal{C}^2$. 
We write $\vct{N} \in \mathcal{C}^1 (\partial \dimq{G} ; \mathbb{R}^n )$ for the outer unit normal field on~$\partial\dimq{G}$.
Besides, let $\emptyset \not= \overline{\dimq{\gamma }} \subset \partial \dimq{G} $ be a compact and connected $\smash{\mathcal{C}^{0,1}}$-submanifold with boundary~$\smash{\partial \overline{\dimq{\gamma}}}$ and dimension~$n-1$.
The interior of~$\smash{\overline{\dimq{\gamma}}}$ is denoted by $\smash{\dimq{\gamma}}$.
Subsequently, we will consider the limit of vanishing width-to-length ratio for an isolated fracture in a porous medium such that $\dimq{\gamma}$ represents the interfacial fracture in the limit model. 
The domain~$\dimq{G}$ plays a purely technical role: It induces an orientation on~$\dimq{\gamma}$; in particular, its boundary~$\partial \dimq{G}$ is endowed with a signed distance function~$\smash{d_\leftrightarrow^{\partial \dimq{G}}}$.
Moreover, $\smash{\partial\dimq{G}}$ has bounded curvature.
Thus, there exists a neighborhood of~$\smash{\partial \dimq{G}}$ where the orthogonal projection~$\smash{\mathcal{P}^{\partial \dimq{G}}}$ and the signed distance function~$\smash{d_\leftrightarrow^{\partial \dimq{G}}}$ are well-defined and differentiable.
We refer to Appendix~\ref{sec:secA1} for the relevant geometric background.

In the following, we define the geometry of the full-dimensional model.
Given aperture functions~$\smash{\dimq{a}_i \in \mathcal{C}^{0,1} \bigl(\overline{\dimq{\gamma } }\bigr)}$ for $i \in \{ + , - \}$ such that the total aperture $\dimq{a} := \dimq{a}_+ + \dimq{a}_- \ge 0 $ is non-negative, we define the fracture domain~$\smash{\dimq{\Omega}_\mathrm{f}}$ and its boundary segments~$\smash{\dimq{\gamma}_\pm}$ by 
%\begin{linenomath}
\begin{subequations}
\begin{align}
\dimq{\Omega}_\mathrm{f} &:= \bigl\{ \dimq{\vct{p}} + \dimq{s} \vct{N}  (\dimq{\vct{p}} ) \in \mathbb{R}^n \  \big\vert\ \dimq{\vct{p}} \in \dimq{\gamma} , \, - \dimq{a}_-( \dimq{\vct{p} } ) < \dimq{s }   <  \dimq{a}_+(\dimq{\vct{p} } ) \bigr\} ,  \\
\dimq{\gamma}_\pm &:= \bigl\{ \dimq{\vct{p}} \pm \dimq{a}_\pm (\dimq{\vct{p}}) \vct{N}  (\dimq{\vct{p}} )  \in \mathbb{R}^n \  \big\vert\ \dimq{\vct{p}}  \in \dimq{\gamma} \bigr\} .
\end{align}
\end{subequations}
%\end{linenomath}
Here and subsequently, we use the index~$\pm$ as an abbreviation to simultaneously refer to two different quantities or domains on the inside~($-$) and outside~($+$) of the domain~$\smash{\dimq{G}}$. 
Further, we distinguish between the parts of the fracture interface~$\smash{\dimq{\gamma}}$ and the boundary segments~$\smash{\dimq{\gamma}_\pm}$ with non-zero and zero aperture~$\smash{\dimq{a}}$, i.e., $\smash{\dimq{\gamma}  = \dimq{\Gamma } \; \dot{\cup} \; \smash{\dimq{\Gamma}}_0^0}$ and $\smash{\dimq{\gamma}_\pm = \dimq{\Gamma}_\pm \; \dot{\cup} \; \dimq{\Gamma}_0}$, where
%\begin{linenomath}
\begin{alignat*}{2}
\dimq{\Gamma} &:= \bigl\{ \dimq{\vct{p}} \in \dimq{\gamma} \ \big\vert \ \dimq{a} (\dimq{\vct{p}}) > 0 \bigr\}  , \qquad 
&\smash{\dimq{\Gamma}}_0^0 &:= \dimq{\gamma} \setminus \dimq{\Gamma } ,  \\
\dimq{\Gamma}_0 &:= \dimq{\gamma}_+  \cap \dimq{\gamma}_- , 
&\dimq{\Gamma}_\pm &:= \dimq{\gamma}_\pm \setminus \dimq{\Gamma}_0 .
\end{alignat*}
%\end{linenomath}
We assume that $\smash{\dimq{\Omega}_\mathrm{f}}$ is connected with $\smash{\lambda_n (\dimq{\Omega}_\mathrm{f}) > 0}$, where $\lambda_n$ denotes the $n$-di\-men\-sio\-nal Lebesgue measure. 
In addition, we assume that the aperture functions~$\dimq{a}_\pm$ are sufficiently small such that $\smash{\dimq{\Omega}_\mathrm{f} \subset \operatorname{unpp}( \partial \dimq{G} )}$ with $\smash{\operatorname{unpp}( \partial \dimq{G} ) \subset \mathbb{R}^n} $ as defined in \Cref{def:unpp}.
Besides, we denote by~$\smash{ \dimq{\Omega}_\pm \subset \mathbb{R}^n}$  two disjoint and bounded Lipschitz domains such that $\smash{\dimq{\Omega}_\pm \cap \dimq{\Omega}_\mathrm{f} = \emptyset}$ and $\smash{\partial \dimq{\Omega}_\pm \cap \partial \dimq{\Omega}_\mathrm{f} = \overline{\dimq{\gamma}}_\pm}$.
$\smash{\dimq{\Omega}_+}$ and $\smash{\dimq{\Omega}_-}$ are bulk domains adjacent to the fracture domain~$\smash{\dimq{\Omega}_\mathrm{f}}$. 
Further, we define the total domain 
%\begin{linenomath}
\begin{align}
\dimq{\Omega} &:=   \dimq{\Omega}_+ \cup \dimq{\Omega}_- \cup \dimq{\Omega}_\mathrm{f} \cup \dimq{\gamma}_+ \cup \dimq{\gamma}_-  ,
\end{align}
%\end{linenomath}
which we assume to be a Lipschitz domain.
Moreover, we write 
%\begin{linenomath}
\begin{subequations}
\begin{align}
\dimq{\rho}_\pm &:= \partial \dimq{\Omega}_\pm \setminus \dimq{\gamma}_\pm  = \dimq{\rho}_{\pm , \mathrm{D}}\; \dot{\cup} \; \dimq{\rho}_{\pm,\mathrm{N}} , \\
\dimq{\rho}_\mathrm{f} &:= \partial \dimq{\Omega}  \setminus \bigl(  \dimq{\rho}_+ \cup  \dimq{\rho}_-  \bigr) = \dimq{\rho}_{\mathrm{f}, \mathrm{D} } \; \dot{\cup} \; \dimq{\rho}_{\mathrm{f}, \mathrm{N}} \subset \partial \dimq{\Omega}_\rmf
\end{align}
\end{subequations}
%\end{linenomath}
for the external boundaries of the bulk domains~$\dimq{\Omega}_i \subset \dimq{\Omega}$, $i \in \{ +,-,\mathrm{f}\}$, which are composed of disjoint Dirichlet and Neumann segments~$\smash{\dimq{\rho}_{i,\mathrm{D}}}$ and $\smash{\dimq{\rho}_{i,\mathrm{N}}}$.
The resulting geometric configuration is sketched in \Cref{fig:dimqOmega}.
\begin{figure}
\centering
\includegraphics[scale=1]{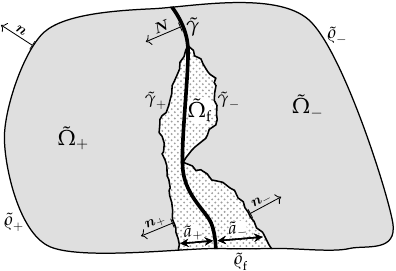}
\caption{Sketch of the geometry in the full-dimensional model~\eqref{eq:darcydim} in dimensional form. }
\label{fig:dimqOmega}
\end{figure}

Now, let~$\smash{\dimq{\matr{K}}_\pm \in L^\infty (\dimq{\Omega}_\pm ; \mathbb{R}^{n\times n} )}$ and~$\smash{\dimq{\matr{K}}_\mathrm{f} \in L^\infty ( \dimq{\Omega}_\mathrm{f} ; \mathbb{R}^{n\times n})}$ be symmetric and uniformly elliptic hydraulic conductivity matrices. 
Further, for $i \in \{ + , -, \mathrm{f} \}$, let $\smash{\dimq{p}_i}$ denote the pressure head in~$\smash{\dimq{\Omega}_i}$.
Then, given the source terms~$\smash{\dimq{q}_\pm \in L^2 (\dimq{\Omega}_\pm )}$ and~$\smash{\dimq{q}_\mathrm{f} \in L^2 (\dimq{\Omega}_\mathrm{f})}$, we consider the following problem of Darcy flow in $\smash{\dimq{\Omega}}$. 

Find $\smash{\dimq{p}_\pm \colon \dimq{\Omega}_\pm \rightarrow \mathbb{R}}$ and $\smash{\dimq{p}_\mathrm{f} \colon \dimq{\Omega}_\mathrm{f} \rightarrow \mathbb{R}}$ such that
%\begin{linenomath}
\begin{subequations}
\begin{alignat}{3}
-\dimq{\nabla} \cdot \bigl( \dimq{\matr{K}}_i \dimq{\nabla} \dimq{p}_i \bigr) &= \dimq{q}_i \qquad &&\text{in } \dimq{\Omega}_i , \quad &&i \in \{ +, - ,\mathrm{f} \} , \\
\dimq{p}_\pm &= \dimq{p}_\mathrm{f} \qquad &&\text{on } \dimq{\Gamma}_\pm ,  \\
\dimq{\matr{K}}_\pm \dimq{\nabla} \dimq{p}_\pm \cdot \vct{n}_\pm &= \dimq{\matr{K}}_\mathrm{f} \dimq{\nabla} \dimq{p}_\mathrm{f} \cdot \vct{n}_\pm \qquad &&\text{on } \dimq{\Gamma}_\pm  , \\
\dimq{p}_+ &= \dimq{p}_- \qquad &&\text{on } \dimq{\Gamma}_0 ,  \\
\dimq{\matr{K}}_+ \dimq{\nabla} \dimq{p}_+ \cdot \vct{n}_+ &= -\dimq{\matr{K}}_- \dimq{\nabla} \dimq{p}_- \cdot \vct{n}_- \qquad &&\text{on } \dimq{\Gamma}_0 ,  \\
\dimq{p}_i &= 0 \qquad &&\text{on } \dimq{\rho}_{i,\mathrm{D}} , \quad &&i\in \{ +,-, \mathrm{f}\} , \\
\dimq{\matr{K}}_i \dimq{\nabla} \dimq{p}_i \cdot \vct{n} &= 0 \qquad &&\text{on } \dimq{\rho}_{i,\mathrm{N}} ,  \quad &&i \in \{ + , - , \mathrm{f} \}.
\end{alignat}%
\label{eq:darcydim}%
\end{subequations}%
%\end{linenomath}
Here, $\vct{n}$ is the outer unit normal on $\partial\dimq{\Omega}$ and $\vct{n}_\pm$ denotes the unit normal on~$\dimq{\gamma}_\pm$ pointing into~$\smash{\dimq{\Omega}_\pm}$. 
We remark that the choice of homogeneous boundary conditions in \cref{eq:darcydim} is only made for the sake of simplicity.
The extension to the inhomogeneous case is straightforward. 

\subsection{Full-Dimensional Model in Non-Dimensional Form} \label{sec:sec22}
We write $\smash{\refq{L}}$~$\unit{[m]}$ and $\smash{\refq{a}}$~$\unit{[m]}$ for the characteristic values of the length and aperture of the fracture given by
%\begin{linenomath}
\begin{align}
\refq{L} &:= {\lambda_{\partial \dimq{G}} ( \dimq{\Gamma} ) }^\frac{1}{n-1} \quad\enspace\text{and}\quad\enspace \refq{a} := \frac{1}{\lambda_{\partial \dimq{G}} ( \dimq{\Gamma} )} \int_{\dimq{\Gamma }}  \dimq{a} \,\D \lambda_{\partial\dimq{G}} .
\end{align}
%\end{linenomath}
Then, we define $\epsilon := {\refq{a}} / {\refq{L}} > 0$ as the characteristic width-to-length ratio of the fracture.  
Subsequently, in \Cref{sec:sec3,sec:sec4}, we will treat $\epsilon$ as scaling parameter and analyze the limit behavior as~$\epsilon \rightarrow 0$.

Next, let $\refq{K}_\mathrm{b}$~$\unit{[m/s]}$ and $\refq{K}_\mathrm{f}$~$\unit{[m/s]}$ be characteristic values of the hydraulic conductivities~$\dimq{\matr{K}}_\pm$ and~$\dimq{\matr{K}}_\mathrm{f}$ in the bulk and fracture. 
In addition, we define the non-dimensional position vector~$\smash{\vct{x} := \dimq{\vct{x}} / \refq{L}}$.
The non-dimensionalization of the position vector~$\vct{x}$ results in a rescaling of spatial derivative operators, e.g., $\nabla = \refq{L} \dimq{\nabla}$.
Besides, it necessitates the definition of non-dimensional domains, which will be denoted without tilde, e.g., $\smash{\Omega := \dimq{\Omega} / \refq{L}}$. 
Moreover, we define
%\begin{linenomath}
\begin{align}
 \rho_{\mathrm{b}, \mathrm{D}}^\epsilon &:= \rho_{+, \mathrm{D}}^\epsilon \cup \rho_{-, \mathrm{D}}^\epsilon , \quad\enspace
\rho_\mathrm{D}^\epsilon := \rho_{+, \mathrm{D}}^\epsilon \cup \rho_{-, \mathrm{D}}^\epsilon \cup \rho_{\mathrm{f}, \mathrm{D}}^\epsilon .
\end{align}
%\end{linenomath}
We require $\smash{\lambda_{\partial\Omega } ( \rho_{\mathrm{b}, \mathrm{D}}^\epsilon ) > 0  }$. 
Besides, we sometimes require the stronger assumption 
%\begin{linenomath}
\begin{align}
\tag{$\smash{\mathbb{A}}$} \hspace{1.4cm} \lambda_{\partial\Omega } (\rho^\epsilon_{+, \mathrm{D} } ) > 0 \quad\text{and}\quad \lambda_{\partial\Omega } (\rho^\epsilon_{-, \mathrm{D} } ) > 0 , \label{asm:rhoD}
\end{align}
%\end{linenomath}
which is subsequently referred to as \enquote{assumption~\eqref{asm:rhoD}}.
Further, we define the non-dimensional quantities
%\begin{linenomath}
\begin{equation} 
\label{eq:nondim2}
\begin{alignedat}{5}
p_\pm^\epsilon &:= \frac{\dimq{p}_\pm }{ \refq{p}_\mathrm{b} } , 
\quad\enspace &\matr{K}_\pm^\epsilon &:= \frac{ \dimq{\matr{K}}_\pm } {\refq{K}_\mathrm{b}}, 
\quad\enspace &q_\pm^\epsilon &:= \frac{ \dimq{q}_\pm }{\refq{q}_\mathrm{b}},
\quad\enspace &a_\pm &:= \frac{\dimq{a}_\pm}{\refq{a}} , 
\quad\enspace &a &:= \frac{\dimq{a}}{\refq{a}} , \\
p_\mathrm{f}^\epsilon &:= \frac{\dimq{p}_\mathrm{f} }{ \refq{p}_\mathrm{f} } , 
\quad\enspace &\matr{K}_\mathrm{f}^\epsilon &:= \frac{ \dimq{\matr{K}}_\mathrm{f} } {\refq{K}_\mathrm{f}},
\quad\enspace &q_\mathrm{f}^\epsilon &:= \frac{ \dimq{q}_\mathrm{f} }{\refq{q}_\mathrm{f}},
\end{alignedat}
\end{equation}
%\end{linenomath}
where $\refq{p}_\mathrm{b} := \refq{L}$, $\refq{p}_\mathrm{f} := \refq{L}$, and $\refq{q}_\mathrm{b} := \refq{K}_\mathrm{b} / \refq{L}$. 
We assume that there exist parameters~$\alpha \in \mathbb{R}$ and $\beta \ge -1$ such that the characteristic fracture quantities~$\refq{K}_\mathrm{f}$ and~$\refq{q}_\mathrm{f}$ scale like
%\begin{linenomath}
\begin{align}
\label{eq:epsscaling} \refq{K}_\mathrm{f} &= \epsilon^\alpha \refq{K}_\mathrm{b} \quad\enspace\text{and}\quad\enspace \refq{q}_\mathrm{f} = \epsilon^\beta \refq{q}_\mathrm{b} .
\end{align}
%\end{linenomath}

The dimensional Darcy system in \cref{eq:darcydim} now corresponds to the following non-dimensional problem. 

Find $\smash{p_\pm^\epsilon \colon \Omega_\pm^\epsilon \rightarrow \mathbb{R}}$ and $\smash{p_\mathrm{f}^\epsilon \colon \Omega_\mathrm{f}^\epsilon \rightarrow \mathbb{R}}$ such that
%\begin{linenomath}
\begin{subequations}
\begin{alignat}{3}
-\nabla \cdot \bigl( \matr{K}_\pm^\epsilon \nabla p_\pm^\epsilon \bigr) &= q_\pm^\epsilon \qquad &&\text{in } \Omega_\pm^\epsilon , \\
-\nabla \cdot \bigl( \epsilon^\alpha \matr{K}_\mathrm{f}^\epsilon \nabla p_\rmf^\epsilon \bigr) &= \epsilon^\beta q_\mathrm{f}^\epsilon \qquad &&\text{in } \Omega_\mathrm{f}^\epsilon , \\
p_\pm^\epsilon &= p_\mathrm{f}^\epsilon \qquad &&\text{on } \Gamma_\pm^\epsilon , \\
\matr{K}_\pm^\epsilon \nabla p_\pm^\epsilon \cdot \vct{n}_\pm^\epsilon &= \epsilon^\alpha \matr{K}_\mathrm{f}^\epsilon \nabla p_\mathrm{f}^\epsilon \cdot \vct{n}_\pm^\epsilon \qquad &&\text{on } \Gamma_\pm^\epsilon , \\
p_+^\epsilon &= p_-^\epsilon \qquad &&\text{on } \Gamma_0^\epsilon ,  \\
\matr{K}_+^\epsilon \nabla p_+^\epsilon \cdot \vct{n}_+^\epsilon &= -\matr{K}_-^\epsilon \nabla p_-^\epsilon\cdot \vct{n}_-^\epsilon \qquad &&\text{on } \Gamma_0^\epsilon ,  \\
p_i^\epsilon &= 0 \qquad &&\text{on } \rho_{i,\mathrm{D}}^\epsilon , \quad &&i\in \{ +,-, \mathrm{f}\} , \\
\matr{K}_i^\epsilon \nabla p_i^\epsilon \cdot \vct{n} &= 0 \qquad &&\text{on } \rho_{i,\mathrm{N}}^\epsilon ,  \quad &&i \in \{ +, -  , \mathrm{f} \} .
\end{alignat}%
\label{eq:strongdarcyeps}%
\end{subequations}%
%\end{linenomath}
In \cref{eq:strongdarcyeps}, $\smash{\vct{n}}$ is the outer unit normal on~$\smash{\partial\Omega}$ and $\smash{\vct{n}_\pm^\epsilon}$ denotes the unit normal on~$\smash{\gamma_\pm^\epsilon}$ pointing into~$\smash{\Omega^\epsilon_\pm}$.
The geometry of the non-dimensional problem~\eqref{eq:strongdarcyeps} with full-dimensional fracture~$\smash{\Omega_\rmf^\epsilon}$, as well as the limit geometry as~$\epsilon \rightarrow 0$, are sketched in \Cref{fig:nondimOmega}. 
\begin{figure}
\centering
\includegraphics[scale=1.049]{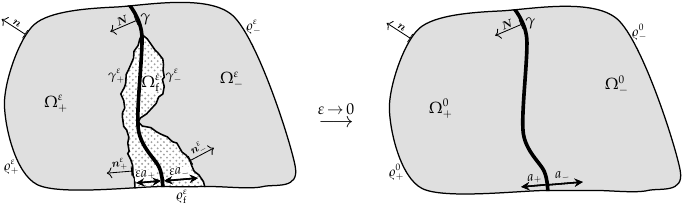}
\caption{Sketch of the geometry in the full-dimensional model~\eqref{eq:strongdarcyeps} in non-dimensional form~(left) and in the limit of vanishing width-to-length ratio~$\epsilon\rightarrow 0$~(right). }
\label{fig:nondimOmega}
\end{figure}

Next, we define the space
%\begin{linenomath}
\begin{equation}
\begin{multlined}[c][0.875\displaywidth]
\Phi^\epsilon := \Bigl\{ ( \phi_+^\epsilon , \phi_-^\epsilon , \phi_\mathrm{f}^\epsilon ) \in \bigtimes\nolimits_{i\in\{ + , - ,\rmf \}} H_{0,\rho_{i,\mathrm{D}}^\epsilon}^1 (\Omega_i^\epsilon )  \ \Big\vert\\
 \phi_-^\epsilon \bigr\vert_{\Gamma_0^\epsilon} = \phi_+^\epsilon \bigr\vert_{\Gamma_0^\epsilon } ,\enspace \phi_\pm^\epsilon \bigr\vert_{\Gamma_\pm^\epsilon} \!  = \phi_\mathrm{f}^\epsilon \bigr\vert_{\Gamma_\pm^\epsilon}  \Bigr\} \cong H^1_{0, \rho^\epsilon_{ \mathrm{D}}  } (\Omega ) .
\end{multlined}
\end{equation}
%\end{linenomath}
Then, a weak formulation of the system in \cref{eq:strongdarcyeps} is given by the following problem.

Find $ \bigl( p_+^\epsilon , p_-^\epsilon , p_\mathrm{f}^\epsilon \bigr) \in \Phi^\epsilon$ such that, for all $\bigl( \phi_+^\epsilon , \phi_-^\epsilon , \phi_\mathrm{f}^\epsilon \bigr) \in \Phi^\epsilon$, 
\begin{equation}
\begin{multlined}[c][0.875\displaywidth]
\sum_{i = \pm } \bigl( \matr{K}_i^\epsilon \nabla p_i^\epsilon , \nabla \phi_i^\epsilon \bigr)_{\vct{L}^2 (\Omega_i^\epsilon  )} + \epsilon^\alpha \bigl(  \matr{K}_\mathrm{f}^\epsilon \nabla p_\mathrm{f}^\epsilon , \nabla \phi_\mathrm{f}^\epsilon \bigr)_{\vct{L}^2 (\Omega_\mathrm{f}^\epsilon  )} \\
= \sum_{i = \pm } \bigl( q_i^\epsilon , \phi_i^\epsilon \bigr)_{L^2 (\Omega_i^\epsilon )} + \epsilon^\beta \bigl(  q_\mathrm{f}^\epsilon , \phi_\mathrm{f}^\epsilon \bigr)_{L^2 (\Omega_\mathrm{f}^\epsilon )} .
\label{eq:weakdarcyeps}%
\end{multlined}
\end{equation}
As a consequence of the Lax-Milgram theorem, the Darcy problem~\eqref{eq:weakdarcyeps} admits a unique solution~$\smash{\bigl( p_+^\epsilon , p_-^\epsilon , p_\mathrm{f}^\epsilon \bigr) \in \Phi^\epsilon }$.

\subsection{Scaling of Domains and Parameters with Respect to~\texorpdfstring{$\epsilon$}{Epsilon}} \label{sec:sec23}
Let $\smash{\kappa_k  \in \mathcal{C}^0 ( \partial G ) } $, $k \in \{ 1, \dots , n-1\}$,
denote the principal curvatures on~$\partial G$ and set 
%\begin{linenomath}
\begin{align}
\kappa_{\max} := \max_{\vct{p} \in \overline{\gamma} } \max_{ k \in \{ 1, \dots , n-1 \}} \abs{\kappa_k (\vct{p})} .
\end{align}
%\end{linenomath}
Then, we have $\smash{\kappa_{\max} < \infty }$ due to the compactness of~$\smash{\overline{\gamma}}$. 
Further, we define
%\begin{linenomath}
\begin{align}
\hathat{\epsilon } &:= \min\left\{ 1, \frac{1}{3 \kappa_{\max} } , \mathrm{reach}(\partial G ) \right\} > 0 , \quad
\hat{\epsilon} :=  \frac{\hathat{\epsilon}}{2} \Bigl[\, \max_{i=\pm }\bigl\{ \norm{a_i}_{L^\infty (\gamma ) } \bigr\} \Bigr]^{-1} > 0 \label{eq:eps}%
\end{align}%
%\end{linenomath}
with $\smash{\operatorname{reach}(\partial G) }$ as defined in \Cref{def:unpp}.
In the following, we require $\epsilon \in (0, \hat{\epsilon}]$, which allows us to use the results from Appendix~\ref{sec:secA1} on the regularity and wellposedness of the orthogonal projection~$\smash{\mathcal{P}^{\partial G}}$ and the signed distance function~$\smash{d_\leftrightarrow^{\partial G}}$. 

The dependence of the non-dimensional domains and quantities on the width-to-length ratio~$\epsilon$ of the fracture is made explicit in the notation. 
For the non-dimensional fracture domain~$\Omega_\mathrm{f}^\epsilon$, the $\epsilon$-dependence is evident.
Specifically, we have 
%\begin{linenomath}
\begin{align}
\Omega_\mathrm{f}^\epsilon &=  \bigl\{ \vct{p} + s \vct{N} (\vct{p}) \in \mathbb{R}^n \ \big\vert\ \vct{p} \in \gamma , \ -\epsilon a_-(\vct{p}  ) < s < \epsilon a_+ (\vct{p}  ) \bigr\} .
\end{align}
%\end{linenomath}
Accordingly, the hydraulic conductivity~$\smash{\matr{K}_\mathrm{f}^\epsilon}$ and the source term~$\smash{q_\mathrm{f}^\epsilon}$ scale like 
%\begin{linenomath}
\begin{align}
\matr{K}_\mathrm{f}^\epsilon (\vct{x} ) &=  \matr{K}_\mathrm{f}^{\hat{\epsilon}} \bigl( \vct{T}_\mathrm{f}^\epsilon (\vct{x} )  \bigr) , \qquad
q_\mathrm{f}^\epsilon ( \vct{x} ) =  q_\mathrm{f}^{\hat{\epsilon}} \bigl( \vct{T}_\mathrm{f}^\epsilon (\vct{x}) \bigr) ,
\end{align}
%\end{linenomath}
where the transformation $\smash{\vct{T}_\mathrm{f}^\epsilon \colon \Omega_\mathrm{f}^\epsilon \rightarrow \Omega_\mathrm{f}^{\hat{\epsilon}} }$ is given by
%\begin{linenomath}
\begin{align}
\vct{T}_\mathrm{f}^\epsilon (\vct{x} ) = \mathcal{P}^{\partial G} (\vct{x}) -\frac{\hat{\epsilon}}{\epsilon}  d_\leftrightarrow^{\partial G} (\vct{x} ) \vct{N} \bigl( \mathcal{P}^{\partial G} (\vct{x} ) \bigr) .
\end{align}
%\end{linenomath}
Further, we define 
%\begin{linenomath}
\begin{subequations}
\begin{align}
\Omega_{\pm , \mathrm{in}}^\epsilon &:= \Omega^\epsilon_\pm \cap \left\{ \vct{p} \pm s \vct{N}(\vct{p})  \ \middle\vert \ \vct{p}\in \gamma ,\enspace \epsilon a_\pm (\vct{p})  < s < \hathat{\epsilon } \right\}  , \\
\Omega_{\pm , \mathrm{out}} &:= \Omega_\pm^\epsilon \setminus \Omega_{\pm , \mathrm{in}}^\epsilon .
\end{align}
\end{subequations}
%\end{linenomath}
Note that only the inner region~$\smash{\Omega_{\pm , \mathrm{in}}^\epsilon}$ of the bulk domain~$\smash{\Omega_\pm^\epsilon}$ depends on the scaling parameter~$\epsilon$, while the outer region~$\smash{\Omega_{\pm ,\mathrm{out}}}$ does not. 
For the inner region~$\smash{\Omega_{\pm , \mathrm{in}}^\epsilon}$, we impose a linear deformation in normal direction with decreasing~$\epsilon$, i.e., the hydraulic conductivity~$\smash{\matr{K}_\pm^\epsilon}$ and the source term~$\smash{q_\pm^\epsilon}$ satisfy
%\begin{linenomath}
\begin{align}
\matr{K}_\pm^\epsilon (\vct{x} ) &= \matr{K}^0_\pm \bigl( \vct{T}^\epsilon_\pm (\vct{x} ) \bigr) , \qquad 
q_\pm^\epsilon (\vct{x} ) = q^0_\pm \bigl( \vct{T}^\epsilon_\pm (\vct{x} ) \bigr) 
\end{align}
%\end{linenomath}
for $\smash{\vct{x} \in \Omega_{\pm , \mathrm{in}}^\epsilon }$, where the transformation $\smash{ \vct{T}^\epsilon_\pm \colon \Omega_{\pm ,\mathrm{in}}^\epsilon \rightarrow \Omega_{\pm, \mathrm{in}}^0 }$ is given by
%\begin{linenomath}
\begin{subequations}
\begin{align}
\vct{T}^\epsilon_\pm (\vct{x} ) &:= \mathcal{P}^{\partial G} (\vct{x} ) + t^\epsilon_\pm \bigl(  \mathcal{P}^{\partial G} (\vct{x}) , -  d_\leftrightarrow^{\partial G} (\vct{x} )  \bigr) \vct{N} \bigl( \mathcal{P}^{\partial G} (\vct{x} ) \bigr) ,\\
t^\epsilon_\pm (\vct{p}, \lambda ) &:= \frac{\hathat{\epsilon}}{\hathat{\epsilon} - \epsilon a_\pm (\vct{p})} \bigl[ \lambda \mp \epsilon a_\pm (\vct{p}) \bigr] .
\end{align}
\end{subequations}
%\end{linenomath}
It is now easy to see that the following lemma holds.
\begin{lemma}
Let $\epsilon \in ( 0 , \hat{\epsilon} ]$. 
Then, $\smash{\vct{T}_\mathrm{f}^\epsilon \colon \Omega_\rmf^\epsilon \rightarrow \Omega_\rmf^{\hat{\epsilon}}}$ is a $\smash{\mathcal{C}^1}$-diffeomorphism. Besides, $\smash{\vct{T}_\pm^\epsilon \colon \Omega_{\pm ,\mathrm{in}}^\epsilon \rightarrow \Omega_{\pm, \mathrm{in}}^0 }$ is bi-Lipschitz.
The inverses
$\smash{\invT_\rmf^\epsilon := ( \vct{T}^\epsilon_\rmf )^{-1}}$ and $\smash{\invT_\pm^\epsilon := ( \vct{T}_\pm^\epsilon )^{-1}}$ 
are given by
%\begin{linenomath}
\begin{align}
 \invT^\epsilon_\mathrm{f}(\vct{x} ) &=  \mathcal{P}^{\partial G} (\vct{x}) -\frac{\epsilon}{\hat{\epsilon}}  d_\leftrightarrow^{\partial G} (\vct{x} ) \vct{N} \bigl( \mathcal{P}^{\partial G} (\vct{x} ) \bigr) ,
\end{align} 
%\end{linenomath}
\vspace*{-1.5em}
%\begin{linenomath}
\begin{subequations}
\begin{align}
\invT^\epsilon_\pm (\vct{x} ) &=  \mathcal{P}^{\partial G} (\vct{x} ) + \invt^\epsilon_\pm \bigl(  \mathcal{P}^{\partial G} (\vct{x}) ,  -  d_\leftrightarrow^{\partial G} (\vct{x} )  \bigr) \vct{N} \bigl( \mathcal{P}^{\partial G} (\vct{x} ) \bigr) ,\\
\invt^\epsilon_\pm (\vct{p} , \lambda ) &= \frac{\hathat{\epsilon} - \epsilon a_\pm (\vct{p} )}{\hathat{\epsilon} }  \lambda \pm \epsilon a_\pm (\vct{p})  .
\end{align}
\end{subequations}
%\end{linenomath}
\end{lemma}

\subsection{Local Parameterization} \label{sec:sec24}
Subsequently, we use the definitions and notations from Appendix~\ref{sec:secA3}.
We observe that $\Gamma \subset \partial G $ is open so that $\Gamma \subset \mathbb{R}^n$ is itself a $\mathcal{C}^2$-submanifold of dimension~$n-1$. 
Besides, $\smash{\overline{\Gamma} \subset \mathbb{R}^n}$ is a $\smash{\mathcal{C}^{0,1}}$-submanifold with boundary. 
Now, let $\smash{\{ ( U_j , \vct{\psi}_j , V_j ) \}_{j\in{ J }}}$ be a $\mathcal{C}^2$-atlas for~$\Gamma$ consisting of charts~$\vct{\psi}_j \colon  U_j \rightarrow V_j  $, where $U_j \subset \Gamma $ and $V_j \subset \mathbb{R}^{n-1}$ are open. 
Then, for $j \in  J $ and $\epsilon \in ( 0, \hat{\epsilon} ]$, we write~$\smash{\invpsi_j := \vct{\psi}_j^{-1}}$ for the inverse charts and define
%\begin{linenomath}
\begin{subequations}
\begin{align}
U_{\mathrm{f},j}^{\epsilon} &:= \bigl\{ \vct{p} + s \vct{N}(\vct{p} ) \ \big\vert\ \vct{p} \in U_j , -\epsilon a_-(\vct{p}) < s < \epsilon a_+ (\vct{p} ) \bigr\} , \\
U_{\pm ,j}^{\epsilon} &:= \bigl\{ \vct{p} \pm s \vct{N}(\vct{p} ) \ \big\vert\ \vct{p} \in U_j , \ \epsilon a_\pm (\vct{p} ) < s < \hathat{\epsilon}  \bigr\} , \\
 \StepSubequations V_{\mathrm{f},j} &:= \bigl\{ ( \vct{\theta}^\prime , \theta_n ) \in \mathbb{R}^n \ \vert \ \vct{\theta}^\prime \in V_j , \ - a_- \bigl( \invpsi_j (\vct{\theta}^\prime )\bigr) < \theta_n < a_+ \bigl( \invpsi_j (\vct{\theta}^\prime )\bigr)  \bigr\} , \\
 V_{\pm , j}&:= \bigl\{ ( \vct{\theta}^\prime , \pm \theta_n ) \in \mathbb{R}^n \ \vert \ \vct{\theta}^\prime \in V_j , \ 0 < \theta_n < \hathat{\epsilon} \bigr\}  .
\end{align}
\end{subequations}
%\end{linenomath}
In the following, we will also think of the subdomains~$\smash{\Omega_\mathrm{f}^\epsilon} , \ \smash{\Omega_{\pm , \mathrm{in}}^\epsilon} \subset \mathbb{R}^n$ as $n$-dimensional $\mathcal{C}^{0,1}$-submanifolds. 
With the given atlas for~$\Gamma$, we can construct $\mathcal{C}^{0,1}$-atlases
$\smash{\{ ( U_{\mathrm{f},j}^\epsilon , \vct{\psi}_{\mathrm{f},j}^\epsilon , V_{\mathrm{f} , j } ) \}_{j\in J }}$ for~$\smash{\Omega_\mathrm{f}^\epsilon}$ and $\smash{\{( U_{\pm,j}^\epsilon , \vct{\psi}_{\pm,j}^\epsilon , V_{\pm , j } ) \}_{j\in J }}$ for~$\smash{\Omega_{\pm ,\mathrm{in}}^\epsilon}$.
For $j \in  J $, the charts~$\smash{\vct{\psi}_{\mathrm{f},j}^\epsilon }$ and $\smash{\vct{\psi}_{\pm,j}^\epsilon }$, as well as their inverses~$\smash{\invpsi_{\rmf , j}^\epsilon}$ and~$\smash{\invpsi_{\pm , j}^\epsilon }$, are given by 
%\begin{linenomath}
\begin{subequations}
\begin{alignat}{2}
\vct{\psi}_{\mathrm{f},j}^\epsilon &\colon U_{\mathrm{f},j}^\epsilon \rightarrow V_{\mathrm{f},j} , \quad &&\vct{x} \mapsto \bigl( \vct{\psi}_j \bigl(\mathcal{P}^{\partial G } (\vct{x})\bigr) , \ - \epsilon^{-1} d_\leftrightarrow^{\partial G} (\vct{x}) \bigr) , \\
\invpsi_{\mathrm{f},j }^\epsilon  &\colon V_{\mathrm{f},j} \rightarrow U_{\mathrm{f},j}^\epsilon , \quad &&(\vct{\theta}^\prime , \theta_n ) \mapsto \invpsi_j (\vct{\theta}^\prime ) + \epsilon \theta_n \vct{N} \bigl( \invpsi_j (\vct{\theta}^\prime) \bigr) , \\
\StepSubequations \vct{\psi}_{\pm,j}^\epsilon &\colon U_{\pm,j}^\epsilon \rightarrow V_{\pm,j} , \quad &&\vct{x} \mapsto \bigl( \vct{\psi}_j \bigl( \mathcal{P}^{\partial G } (\vct{x})\bigr) , \ t_\pm^\epsilon \bigl( \mathcal{P}^{\partial G} (\vct{x} ),  - d_\leftrightarrow^{\partial G} (\vct{x})  \bigr)   \bigr) , \\
\invpsi_{\pm ,j }^\epsilon &\colon V_{\pm,j} \rightarrow U_{\pm,j}^\epsilon , \quad &&(\vct{\theta}^\prime , \theta_n ) \mapsto \invpsi_j (\vct{\theta}^\prime ) +  \invt_\pm^\epsilon \bigl(  \invpsi_j (\vct{\theta}^\prime ) , \theta_n \bigr) \vct{N} \bigl( \invpsi_j (\vct{\theta}^\prime) \bigr) .
\end{alignat}
\end{subequations}
%\end{linenomath}
Further, we introduce the product-like $n$-dimensional $\mathcal{C}^2$-submanifold
%\begin{linenomath}
\begin{align}
\Gamma_a := \bigl\{ (\vct{p} , \theta_n ) \ \vert \ \vct{p} \in \Gamma , \ -a_-(\vct{p} ) < \theta_n < a_+(\vct{p} ) \bigr\} \subset \mathbb{R}^{n} \times \mathbb{R}. 
\end{align}
%\end{linenomath}
Then, $\smash{\Gamma_a }$ is the interior of the following $\smash{\mathcal{C}^{0,1}}$-manifolds with boundary.
%\begin{linenomath}
\begin{subequations}
\begin{align}
\overline{\Gamma}_a^\smallperp &:= \bigl\{ (\vct{p} , \theta_n ) \ \vert \ \vct{p} \in \Gamma , \ -a_-(\vct{p} ) \le \theta_n \le a_+(\vct{p} ) \bigr\} \subset \mathbb{R}^{n} \times \mathbb{R}, \\
\overline{\Gamma}_a^\smallpar &:= \bigl\{ (\vct{p} , \theta_n ) \ \vert \ \vct{p} \in \overline{\Gamma} , \ -a_-(\vct{p} ) < \theta_n < a_+(\vct{p} ) \bigr\} \subset \mathbb{R}^{n} \times \mathbb{R} .
\end{align}
\end{subequations}
%\end{linenomath}
Besides, we write
%\begin{linenomath}
\begin{align}
\rho_{a ,\mathrm{D} } &:= \bigl\{ (\vct{p} , \theta_n ) \in \overline{\Gamma}_a^\smallpar \ \big\vert \ \vct{p} + \hat{\epsilon} \theta_n \vct{N} ( \vct{p} ) \in \rho_{\rmf , \mathrm{D}}^{\hat{\epsilon}} \bigr\} \subset \partial \overline{\Gamma}_a^\smallpar 
\end{align}
%\end{linenomath}
for the external boundary segment of $\smash{\overline{\Gamma}_a^\smallpar}$ with Dirichlet conditions.
A $\mathcal{C}^2$-atlas of~$\Gamma_a$ is given by $\smash{\{ ( U_j^a , \vct{\psi}_j^a , V_{\rmf, j } ) \}_{j\in J}}$, where 
%\begin{linenomath}
\begin{subequations}
\begin{align}
U_j^a &:= \bigl\{ (\vct{p} , \theta_n ) \ \big\vert\ \vct{p} \in U_j , -a_- (\vct{p} ) < \theta_n < a_+(\vct{p} ) \bigr\} , \\
\vct{\psi}_j^a &\colon U_j^a \rightarrow V_{\rmf , j} , \quad (\vct{p} , \theta_n ) \mapsto \bigl( \vct{\psi}_j (\vct{p})  , \theta_n \bigr) .
\end{align}
\end{subequations}
%\end{linenomath}
Further, for $f \in H^1 (\Gamma_a )$, we decompose the gradient~$\smash{\nabla_{\! \Gamma_a }} f$ into a tangential and a normal component, i.e., 
%\begin{linenomath}
\begin{align}
\nabla_{\!\Gamma_a } f &\hphantom{:}= \nablaGamma f + \nablaN f , \quad\enspace
\nablaN f (\vct{p} , \theta_n ) := \partialDer{f (\vct{p}, \theta_n) }{\theta_n } \vct{N} (\vct{p}) .
\end{align}
%\end{linenomath}

Next, we write $\smash{\matr{S}^{\vct{\psi}_j} (\vct{\theta}^\prime ) \in \mathbb{R}^{(n-1)\times (n-1)}}$ for the matrix representation of the shape operator~$\smash{\mathcal{S}_{\invpsi_j (\vct{\theta}^\prime )}}$ of~$\Gamma$ at~$\smash{\invpsi_j(\vct{\theta}^\prime ) }$  with respect to the basis
%\begin{linenomath}
\begin{align}
\biggl\{ \partialDer{\invpsi_j (\vct{\theta}^\prime )}{\theta_1} , \dots , \partialDer{\invpsi_j (\vct{\theta}^\prime )}{\theta_{n-1}}\biggr\} \subset \rmT_{\invpsi_j (\vct{\theta}^\prime )} \Gamma  . \label{eq:tangentbasis}
\end{align}
%\end{linenomath}
Details on the shape operator~$\smash{\mathcal{S}_{\invpsi_j (\vct{\theta}^\prime )}}$ can be found in Appendix~\ref{sec:secA2}.
In addition, for $j \in J$ and $\smash{\vct{\theta} = ( \vct{\theta}^\prime , \theta_n ) \in V_{\rmf , j}}$ or~$\smash{V_{\pm , j}}$, we introduce the abbreviations
%\begin{linenomath}
\begin{subequations}
\begin{align}
\matr{R}_{\rmf , j}^\epsilon (\vct{\theta})  &:= \matr{I}_{n-1} - \epsilon \theta_n \matr{S}^{\vct{\psi}_j} (\vct{\theta}^\prime )  , \\
\matr{R}_{\pm , j }^\epsilon (\vct{\theta} ) &:= \matr{I}_{n-1} - \invt_\pm^\epsilon \bigl(  \invpsi_j (\vct{\theta}^\prime ) , \theta_n \bigr) \matr{S}^{\vct{\psi}_j} (\vct{\theta}^\prime ) , 
\end{align}%
\end{subequations}%
%\end{linenomath}
where $\smash{\matr{I}_{n-1} \in \mathbb{R}^{(n-1)\times (n-1)}}$ is the identity matrix. 
Besides, we define the operators%
%\begin{linenomath}
\begin{subequations}
\begin{align}
\invR^\epsilon_\rmf \bigr\vert_{(\vct{p} , \theta_n ) }  &\colon \rmT_\vct{p}\Gamma \rightarrow \rmT_\vct{p}\Gamma , \\
\invR^\epsilon_\pm \bigr\vert_{\vct{x}} &\colon \rmT_{\mathcal{P}^{\partial G} (\vct{x})} \Gamma \rightarrow  \rmT_{\mathcal{P}^{\partial G} (\vct{x})} \Gamma , \\
\mathcal{R}^0_\pm \bigr\vert_{\vct{x}} &\colon \rmT_{\mathcal{P}^{\partial G} (\vct{x})} \Gamma \rightarrow  \rmT_{\mathcal{P}^{\partial G} (\vct{x})} \Gamma 
\end{align}%
\label{eq:Roperators}%
\end{subequations}%
%\end{linenomath}
for all $\smash{(\vct{p} , \theta_n ) \in \Gamma_a }$ and $\smash{\vct{x}\in \Omega_{\pm , \mathrm{in}}^0}$ by
%\begin{linenomath}
\begin{subequations}
\begin{align}
\invR^\epsilon_\rmf \bigr\vert_{(\vct{p} , \theta_n ) }  &:=  \Bigl( \operatorname{id}_{T_\vct{p}\Gamma }  - \, \epsilon \theta_n \mathcal{S}_\vct{p} \Bigr)^{\! -1} , \\
\invR^\epsilon_\pm \bigr\vert_{\vct{x}} &:= \Bigl( \operatorname{id}_{T_{\mathcal{P}^{\partial G} (\vct{x})}\Gamma }  -\, \invt_\pm^\epsilon \bigl( {\mathcal{P}^{\partial G} (\vct{x})} , -d_\leftrightarrow^{\partial G} (\vct{x}) \bigr) \mathcal{S}_{\mathcal{P}^{\partial G} (\vct{x})} \Bigr)^{\! -1} , \\
\mathcal{R}^0_\pm \bigr\vert_{\vct{x}} &:=  \operatorname{id}_{T_{\mathcal{P}^{\partial G} (\vct{x})}\Gamma }  +\, d_\leftrightarrow^{\partial G} (\vct{x})  \mathcal{S}_{\mathcal{P}^{\partial G} (\vct{x})} .
\end{align}%
\end{subequations}%
%\end{linenomath}
The operators in \cref{eq:Roperators} have the following properties.
In particular, we can characterize their behavior as $\epsilon \rightarrow 0$.
\begin{lemma} \label{lem:Roperators}
\begin{enumerate}[topsep=0pt,label=(\roman*),wide]
\item The operators $\smash{\invR_\rmf^\epsilon }$ and $\smash{\invR_\pm^\epsilon }$ exist for all $\epsilon \in (0, \hat{\epsilon} ]$.
\item \label{item:Rselfadj} For all $\smash{(\vct{p}, \theta_n ) \in \Gamma_a}$ and $\smash{\vct{x} \in \Omega^0_{\pm , \mathrm{in}}}$, the operators 
%\begin{linenomath}
\begin{align*}
\invR_\rmf^\epsilon \bigr\vert_{(\vct{p},\theta_n )} , \enspace \invR_\pm^\epsilon \bigr\vert_\vct{x} , \enspace \text{and} \enspace \mathcal{R}^0_\pm \bigr\vert_\vct{x}
\end{align*}
%\end{linenomath}
are self-adjoint for $\smash{\epsilon \in ( 0 , \hat{\epsilon}]}$.
In particular, for $i \in \{ + , - , \rmf \} $, it is 
%\begin{linenomath}
\begin{align}
\matr{g}\vert^{\vct{\psi}_j } (\vct{\theta}^\prime ) \matr{R}_{i, j }^\epsilon (\vct{\theta} ) &= 
\bigl[ \matr{R}_{i, j }^\epsilon (\vct{\theta} ) \bigr]^\trp  \matr{g}\vert^{\vct{\psi}_j } (\vct{\theta}^\prime ) .
\end{align}
%\end{linenomath}
\item For $j\in J$ and $\smash{\epsilon \in ( 0, \hat{\epsilon}]}$, the matrix representations of the operators
%\begin{linenomath}
\begin{align*}
\invR_\rmf^\epsilon \bigr\vert_{\invpsi_j^a (\vct{\theta } ) }, \enspace \invR_\pm^\epsilon \bigr\vert_{\invpsi_{\pm , j}^0 (\vct{\theta})}, \enspace \text{and}\enspace \mathcal{R}_\pm^0 \bigr\vert_{\invpsi_{\pm , j}^0 (\vct{\theta})}
\end{align*}
%\end{linenomath}
with respect to the basis~\eqref{eq:tangentbasis} are given by $\smash{\bigl[\matr{R}_{\rmf , j}^\epsilon (\vct{\theta }) \bigr]^{-1}}$, $\smash{\bigl[\matr{R}_{\pm , j}^\epsilon (\vct{\theta }) \bigr]^{-1}}$, and $\smash{\matr{R}_{\pm , j}^0 (\vct{\theta }) }$. 
\item As $\epsilon \rightarrow 0$, we have
%\begin{linenomath}
\begin{subequations}
\begin{enumerate}[itemindent=15pt,itemsep=2pt,topsep=4pt]
\item \hphantom{(b)} \hspace{-14pt} $\sup\nolimits_{( \vct{p} ,\theta_n ) \in \Gamma_a } \normsize{\Big}{  \operatorname{id}_{\rmT_{\vct{p}}\Gamma} - \invR^\epsilon_\rmf \bigr\vert_{(\vct{p},\theta_n )} } =  \mathcal{O} (\epsilon ), \hfill \textup{\inlineeqnum}$ 
\item \hphantom{(a)} \hspace{-14pt} $\sup\nolimits_{ \vct{x} \in \Omega^0_{\pm , \mathrm{in}}} \normsize{\Big}{ \operatorname{id}_{\rmT_{\mathcal{P}^{\partial G}(\vct{x})}\Gamma} - \invR^\epsilon_\pm \bigr\vert_{\vct{x}} \circ \mathcal{R}^0_\pm \bigr\vert_{\vct{x}} } =  \mathcal{O} (\epsilon )  \hfill \textup{\inlineeqnum}$
\end{enumerate}
for $\smash{(\vct{p},\theta_n ) \in \Gamma_a }$ and $\smash{\vct{x}\in \Omega^0_{\pm , \mathrm{in}}}$.
\end{subequations}
%\end{linenomath}
\end{enumerate}
\end{lemma}
\begin{proof}
\begin{enumerate}[topsep=0pt,label=(\roman*),wide]
\item Using \cref{eq:eps} and the self-adjointness of~$\smash{\mathcal{S}_\vct{p}}$, we find 
%\begin{linenomath}
\begin{align*}
\norm{\epsilon \theta_n \mathcal{S}_\vct{p}} \le \epsilon \kappa_\mathrm{max} \max_{i=\pm } \bigl\{ \norm{a_i}_{L^\infty (\gamma ) }  \bigr\} \le \frac{\hathat{\epsilon }}{2} \kappa_{\mathrm{max}} \le \frac{1}{6 } < 1 .
\end{align*}
%\end{linenomath}
Thus, the operator $\smash{\invR_\rmf^\epsilon \bigr\vert_{(\vct{p} , \theta_n )}}$ exists for all $\smash{(\vct{p} , \theta_n ) \in \Gamma_a }$ and $\epsilon \in ( 0 , \hat{\epsilon } ]$.

Further, with \cref{eq:eps} and the self-adjointness of~$\smash{\mathcal{S}_\vct{p}}$, we have 
%\begin{linenomath}
\begin{align*}
\normsize{\big}{d_\leftrightarrow^{\partial G} (\vct{x}) \mathcal{S}_{\mathcal{P}^{\partial G} (\vct{x})}} \le \hathat{\epsilon } \kappa_\mathrm{max} \le \frac{1}{3} < 1
\end{align*}
%\end{linenomath}
for all $\smash{\vct{x} \in \Omega_{\pm , \mathrm{in}}^0}$ so that $\smash{\mathcal{R}^0_\pm \bigr\vert_\vct{x}}$ is invertible with 
%\begin{linenomath}
\begin{align*}
\norm{\bigl[ \mathcal{R}^0_\pm \bigr\vert_\vct{x} \bigr]^{-1} } \le \frac{1}{1 - \normsize{\big}{d_\leftrightarrow^{\partial G} (\vct{x}) \mathcal{S}_{\mathcal{P}^{\partial G} (\vct{x})}}} \le \frac{3}{2} .
\end{align*}
%\end{linenomath}
Besides, it is 
%\begin{linenomath}
\begin{multline*}
\norm{\epsilon a_\pm \bigl( \mathcal{P}^{\partial G} (\vct{x}) \bigr) \bigl[ \hathat{\epsilon}^{-1} d_\leftrightarrow^{\partial G} (\vct{x} ) \pm 1 \bigr] \mathcal{S}_{\mathcal{P}^{\partial G} (\vct{x})}} \\
 \le  \frac{3}{2}\epsilon \norm{a_\pm}_{L^\infty (\gamma ) } \kappa_\mathrm{max}    \le \frac{1}{4} < \norm{\bigl[ \mathcal{R}^0_\pm \bigr\vert_\vct{x} \bigr]^{-1} }^{-1} ,
\end{multline*}
%\end{linenomath}
where we have used that $\smash{0 \le \abssize{ }{\hathat{\epsilon}^{-1} d_\leftrightarrow^{\partial G} (\vct{x} ) \pm 1} \le \frac{3}{2}}$.
Consequently, the operator 
%\begin{linenomath}
\begin{align*}
\invR^\epsilon_\pm \bigr\vert_\vct{x} = \Bigl[ \mathcal{R}^0_\pm \bigr\vert_\vct{x} - \epsilon a_\pm \bigl( \mathcal{P}^{\partial G} (\vct{x}) \bigr) \bigl[ \hathat{\epsilon}^{-1} d_\leftrightarrow^{\partial G} (\vct{x} ) \pm 1 \bigr] \mathcal{S}_{\mathcal{P}^{\partial G} (\vct{x})} \Bigr]^{-1}
\end{align*}
%\end{linenomath}
exists for all $\smash{\vct{x} \in \Omega_{\pm , \mathrm{in}}^0}$ and $\smash{\epsilon \in ( 0 , \hat{\epsilon }]}$. 
\item The result follows directly from the self-adjointness of the shape operator.
\item We have 
%\begin{linenomath}
\begin{align*}
\partialDer{\invpsi_j (\vct{\theta}^\prime )}{\theta_i } &= \matr{D} \invpsi_j (\vct{\theta}^\prime )  \matr{R}_{\rmf, j}^\epsilon (\vct{\theta} ) \bigl[ \matr{R}_{\rmf, j}^\epsilon (\vct{\theta} )  \bigr]^{-1} \vct{e}_i \\
&= \Bigl( \operatorname{id}_{\rmT_{\invpsi_j(\vct{\theta}^\prime ) }\Gamma} - \epsilon \theta_n \mathcal{S}_{\invpsi_j (\vct{\theta}^\prime ) } \Bigr) \Bigl( \matr{D} \invpsi_j (\vct{\theta}^\prime ) \bigl[ \matr{R}_{\rmf, j}^\epsilon (\vct{\theta} )  \bigr]^{-1} \vct{e}_i \Bigr)
\end{align*}
%\end{linenomath}
for $i \in \{ 1 , \dots , n-1\}$, where $\smash{\vct{e}_i \in \mathbb{R}^{n-1}}$ denotes the $i$th unit vector, and hence 
%\begin{linenomath}
\begin{align*}
\matr{D} \invpsi_j (\vct{\theta}^\prime ) \bigl[ \matr{R}_{\rmf, j}^\epsilon (\vct{\theta} )  \bigr]^{-1} \vct{e}_i &= \invR^\epsilon_\rmf \bigr\vert_{\invpsi_j^a (\vct{\theta})} \!\biggl( \partialDer{\invpsi_j (\vct{\theta}^\prime )}{\theta_i }\biggr) .
\end{align*}
%\end{linenomath}
The result for $\smash{\invR^\epsilon_\pm }$ follows analogously. The result for $\smash{\mathcal{R}_\pm^0}$ is trivial.
\item[(iv-a)] Using (ii), we find 
%\begin{linenomath}
\begin{align*}
&\sup_{( \vct{p} ,\theta_n ) \in \Gamma_a } \norm{ \operatorname{id}_{\rmT_\vct{p}\Gamma } -  \invR^\epsilon_\rmf \bigr\vert_{(\vct{p},\theta_n )} } = \sup_{( \vct{p} ,\theta_n ) \in \Gamma_a } \max_{k\in\{ 1, \dots , n-1\} } \abs{1 - \frac{1}{1 - \epsilon \theta_n \kappa_k (\vct{p})}}   =  \mathcal{O} (\epsilon ) . 
\end{align*}
%\end{linenomath}
Here, $\smash{\kappa_k \in \mathcal{C}^0 ( \partial G) }$, $k \in \{ 1, \dots , n-1 \}$, denote the principal curvatures on~$\partial G$, which are bounded due to the compactness of~$\partial G$.  
\item[(iv-b)] Using \cref{eq:eps} and the self-adjointness of~$\smash{\mathcal{S}_{\mathcal{P}^{\partial G} (\vct{x})}}$, we find
%\begin{linenomath}
\begin{align*}
\sup_{ \vct{x} \in \Omega^0_{\pm , \mathrm{in}}} \abssize{\big}{\invt_\pm^\epsilon \bigl( \mathcal{P}^{\partial G} (\vct{x} ) , -d_\leftrightarrow^{\partial G} (\vct{x} ) \bigr) } \normsize{\big}{\mathcal{S}_{\mathcal{P}^{\partial G}(\vct{x})}} \le \Bigl[ \hathat{\epsilon } + \frac{3}{2}\epsilon \norm{a_\pm}_{L^\infty (\gamma ) } \Bigr] \kappa_{\max} \le \frac{7}{12} < 1.
\end{align*}
%\end{linenomath}
Thus, we can express $\smash{\invR_\pm^\epsilon \bigr\vert_{\vct{x}}}$ as a Neumann series and obtain
%\begin{linenomath}
\begin{align*}
&\invR^\epsilon_\pm \bigr\vert_{\vct{x}} \circ \mathcal{R}^0_\pm \bigr\vert_{\vct{x}} =  \biggl[ \sum_{k=0}^\infty \invt_\pm^\epsilon \bigl( \mathcal{P}^{\partial G } (\vct{x} ) , - d_\leftrightarrow^{\partial G} (\vct{x}) \bigr)^k \mathcal{S}^k_{\mathcal{P}^{\partial G} (\vct{x})} \biggr] \circ \mathcal{R}^0_\pm \bigr\vert_{\vct{x}} \\
&\hspace{1cm} = \operatorname{id}_{\rmT_{\mathcal{P}^{\partial G} (\vct{x})}\Gamma } + \bigl[ \invt_\pm^\epsilon \bigl( \mathcal{P}^{\partial G } (\vct{x} ) , - d_\leftrightarrow^{\partial G} (\vct{x}) \bigr) + d_\leftrightarrow^{\partial G} (\vct{x}) \bigr] \invR^\epsilon_\pm \bigr\vert_{\vct{x}} \circ \mathcal{S}_{\mathcal{P}^{\partial G} (\vct{x})} ,
\end{align*}
%\end{linenomath}
where $\smash{ \invt_\pm^\epsilon \bigl( \mathcal{P}^{\partial G } (\vct{x} ) , - d_\leftrightarrow^{\partial G} (\vct{x}) \bigr) + d_\leftrightarrow^{\partial G} (\vct{x})   = \mathcal{O}(\epsilon )}$. 
\qedhere
\end{enumerate} 
\end{proof}

Further, for $j \in J$, the Jacobians of the inverse charts~$\smash{\invpsi_{\rmf , j}^\epsilon}$, $\smash{\invpsi_{\pm , j}^\epsilon}$ are given by
%\begin{linenomath}
\begin{subequations}
\begin{align}
\label{eq:jac_f}\matr{D} \invpsi_{\rmf , j}^\epsilon (\vct{\theta}) &= \bigl[\! \begin{array}{c|c}
\matr{D}\invpsi_j (\vct{\theta}^\prime ) \matr{R}_{\rmf , j}^\epsilon (\vct{\theta} ) & \epsilon \vct{N} \bigl( \invpsi_j (\vct{\theta}^\prime ) \bigr)
\end{array} \! \bigr], \\
\label{eq:jac_pm}\matr{D} \invpsi_{\pm , j}^\epsilon (\vct{\theta} ) &= \matr{A}^\epsilon_{\pm , j} (\vct{\theta }) + \epsilon \vct{N} \bigl(\invpsi_j (\vct{\theta}^\prime ) \bigr) \bigl[\vct{v}_{\pm , j } (\vct{\theta} ) \bigr]^\mathrm{t} ,
\end{align}
\end{subequations}
%\end{linenomath}
where 
%\begin{linenomath}
\begin{subequations}
\begin{align}
\matr{A}^\epsilon_{\pm , j} (\vct{\theta }) &:=\left[ \! \begin{array}{c|c}
\matr{D} \invpsi_j (\vct{\theta}^\prime ) \matr{R}_{\pm , j}^\epsilon (\vct{\theta}) & \vct{N} \bigl( \invpsi_j (\vct{\theta}^\prime ) \bigr) 
\end{array} \! \right] , \\
\vct{v}_{\pm , j } (\vct{\theta} ) &:= \left[\begin{array}{c}
\bigl[ \pm  1 - \hathat{\epsilon}^{-1} \theta_n \bigr] \bigl[ \matr{D} \invpsi_j (\vct{\theta}^\prime ) \bigr]^\trp \nablaGamma a_\pm \bigl( \invpsi_j (\vct{\theta}^\prime ) \bigr)  \\ [3pt]
\hline \\[-10pt]
-\hathat{\epsilon}^{-1} a_\pm \bigl( \invpsi_j (\vct{\theta}^\prime ) \bigr)
\end{array} \right].
\end{align}
\end{subequations}
%\end{linenomath}
Consequently, with $\smash{\bigl[ \matr{D} \invpsi_j (\vct{\theta}^\prime ) \bigr]^\trp \vct{N} \bigl(\invpsi_j (\vct{\theta}^\prime ) \bigr) = \vct{0}  }$, we find that the metric tensors of $\smash{\Omega_\rmf^\epsilon }$ and $\smash{\Omega_{\pm , \mathrm{in}}^\epsilon }$ in coordinates of the charts~$\smash{\vct{\psi}_{\rmf ,j}^\epsilon}$ and $\smash{\vct{\psi}_{\pm ,j}^\epsilon}$, $j \in J$, are given by 
%\begin{linenomath}
\begin{subequations}
\begin{align} \label{eq:metricf}
{\matr{g}}\vert^{\vct{\psi}_{\rmf ,j}^\epsilon } (\vct{\theta }) &= \left[\begin{array}{c|c}
\bigl[ \matr{R}_{\rmf , j}^\epsilon (\vct{\theta}) \bigr]^\mathrm{t} {\matr{g}}\vert^{\vct{\psi}_j } (\vct{\theta}^\prime ) \matr{R}_{\rmf , j}^\epsilon (\vct{\theta}) & \vct{0} \\[3pt]
\hline \\[-9pt]
\vct{0} & \epsilon^2
\end{array}\right] , \\
\begin{split}
{\matr{g}}\vert^{\vct{\psi}_{\pm ,j}^\epsilon } (\vct{\theta }) &= \left[\begin{array}{c|c}
 \bigl[ \matr{R}_{\pm , j}^\epsilon (\vct{\theta }) \bigr]^\mathrm{t} {\matr{g}}\vert^{\vct{\psi}_j } (\vct{\theta}^\prime )\matr{R}_{\pm , j}^\epsilon (\vct{\theta })  & \vct{0}  \\[3pt]
\hline \\[-9pt]
\vct{0} & 1
\end{array}\right] \\
&\hspace{2.25cm} + \bigl( \epsilon \vct{v}_{\pm ,j} (\vct{\theta }) + \vct{e}_n \bigr)\bigl( \epsilon \vct{v}_{\pm ,j} (\vct{\theta }) + \vct{e}_n \bigr)^{\!\trp } - \vct{e}_n \vct{e}_n^\trp  , 
\end{split}  
\end{align}%
\label{eq:metric}%
\end{subequations}%
%\end{linenomath}
where $\smash{\vct{e}_n \in \mathbb{R}^n}$ is the $n$th unit vector and $\smash{{\matr{g}}\vert^{\vct{\psi}_{j} }  }$ denotes the metric tensor on~$\Gamma$ in coordinates of the chart~$\smash{\vct{\psi}_j}$. 
Subsequently, for $j \in J$, we will use the notation
%\begin{linenomath}
\begin{align}
\mu_j &:= \sqrt{\det \matr{g}\vert^{\vct{\psi}_j}} , \quad\enspace 
\mu_{\pm , j}^\epsilon := \sqrt{\det \matr{g}\vert^{\vct{\psi}_{\pm , j}^\epsilon }} , \quad\enspace
\mu_{\rmf , j }^\epsilon := \sqrt{\det \matr{g}\vert^{\vct{\psi}_{\rmf , j}^\epsilon }} .
\end{align}
%\end{linenomath}
Moreover, we have the following result.
\begin{lemma} \label{lem:det}
Let $\epsilon \in ( 0, \hat{\epsilon} ]$. Then, for $j\in J$, we have
%\begin{linenomath}
\begin{subequations}
\begin{align}
\det \matr{R}_{\rmf , j}^\epsilon (\vct{\theta } ) &= 1 + \mathcal{O} (\epsilon ) , \\
\det \matr{R}_{\pm , j}^\epsilon (\vct{\theta } ) &= \bigl[ 1 + \mathcal{O}(\epsilon ) \bigr] \det \matr{R}_{\pm , j}^0 (\vct{\theta } ) \label{eq:det_Rpm}
\end{align}%
\label{eq:det_R}%
\end{subequations}%
%\end{linenomath}
as $\epsilon \rightarrow 0$ and, consequently, 
%\begin{linenomath}
\begin{subequations}
\begin{align}
\label{eq:det_f} \det {\matr{g}}\vert^{\vct{\psi}_{\rmf ,j}^\epsilon } (\vct{\theta }) &= \epsilon^2 \bigl[ 1 + \mathcal{O} (\epsilon ) \bigr] \det {\matr{g}}\vert^{\vct{\psi}_j } (\vct{\theta}^\prime ) , \\
\label{eq:det_pm} \det {\matr{g}}\vert^{\vct{\psi}_{\pm ,j}^\epsilon } (\vct{\theta }) &= \bigl[ 1 + \mathcal{O} (\epsilon ) \bigr] \det {\matr{g}}\vert^{\vct{\psi}_{\pm ,j}^0 } (\vct{\theta }) .
\end{align}%
\label{eq:det_f_pm}%
\end{subequations}%
%\end{linenomath}
The prefactors on the right-hand side of \cref{eq:det_f_pm} do not depend on~$j \in J$.
\end{lemma}
\begin{proof}
Given the principal curvatures~$\smash{\kappa_k \in \mathcal{C}^0 ( \partial G ) }$ on~$\partial G$, $k\in\{ 1, \dots , n-1\}$, which are bounded on the compact submanifold~$\partial G$, we have
%\begin{linenomath}
\begin{subequations}
\begin{align*}
\det \matr{R}_{\rmf , j}^\epsilon (\vct{\theta} )  &= \prod_{k=1}^{n-1} \bigl[ 1 - \epsilon \theta_n \kappa_k \bigl( \invpsi_j (\vct{\theta}^\prime ) \bigr) \bigr] , \\
\det \matr{R}_{\pm , j}^\epsilon (\vct{\theta}) &= \prod_{k=1}^{n-1} \bigl[ 1- \invt_\pm^\epsilon \bigl(  \invpsi_j (\vct{\theta}^\prime ) , \theta_n \bigr) \kappa_k \bigl( \invpsi_j (\vct{\theta}^\prime ) \bigr)  \bigr] 
\end{align*}
\end{subequations}
%\end{linenomath}
for $j\in J$, where $\smash{\invt_\pm^\epsilon \bigl(  \invpsi_j (\vct{\theta}^\prime ) , \theta_n \bigr) = \invt_\pm^0 \bigl(  \invpsi_j (\vct{\theta}^\prime ) , \theta_n \bigr) + \mathcal{O}(\epsilon ) }$.
This yields \cref{eq:det_R} and \cref{eq:det_f}. 
Moreover, with \Cref{lem:detlemma} below and \cref{eq:det_Rpm}, we find 
%\begin{linenomath}
\begin{multline*}
\det {\matr{g}}\vert^{\vct{\psi}_{\pm ,j}^\epsilon } (\vct{\theta }) = \Bigl( 1 - \epsilon \hathat{\epsilon}^{-1} a_\pm \bigl( \invpsi_j (\vct{\theta}^\prime ) \bigr)  \Bigr)^2 \bigl( \det \matr{R}_{\pm ,j}^\epsilon (\vct{\theta}) \bigr)^2 \det \matr{g}\vert^{\vct{\psi}_j } (\vct{\theta}^\prime )    \\
= \bigl[ 1 + \mathcal{O} (\epsilon ) \bigr] \bigl( \det \matr{R}_{\pm , j}^0 (\vct{\theta})\bigr)^2 \det \matr{g}\vert^{\vct{\psi}_j } (\vct{\theta}^\prime ) = \bigl[ 1 +\mathcal{O}(\epsilon ) \bigr] \det \matr{g}\vert^{\vct{\psi}_{\pm ,j}^0 } (\vct{\theta}) 
. \tag*{\qedhere}
\end{multline*}
%\end{linenomath}
\end{proof}
The proof of \Cref{lem:det} makes use of the following determinant lemma, which is a consequence of Sylvester's determinant theorem.
\begin{lemma} \label{lem:detlemma}
Let $\smash{\matr{A} \in \mathbb{R}^{n\times n}}$ be invertible and $\vct{c} , \vct{d}, \vct{e} , \vct{f} \in \mathbb{R}^n$. 
Then, we have
%\begin{linenomath}
\begin{align}
\det \bigl( \matr{A} + \vct{c} \vct{d}^\trp + \vct{e} \vct{f}^\trp \bigr) &= \det (\matr{A} ) \bigl[ \bigl( \vct{d}^\trp \matr{A}^{-1} \vct{c} + 1 \bigr) \bigl( \vct{f}^\trp \matr{A}^{-1} \vct{e} + 1 \bigr) - \vct{d}^\trp \matr{A}^{-1} \vct{e} \vct{f}^\trp \matr{A}^{-1} \vct{c} \bigr] . 
\end{align}
%\end{linenomath}
\end{lemma}

Next, given a partition of unity~$\smash{\{ \chi_j \}_{j\in J }}$ of~$\Gamma$ that is subordinate to the covering~$\smash{\{ U_j \}_{j\in  J }}$, we define the partitions of unity
\begin{itemize}[wide]
\item $\smash{\{ \chi^\epsilon_{\pm , j }\}_{j\in J }}$ on~$\smash{\Omega_{\pm , \mathrm{in}}^\epsilon}$ subordinate to~$\smash{ \{ U_{\pm ,j}^\epsilon \}_{j\in J }}$ by $\smash{\chi^\epsilon_{\pm , j} := \chi_j \circ \mathcal{P}^{\partial G}\bigr\vert_{\Omega_{\pm,\mathrm{in}}^\epsilon }}$,
\item $\smash{\{ \chi^\epsilon_{\mathrm{f}, j} \}_{j\in J }}$ on~$\smash{\Omega_\mathrm{f}^\epsilon}$ subordinate to~$\smash{ \{ U_{\mathrm{f},j}^\epsilon \}_{j\in J }}$ by $\smash{\chi^\epsilon_{\mathrm{f}, j} := \chi_j \circ\mathcal{P}^{\partial G}\bigr\vert_{\Omega_\rmf^\epsilon }}$,
\item $\smash{\{ \chi^a_j \}_{j\in J }}$ on~$\smash{\Gamma_a}$ subordinate to~$\smash{ \{ U_j^a \}_{j\in J }}$ by $\smash{\chi^a_{ j} (\vct{p} , \theta_n ) := \chi_j (\vct{p}) }$.
\end{itemize}
Further, for $\epsilon \in (0, \hat{\epsilon} ]$, we define the transformations $\smash{\mathcal{Y}^\epsilon_\pm \colon  L^2 ( \Omega_\pm^0 ) \rightarrow  L^2 ( \Omega_\pm^\epsilon )}$ and $\smash{\mathcal{Y}^\epsilon_\mathrm{f} \colon  L^2 ( \Gamma_a ) \rightarrow  L^2 ( \Omega_\mathrm{f}^\epsilon )}$ by
%\begin{linenomath}
\begin{subequations}
\begin{align}
(\mathcal{Y}^\epsilon_\pm \phi_\pm ) (\vct{x}) &:= \begin{cases}
\phi_\pm \bigl( \vct{T}_\pm^\epsilon (\vct{x} ) \bigr) , &\text{if } \vct{x} \in \Omega_{\pm , \mathrm{in}}^\epsilon , \\
\phi_\pm (\vct{x} ) &\text{if } \vct{x} \in \Omega_{\pm , \mathrm{out}} ,
\end{cases} \\
(\mathcal{Y}^\epsilon_\mathrm{f} \phi_\mathrm{f} )  ( \vct{x} ) &\mapsto \phi_\mathrm{f} \bigl( \mathcal{P}^{\partial G } (\vct{x} ) , -\epsilon^{-1} d_\leftrightarrow^{\partial G } (\vct{x} ) \bigr) .
\end{align}
\end{subequations}
%\end{linenomath}
The inverse maps $\smash{\invY_\pm^\epsilon := ( \mathcal{Y}_\pm^\epsilon )^{-1}}$ and $\smash{\invY_\rmf^\epsilon := (\mathcal{Y}^\epsilon_\rmf )^{-1} }$  are given by 
%\begin{linenomath}
\begin{subequations}
\begin{align}
 \bigl( \invY_\pm^\epsilon   \phi_\pm^\epsilon \bigr) (\vct{x} ) &:= \begin{cases}
\phi_\pm^\epsilon \bigl( \invT_\pm^\epsilon (\vct{x} ) \bigr) , &\text{if } \vct{x} \in \Omega_{\pm , \mathrm{in}}^0 , \\
\phi_\pm ^\epsilon (\vct{x} ) &\text{if } \vct{x} \in \Omega_{\pm , \mathrm{out}} ,
\end{cases} \\
 \bigl( \invY_\rmf^\epsilon \phi_\mathrm{f}^\epsilon \bigr) (\vct{p} , \theta_n ) &:= \phi_\mathrm{f}^\epsilon \bigl( \vct{p} + \epsilon \theta_n \vct{N}(\vct{p} ) \bigr) .
\end{align}
\end{subequations}
%\end{linenomath}
Moreover, we define the product map
%\begin{linenomath}
\begin{align}
\begin{split}
\mathcal{Y}^\epsilon \colon L^2 (\Omega_+^0 ) \times L^2 (\Omega_-^0 ) \times L^2 (\Gamma_a ) &\rightarrow L^2 (\Omega_+^\epsilon ) \times L^2 (\Omega_-^\epsilon ) \times L^2 (\Omega_\rmf^\epsilon ) , \\
(\phi_+ , \phi_- , \phi_\rmf ) &\mapsto \bigl( \mathcal{Y}_+^\epsilon \phi_+ , \mathcal{Y}_-^\epsilon \phi_- , \mathcal{Y}_\rmf^\epsilon \phi_\rmf \bigr) 
\end{split}
\end{align}
%\end{linenomath}
and write $\smash{\invY^\epsilon := (\mathcal{Y}^\epsilon)^{-1}}$ for its inverse. 
Then, the following result holds true.  
\begin{lemma} \label{lem:trafo}
There is an $\bar{\epsilon} > 0$ such that the following results hold for all $\epsilon \in ( 0,\bar{\epsilon} ]$.
\begin{enumerate}[label=(\roman*),wide,topsep=0pt]
\item \label{item:trafo1} $\mathcal{Y}_\rmf^\epsilon \colon L^2 (\Gamma_a ) \rightarrow L^2 (\Omega_\rmf^\epsilon )$ defines an isomorphism with
%\begin{linenomath}
\begin{align}
\norm{\mathcal{Y}_\rmf^\epsilon \phi_\rmf }^2_{L^2 (\Omega_\rmf^\epsilon ) } = \epsilon \bigl[ 1 + \mathcal{O}(\epsilon ) \bigr] \norm{\phi_\rmf}_{L^2 (\Gamma_a ) }^2 
\end{align}
%\end{linenomath}
for all $\smash{ \phi_\rmf \in L^2 (\Gamma_a ) }$ as $\epsilon \rightarrow 0$.
\item \label{item:trafo2} $\mathcal{Y}_\pm^\epsilon \colon L^2 (\Omega_\pm^0 ) \rightarrow L^2 (\Omega_\pm^\epsilon )$ defines an isomorphism with
%\begin{linenomath}
\begin{align}
\norm{\mathcal{Y}_\pm^\epsilon \phi_\pm }_{L^2 (\Omega_\pm^\epsilon ) } =  \bigl[ 1 + \mathcal{O}(\epsilon ) \bigr] \norm{\phi_\pm }_{L^2 (\Omega_\pm^0 ) } 
\end{align}
%\end{linenomath}
for all $\smash{ \phi_\pm \in L^2 (\Omega_\pm^0 ) }$ as $\epsilon \rightarrow 0$. 
\item \label{item:trafo3} ${\mathcal{Y}_\rmf^\epsilon \bigr\vert_{H^1 (\Gamma_a )  } \colon H^1 (\Gamma_a ) \rightarrow H^1 (\Omega_\rmf^\epsilon ) }$ is an isomorphism.
In particular, we have  
%\begin{linenomath}
\begin{align}
\nabla \bigl( \mathcal{Y}_\rmf^\epsilon  \phi_\rmf \bigr) \bigl(\vct{p} + \epsilon \theta_n \vct{N}(\vct{p}) \bigr) &=  \bigl( \invR_\rmf^\epsilon \nablaGamma \phi_\rmf \bigr) (\vct{p} , \theta_n )  + \epsilon^{-1} \nablaN \phi_\rmf (\vct{p} , \theta_n )   
\end{align}
%\end{linenomath}
for $\phi_\rmf \in H^1 (\Gamma_a )$ and a.a.~$(\vct{p} ,\theta_n ) \in \Gamma_a$ and hence, as $\epsilon \rightarrow 0$,
%\begin{linenomath}
\begin{align}
\normsize{\big}{\nabla ( \mathcal{Y}_\rmf^\epsilon \phi_\rmf ) }_{\vct{L}^2 (\Omega_\rmf^\epsilon ) }^2 &= \bigl[ 1 + \mathcal{O}(\epsilon ) \bigr] \bigl( \epsilon \normsize{\big}{\nablaGamma \phi_\rmf }_{\vct{L}^2 (\Gamma_a ) }^2 + \epsilon^{-1} \normsize{\big}{\nablaN \phi_\rmf }_{\vct{L}^2 (\Gamma_a ) }^2 \bigr)  . \label{eq:gradtrafo_f_L2}
\end{align}
%\end{linenomath}
\item \label{item:trafo4} ${\mathcal{Y}_\pm^\epsilon \bigr\vert_{H^1 (\Omega_\pm^0 )  } \colon H^1 (\Omega_\pm^0 ) \rightarrow H^1 (\Omega_\pm^\epsilon ) }$ is an isomorphism. In particular, we have 
%\begin{linenomath}
\begin{align}
\nabla \bigl( \mathcal{Y}_\pm^\epsilon \phi_\pm \bigr) \bigl( \invT_\pm^\epsilon (\vct{x} ) \bigr) &= \matr{M}_\pm^\epsilon (\vct{x} ) \nabla \phi_\pm (\vct{x} )  \label{eq:gradtrafo_pm}
\end{align}
%\end{linenomath}
for $\smash{\phi_\pm \in H^1 (\Omega_\pm^0 )}$ and a.a.~$\vct{x} \in \Omega_{\pm , \mathrm{in}}^0$, where 
%\begin{linenomath}
\begin{align}
\matr{M}_\pm^\epsilon ( \vct{x} ) &:= \bigl[ \matr{D}  \vct{T}_\pm^\epsilon \bigl( \invT^\epsilon_\pm ( \vct{x} ) \bigr)\bigr]^\trp  = \bigl[ \matr{D} \invT_\pm^\epsilon (\vct{x}) \bigr]^{-\trp}.
\end{align} 
%\end{linenomath}
Besides, as $\epsilon \rightarrow 0$, it is
%\begin{linenomath}
\begin{align}
\sup\nolimits_{\vct{x}\in \Omega_{\pm , \mathrm{in}}^0} \normsize{\big}{\matr{M}^\epsilon_\pm (\vct{x}) - \matr{I}_n} = \mathcal{O}(\epsilon ) ,
\end{align}
%\end{linenomath}
where $\smash{\matr{I}_n \in \mathbb{R}^{n\times n }}$ denotes the identity matrix. 
Thus, we obtain
%\begin{linenomath}
\begin{align}
\normsize{\big}{\nabla ( \mathcal{Y}_\pm^\epsilon \phi_\pm ) }_{\vct{L}^2 (\Omega_\pm^\epsilon ) } &= \bigl[ 1 + \mathcal{O}(\epsilon ) \bigr]  \normsize{\big}{\nabla \phi_\pm }_{\vct{L}^2 (\Omega_\pm^0 ) }  . \label{eq:gradtrafo_pm_L2}
\end{align}
%\end{linenomath}
\end{enumerate}
\end{lemma}
\begin{proof}
\begin{enumerate}[label=(\roman*),wide,topsep=0pt]
\item It is easy to see that $\smash{\mathcal{Y}_\rmf^\epsilon }$ is linear and bijective with inverse~$\smash{\invY_\rmf^\epsilon}$.
Moreover, with \Cref{lem:det}, we have 
%\begin{linenomath}
\begin{align*}
\norm{\mathcal{Y}_\rmf^\epsilon \phi_\rmf }^2_{L^2 (\Omega_\rmf^\epsilon ) } &= \sum_{j\in J} \int_{V_{\rmf , j}}  \bigl[ \chi_{\rmf , j}^\epsilon ( \mathcal{Y}_\rmf^\epsilon \phi_\rmf  )^2 \bigr] \bigr\vert_{\invpsi_{\rmf, j}^\epsilon (\vct{\theta ) }} \, \mu_{\rmf , j}^\epsilon (\vct{\theta}) \,\D\lambda_n (\vct{\theta }) \\
&= \epsilon \bigl[ 1+ \mathcal{O}(\epsilon ) \bigr] \sum_{j\in J} \int_{V_{\rmf , j}}  \bigl[ \chi_j^a \phi_\rmf^2 \bigr]  \bigr\vert_{ \invpsi_j^a (\vct{\theta} ) }\, \mu_j (\vct{\theta}^\prime )  \,\D \theta_n \D \lambda_{n-1} (\vct{\theta}^\prime ) \\
&= \epsilon \bigl[ 1+ \mathcal{O}(\epsilon ) \bigr] \norm{\phi_\rmf}^2_{L^2 (\Gamma_a ) } .
\end{align*}
%\end{linenomath}
\item $\smash{\mathcal{Y}_\pm^\epsilon }$ is clearly linear and bijective with inverse~$\invY_\pm^\epsilon $. 
Further, we have
%\begin{linenomath}
\begin{align*}
\norm{\mathcal{Y}_\pm^\epsilon \phi_\pm }^2_{L^2 (\Omega_{\pm }^\epsilon )} = \norm{ \phi_\pm }^2_{L^2 (\Omega_{\pm , \mathrm{out}} )} + \norm{\phi_\pm \circ \vct{T}_\pm^\epsilon }^2_{L^2 (\Omega_{\pm , \mathrm{in}}^\epsilon )} .
\end{align*}
%\end{linenomath}
Then, by using \Cref{lem:det} and $\smash{\vct{T}_\pm^\epsilon \circ \invpsi_{\pm , j}^\epsilon = \invpsi_{\pm , j}^0 }$, we find
%\begin{linenomath}
\begin{align*}
\norm{\phi_\pm \circ \vct{T}_\pm^\epsilon }^2_{L^2 (\Omega_{\pm , \mathrm{in}}^\epsilon )} &= 
\sum_{j\in J} \int_{V_{\pm , j}} \bigl[ \chi_{\pm , j}^\epsilon \bigl( \phi_\pm \circ \vct{T}_\pm^\epsilon \bigr)^2 \bigr] \bigr\vert_{\invpsi_{\pm , j}^\epsilon (\vct{\theta }) } \, \mu_{\pm , j}^\epsilon (\vct{\theta}) \,\D\lambda_n (\vct{\theta }) \\
& = \bigl[ 1 + \mathcal{O}(\epsilon ) \bigr] \sum_{j\in J }  \int_{V_{\pm , j}} \bigl[ \chi_{\pm , j}^0  \phi_\pm^2 \bigr] \bigr\vert_{\invpsi_{\pm , j}^0 (\vct{\theta }) }\, \mu_{\pm , j}^0 (\vct{\theta }) \,\D\lambda_n (\vct{\theta }) \\
& = \bigl[ 1+ \mathcal{O}(\epsilon ) \bigr] \norm{\phi_\pm}^2_{L^2 (\Omega_{\pm , \mathrm{in}}^0 ) } .
\end{align*}
%\end{linenomath}
\item Let $\smash{\phi_\rmf \in H^1 (\Gamma_a )}$ and $\smash{\phi_\rmf^\epsilon := \mathcal{Y}_\rmf^\epsilon \phi_\rmf}$. 
Then, by using the \cref{eq:jac_f,eq:metricf} and \Cref{lem:Roperators}, we find  
%\begin{linenomath}
\begin{align*}
&\nabla \phi_\mathrm{f}^\epsilon \bigl( \invpsi_{\mathrm{f},j}^\epsilon (\vct{\theta } ) \bigr) = \matr{D} \invpsi^\epsilon_{\mathrm{f} , j} (\vct{\theta } ) \, \matr{g}^{-1} \vert^{\vct{\psi}_{\rmf , j}^\epsilon } (\vct{\theta } ) \nabla \bigl( \phi_\mathrm{f}^\epsilon \circ \invpsi_{\mathrm{f},j}^\epsilon \bigr) (\vct{\theta } ) \\
&\quad = \matr{D} \invpsi_j (\vct{\theta}^\prime )\, \matr{g}^{-1} \vert^{\vct{\psi}_j} (\vct{\theta}^\prime ) \bigl[\matr{R}_{\rmf , j}^\epsilon (\vct{\theta }) \bigr]^{-\mathrm{t}}\, {\nabla}^\prime \bigl( \phi_\mathrm{f}^\epsilon \circ \invpsi_{\mathrm{f},j}^\epsilon \bigr) (\vct{\theta }) + \epsilon^{-1} \nablaN \bigl( \phi_\mathrm{f}^\epsilon \circ \invpsi_{\mathrm{f},j}^\epsilon \bigr) (\vct{\theta }) \\
&\quad = \matr{D} \invpsi_j (\vct{\theta}^\prime )\,  \bigl[\matr{R}_{\rmf , j}^\epsilon (\vct{\theta }) \bigr]^{-1} \matr{g}^{-1} \vert^{\vct{\psi}_j} (\vct{\theta}^\prime )\, {\nabla}^\prime \bigl( \phi_\mathrm{f}^\epsilon \circ \invpsi_{\mathrm{f},j}^\epsilon \bigr) (\vct{\theta }) + \epsilon^{-1} \nablaN \bigl( \phi_\mathrm{f}^\epsilon \circ \invpsi_{\mathrm{f},j}^\epsilon \bigr) (\vct{\theta }) \\
&\quad = \invR_\rmf^\epsilon \bigr\vert_{\invpsi_j^a (\vct{\theta })} \bigl( \nabla_\Gamma \phi_\rmf  \bigr) \bigl( \invpsi_j^a (\vct{\theta}) \bigr) + \epsilon^{-1} \nablaN \phi_\rmf \bigl( \invpsi_j^a (\vct{\theta}) \bigr) .
\end{align*}
%\end{linenomath}
for $j \in J$ and a.a.~$\smash{\vct{\theta} = (\vct{\theta}^\prime , \theta_n ) \in V_{\rmf , j}}$.
Thus, with \Cref{lem:Roperators}, we have 
%\begin{linenomath}
\begin{align*}
\abssize{\big}{\nabla \phi_\rmf^\epsilon \bigl(\vct{p} + \epsilon \theta_n \vct{N}(\vct{p}) \bigr) }^2 &= \bigl[ 1+ \mathcal{O}(\epsilon ) \bigr]  \abssize{\big}{  \nablaGamma \phi_\rmf  (\vct{p} , \theta_n ) }^2  + \epsilon^{-2} \abssize{\big} {\nablaN \phi_\rmf (\vct{p} , \theta_n )   }^2
\end{align*}
%\end{linenomath}
for a.a.~$(\vct{p}, \theta_n ) \in \Gamma_a $ so that \cref{eq:gradtrafo_f_L2} follows with \Cref{lem:det}.
\item 
\Cref{eq:gradtrafo_pm} follows  by applying the chain rule. 
Now, let $\smash{\phi_\pm \in H^1 ( \Omega_\pm^0})$ and $\smash{ \phi_\pm^\epsilon := \mathcal{Y}_\pm^\epsilon \phi_\pm }$.
Then, by using that $\smash{\invpsi_{\pm , j}^\epsilon = \invT_\pm^\epsilon \circ \invpsi_{\pm , j}^0}$, the chain rule yields
%\begin{linenomath}
\begin{align*}
\matr{M}_\pm^\epsilon \bigl( \invpsi_{\pm , j}^0 (\vct{\theta} ) \bigr) &= \bigl[ \matr{D} \invT_\pm^\epsilon \bigl( \invpsi^0_{\pm , j} ( \vct{\theta} ) \bigr)  \bigr]^{- \trp} = \bigl[ \matr{D} \invpsi_{\pm , j}^\epsilon (\vct{\theta } )  \bigr]^{-\trp }  \bigl[ \matr{D} \invpsi_{\pm , j}^0 (\vct{\theta} ) \bigr]^\trp .
\end{align*}
%\end{linenomath}
With \cref{eq:jac_pm} and the Sherman-Morrison formula, we obtain
%\begin{linenomath}
\begin{align*} 
\bigl[ \matr{D} \invpsi_{\pm , j}^\epsilon (\vct{\theta} ) \bigr]^{-1} &= \biggl( \matr{I}_n - \epsilon \frac{\vct{e}_n \bigl[ \vct{v}_{\pm , j } (\vct{\theta} ) \bigr]^\trp}{1 - \epsilon \hathat{\epsilon }^{-1} a_\pm \bigl(\invpsi_j (\vct{\theta}^\prime ) \bigr) }  \biggr)  \bigl[\matr{A}^\epsilon_{\pm , j} (\vct{\theta })   \bigr]^{-1} ,
\end{align*} 
%\end{linenomath}
where $\smash{\matr{I}_n \in \mathbb{R}^{n\times n}}$ is the identity matrix and
%\begin{linenomath}
\begin{align*}
\bigl[\matr{A}^\epsilon_{\pm , j} (\vct{\theta })   \bigr]^{-1} &=  \left[ 
\begin{array}{c|c}
\!\bigl[ \matr{R}_{\pm , j}^\epsilon (\vct{\theta} ) \bigr]^{-1} \matr{g}^{-1} \vert^{\vct{\psi}_j} (\vct{\theta}^\prime ) & \vct{0} \\[3pt]  
\hline  \\[-10pt]
\vct{0} & 1
\end{array}
 \right] 
 \bigl[ \!
\begin{array}{c|c}
 \matr{D} \invpsi_j (\vct{\theta}^\prime ) & \vct{N} \bigl(\invpsi_j (\vct{\theta}^\prime ) \bigr)
\end{array} \!
 \bigr]^\trp .
\end{align*}
%\end{linenomath}
Consequently, with \Cref{lem:Roperators}, we find
%\begin{linenomath}
\begin{align*}
&\matr{M}_\pm^\epsilon \bigl( \invpsi_{\pm , j}^0 (\vct{\theta} ) \bigr) = \bigl[ \matr{A}_{\pm , j}^\epsilon (\vct{\theta} ) \bigr]^{-\trp } \biggl( \matr{I}_n - \epsilon \frac{  \vct{v}_{\pm , j } (\vct{\theta} ) \,\vct{e}_n^\trp}{1 - \epsilon \hathat{\epsilon }^{-1} a_\pm \bigl(\invpsi_j (\vct{\theta}^\prime ) \bigr) }  \biggr)  \bigl[ \matr{A}_{\pm , j}^0 (\vct{\theta} ) \bigr]^{\trp } \\
&\enspace = \matr{D}\invpsi_j (\vct{\theta}^\prime ) \bigl[  \matr{R}_{\pm , j}^\epsilon (\vct{\theta} ) \bigr]^{-1} \matr{R}^0_{\pm , j} (\vct{\theta } ) \matr{g}^{-1} \vert^{\vct{\psi}_j} (\vct{\theta}^\prime ) \bigl[ \matr{D} \invpsi_j (\vct{\theta}^\prime ) \bigr]^\trp + \vct{w}_{\pm ,j}^\epsilon (\vct{\theta}) \bigl[ \vct{N} \bigl( \invpsi_j ( \vct{\theta}^\prime )\bigr)  \bigr]^\trp   ,
\end{align*}
%\end{linenomath}
where we have used the abbreviation
%\begin{linenomath}
\begin{align*}
\vct{w}_{\pm ,j}^\epsilon (\vct{\theta} ) := \frac{ \hathat{\epsilon } \vct{N} \bigl( \invpsi_j (\vct{\theta}^\prime ) \bigr) - \epsilon \bigl[  \pm \hathat{\epsilon } - \theta_n \bigr] \bigl[ \invR_\pm^\epsilon \vert_{\invpsi_{\pm , j}^0 (\vct{\theta})} \nablaGamma a_\pm \bigl( \invpsi_j (\vct{\theta}^\prime )  \bigr)  \bigr]   }{\hathat{\epsilon } - \epsilon a_\pm \bigl( \invpsi_j (\vct{\theta}^\prime ) \bigr) } .
\end{align*}
%\end{linenomath}
Thus, using that 
%\begin{linenomath}
\begin{align*}
\matr{I}_n = \matr{D}\invpsi_j (\vct{\theta}^\prime ) \matr{g}^{-1} \vert^{\vct{\psi}_j} (\vct{\theta}^\prime ) \bigl[ \matr{D}\invpsi_j(\vct{\theta}^\prime ) \bigr]^\trp + \vct{N}\bigl(\invpsi_j (\vct{\theta}^\prime ) \bigr) \bigl[ \vct{N}\bigl( \invpsi_j (\vct{\theta}^\prime ) \bigr) \bigr]^\trp ,
\end{align*}
%\end{linenomath}
we find 
%\begin{linenomath}
\begin{multline*}
\bigl[ \matr{M}^\epsilon_\pm \bigl( \invpsi_j (\vct{\theta}^\prime )  - \matr{I}_n \bigr)  \bigr] \vct{z} = \biggl[ \invR_\pm^\epsilon \big\vert_{\invpsi_{\pm , j}^0 (\vct{\theta})} \circ \mathcal{R}_\pm^0 \big\vert_{\invpsi_{\pm ,j}^0 (\vct{\theta }) } - \operatorname{id}_{\rmT_{\invpsi_j (\vct{\theta}^\prime )}\Gamma } \biggr] \bigl( \vct{\Pi}\big\vert_{\invpsi_j ( \vct{\theta}^\prime ) } \vct{z}\bigr) \\
+ \Bigl[ \vct{w}^\epsilon_{\pm , j} (\vct{\theta}) - \vct{N}\bigl( \invpsi_j (\vct{\theta}^\prime ) \bigr) \Bigr] \bigl[ \vct{N} \bigl( \invpsi_j (\vct{\theta}^\prime ) \bigr) \cdot \vct{z} \bigr]
\end{multline*}
%\end{linenomath}
for all $\smash{\vct{z}\in\mathbb{R}^n }$, where 
%\begin{linenomath}
\begin{align*}
\vct{\Pi}\vert_{\invpsi_j (\vct{\theta}^\prime ) } \colon \mathbb{R}^n \rightarrow \mathbb{R}^n , \quad \vct{u} \mapsto \matr{D} \invpsi_j (\vct{\theta}^\prime ) \matr{g}^{-1} \vert^{\vct{\psi}_j } (\vct{\theta}^\prime ) \bigl[ \matr{D} \invpsi_j (\vct{\theta}^\prime ) \bigr]^\trp \vct{u}
\end{align*} 
%\end{linenomath}
denotes the orthogonal projection onto~$\smash{\rmT_{\invpsi_j (\vct{\theta}^\prime ) } \Gamma }$.
Further, it is easy to see that
%\begin{linenomath}
\begin{align*}
\sup_{j\in J} \sup_{\vct{\theta} \in V_{\rmf, j}} \abs{\vct{w}^\epsilon_{\pm , j} (\vct{\theta}) - \vct{N}\bigl( \invpsi_j (\vct{\theta}^\prime ) \bigr)} = \mathcal{O}(\epsilon ) .
\end{align*} 
%\end{linenomath}
Thus, the result follows with \Cref{lem:Roperators}.
\qedhere 
\end{enumerate}
\end{proof}

\subsection{Full-Dimensional Problem with \texorpdfstring{$\epsilon$}{Epsilon}-Independent Domains} \label{sec:sec25}
Subsequently, we will rewrite the integrals in the weak formulation~\eqref{eq:weakdarcyeps} on~$\smash{\Omega_{\pm }^\epsilon}$ and~$\smash{\Omega_\mathrm{f}^\epsilon}$ as integrals on~$\smash{\Omega_{\pm }^0}$ and $\smash{\Gamma_a}$.
In this way, we avoid working with $\epsilon$-dependent domains and can more easily identify the dominant behavior for vanishing~$\epsilon$. 

Let $\bar{\epsilon} > 0$ be as in \Cref{lem:trafo}. 
Then, for $\epsilon\in (0, \bar{\epsilon}]$, we define the solution and test function space  
%\begin{linenomath}
\begin{align}
\Phi := \invY^\epsilon \bigl( \Phi^\epsilon \bigr)  \subset  \Bigl[ \bigtimes\nolimits_{i=\pm }H^1 (\Omega^0_i ) \Bigr] \times H^1 (\Gamma_a ). \label{eq:Phi}
\end{align}
%\end{linenomath}
As a consequence of \Cref{lem:trafo}, the space $\Phi$ does not depend on~$\epsilon$ (cf.\ \Cref{lem:traceineq}). 
In addition, we define
%\begin{linenomath}
\begin{subequations}
\begin{align}
\hat{\matr{K}}_\rmf &:= \invY^\epsilon_\rmf \matr{K}_\rmf^\epsilon  = \invY_\rmf^{\hat{\epsilon}} \matr{K}_\rmf^{\hat{\epsilon}} \in L^\infty (\Gamma_a ; \mathbb{R}^{n\times n } ) , \\
\hat{q}_\mathrm{f} &:= \invY^\epsilon_\rmf q_\rmf^\epsilon =\invY_\rmf^{\hat{\epsilon}} q_\rmf^{\hat{\epsilon}} \in L^2 (\Gamma_a ) .
\end{align}
\end{subequations}
%\end{linenomath}

Next, for $\epsilon \in (0, \bar{\epsilon} ]$, let $\smash{( \phi_+ ,\phi_- , \phi_\mathrm{f} ) \in \Phi}$ and set 
%\begin{linenomath}
\begin{align}
\bigl( \phi_+^\epsilon , \phi_-^\epsilon , \phi_\mathrm{f}^\epsilon \bigr) := \mathcal{Y}^\epsilon  (  \phi_+ , \phi_- , \phi_\rmf ) \in \Phi^\epsilon .
\end{align}
%\end{linenomath}
Further, given the unique solution $( p_+^\epsilon , p_-^\epsilon , p_\mathrm{f}^\epsilon ) \in \Phi^\epsilon$ of \cref{eq:weakdarcyeps}, we define
%\begin{linenomath}
\begin{align}
\bigl( \hat{p}_+^\epsilon , \hat{p}_-^\epsilon , \hat{p}_\mathrm{f}^\epsilon \bigr) := \invY^\epsilon \bigl( p^\epsilon_+ ,  p^\epsilon_- , p^\epsilon_\rmf \bigr) \in \Phi .
\end{align}
%\end{linenomath}
Then, with \Cref{lem:det}, we have 
%\begin{linenomath}
\begin{align}
\begin{split}
\int_{\Omega_\mathrm{f}^\epsilon } q_\mathrm{f}^\epsilon \phi_\mathrm{f}^\epsilon \,\D\lambda_n  &= \sum_{j\in J } \int_{V_{\mathrm{f},j}} \bigl[ \chi_{\mathrm{f},j}^\epsilon q_\mathrm{f}^\epsilon \phi_\mathrm{f}^\epsilon \bigr] \bigr\vert_{ \invpsi_{\mathrm{f},j}^\epsilon  (\vct{\theta }) } \mu_{\rmf , j}^\epsilon (\vct{\theta} )  \,\D\lambda_n (\vct{\theta} ) \\
&= \epsilon \bigl[ 1 + \mathcal{O} (\epsilon ) \bigr] \sum_{j\in J } \int_{V_{\rmf, j}}  \bigl[ \chi_j^a \hat{q}_\mathrm{f}  \phi_\mathrm{f} \bigr] \bigr\vert_{  \invpsi_j^a ( \vct{\theta} ) }  \mu_j (\vct{\theta}^\prime )   \,\D\theta_n \,\D\lambda_{n-1} (\vct{\theta}^\prime )  \\
& = \epsilon \bigl[ 1+ \mathcal{O} (\epsilon ) \bigr] \int_{\Gamma_a} \hat{q}_\mathrm{f} \phi_\mathrm{f} \,\D\lambda_{\Gamma_a} .
\end{split}%
\label{eq:rewrite1}%
\end{align}%
%\end{linenomath}
In the same way, by additionally using \Cref{lem:trafo}, we obtain
%\begin{linenomath}
\begin{equation}
\begin{multlined}[c][0.875\displaywidth]
\bigl[ 1+ \mathcal{O}(\epsilon ) \bigr] \int_{\Omega_\mathrm{f}^\epsilon } \matr{K}_\mathrm{f}^\epsilon \nabla p_\mathrm{f}^\epsilon \cdot \nabla \phi_\mathrm{f}^\epsilon \,\D \lambda_n \\
= \epsilon \int_{\Gamma_a} \hat{\matr{K}}_\mathrm{f}  \invR_\rmf^\epsilon \nablaGamma \hat{p}_\mathrm{f}^\epsilon  \cdot  \invR_\rmf^\epsilon  \nablaGamma \phi_\mathrm{f} \,\D\lambda_{\Gamma_a} +  \int_{\Gamma_a} \hat{\matr{K}}_\mathrm{f} \nablaN \hat{p}_\mathrm{f}^\epsilon  \cdot \invR_\rmf^\epsilon \nablaGamma \phi_\mathrm{f} \,\D\lambda_{\Gamma_a} \\
 +  \int_{\Gamma_a} \hat{\matr{K}}_\mathrm{f}  \invR_\rmf^\epsilon \nablaGamma \hat{p}_\mathrm{f}^\epsilon \cdot \nablaN \phi_\mathrm{f}  \,\D\lambda_{\Gamma_a} + \epsilon^{-1} \!\int_{\Gamma_a} \hat{\matr{K}}_\mathrm{f} \nablaN \hat{p}_\mathrm{f}^\epsilon \cdot \nablaN \phi_\mathrm{f}    \,\D\lambda_{\Gamma_a}  .
\label{eq:rewrite2}%
\end{multlined}%
\end{equation}%
%\end{linenomath}
Moreover, it is 
%\begin{linenomath}
\begin{align}
\begin{split}
\int_{\Omega_\pm^\epsilon } q_\pm^\epsilon \phi_\pm^\epsilon \,\D\lambda_n &= \int_{\Omega_{\pm , \mathrm{out}} } q_\pm^0 \phi_\pm \,\D\lambda_n + \int_{\Omega_{\pm , \mathrm{in}}^\epsilon } q_\pm^\epsilon \phi_\pm^\epsilon \,\D\lambda_n \\
& = \bigl[ 1 + \mathcal{O}(\epsilon ) \bigr] \int_{\Omega_\pm^0 } q_\pm^0 \phi_\pm \,\D\lambda_n ,
\end{split}
\end{align}
%\end{linenomath}
where we have used that $\smash{ \vct{T}_\pm^\epsilon \circ \invpsi_{\pm , j}^\epsilon = \invpsi_{\pm , j }^0  }$ for $j \in  J $ and hence, with \Cref{lem:det},
%\begin{linenomath}
\begin{align}
\begin{split}
&\int_{\Omega_{\pm , \mathrm{in}}^\epsilon } q_\mathrm{f}^\epsilon \phi_\mathrm{f}^\epsilon \,\D\lambda_n  = \sum_{j\in J } \int_{V_{\pm ,j}} \bigl[ \chi_{\pm ,j}^\epsilon q_\pm^\epsilon \phi_\pm^\epsilon \bigr]\bigr\vert_{ \invpsi_{\pm ,j}^\epsilon  (\vct{\theta }) }\, \mu_{\pm , j}^\epsilon (\vct{\theta}) \,\D\lambda_n (\vct{\theta} ) \\
&\hspace{1.5cm}= \bigl[ 1 + \mathcal{O}(\epsilon ) \bigr] \sum_{j\in  J  } \int_{V_{\pm , j}} \bigl[ \chi_{\pm ,j}^0 q_\pm^0 \phi_\pm \bigr]\bigr\vert_{ \invpsi_{\pm ,j}^0  (\vct{\theta }) } \, \mu_{\pm , j}^0 (\vct{\theta})  \,\D\lambda_n (\vct{\theta} ) \\
&\hspace{1.5cm}= \bigl[ 1 + \mathcal{O} (\epsilon ) \bigr] \int_{\Omega_{\pm , \mathrm{in}}^0} q_\pm^0 \phi_\pm \, \D\lambda_n .
\end{split}%
\label{eq:rewrite3}%
\end{align}%
%\end{linenomath}
Analogously, by additionally using \Cref{lem:trafo}, we obtain
%\begin{linenomath}
\begin{align}
\begin{split}
&\int_{\Omega^\epsilon_\pm } \matr{K}_\pm^\epsilon \nabla p_\pm^\epsilon \cdot \nabla \phi_\pm^\epsilon \,\D\lambda_n =  \int_{\Omega_{\pm , \mathrm{out}} }\! \matr{K}_\pm^0 \nabla \hat{p}_\pm^\epsilon \cdot \nabla \phi_\pm \,\D\lambda_n \\
&\hspace{3cm} + \bigl[ 1+ \mathcal{O}(\epsilon ) \bigr] \int_{\Omega_{\pm , \mathrm{in}}^0 }\! \matr{K}_\pm^0  \matr{M}_\pm^\epsilon \! \nabla \hat{p}_\pm^\epsilon \cdot  \matr{M}_\pm^\epsilon \!  \nabla \phi_\pm \,\D\lambda_n  .
\end{split}%
\label{eq:rewrite4}%
\end{align}%
%\end{linenomath}
Thus, by combining the \cref{eq:rewrite1,eq:rewrite2,eq:rewrite3,eq:rewrite4}, we find that, if $\smash{(p^\epsilon_+ , p^\epsilon_- , p^\epsilon_\rmf ) \in \Phi^\epsilon}$ solves the weak formulation~\eqref{eq:weakdarcyeps}, then $\smash{( \hat{p}^\epsilon_+ , \hat{p}^\epsilon_- , \hat{p}^\epsilon_\rmf ) \in \Phi}$ satisfies
%\begin{linenomath}
\begin{align}
\begin{split}
&\sum_{i  = \pm ,\rmf } \mathcal{A}_i^\epsilon (\hat{p}_i^\epsilon , \phi_i )  = \bigl[ 1 + \mathcal{O}(\epsilon ) \bigr] \Biggl[ \sum_{i=\pm } \bigl( q_i^0 , \phi_i \bigr)_{L^2 (\Omega_i^0 ) } + \epsilon^{\beta + 1 } \bigl( \hat{q}_\mathrm{f} , \phi_\mathrm{f} \bigr)_{L^2 (\Gamma_a ) } \Biggr] 
\end{split}%
\label{eq:rescaledweakdarcyeps}%
\end{align}%
%\end{linenomath}
for all $\smash{\phi = ( \phi_+ , \phi_- , \phi_\mathrm{f} ) \in \Phi}$ as $\epsilon \rightarrow 0$. 
The bilinear forms~$\smash{\mathcal{A}_\pm \colon \Omega_\pm^0 \times \Omega_\pm^0 \rightarrow \mathbb{R}}$ and $\smash{\mathcal{A}_\mathrm{f}^\epsilon \colon \Gamma_a \times \Gamma_a \rightarrow \mathbb{R}}$ are given by
%\begin{linenomath}
\begin{align}
\begin{split}
&\mathcal{A}_\pm^\epsilon ( \hat{p}_\pm^\epsilon , \phi_\pm ) := \bigl( \matr{K}^0_\pm \nabla \hat{p}^\epsilon_\pm , \nabla \phi_\pm \bigr)_{\vct{L}^2 (\Omega_{\pm , \mathrm{out}} ) } \\
&\hspace{5cm} + \bigl( \matr{K}^0_\pm  \matr{M}^\epsilon_\pm \! \nabla \hat{p}^\epsilon_\pm , \matr{M}^\epsilon_\pm \! \nabla \phi_\pm \bigr)_{\vct{L}^2 (\Omega_{\pm , \mathrm{in}}^0  ) },
\end{split} \\
\begin{split}
&\mathcal{A}_\mathrm{f}^\epsilon (\hat{p}^\epsilon_\rmf , \phi_\mathrm{f} ) := \epsilon^{\alpha + 1} \bigl( \hat{\matr{K}}_\rmf \bigl[ \invR_\rmf^\epsilon \nablaGamma \hat{p}^\epsilon_\rmf + \tfrac{1}{\epsilon }\nablaN \hat{p}_\rmf^\epsilon \bigr] , \bigl[ \invR_\rmf^\epsilon \nablaGamma \phi_\rmf + \tfrac{1}{\epsilon} \nablaN \phi_\rmf \bigr] \bigr)_{\!\vct{L}^2 (\Gamma_a ) } \\
&\hspace{0.85cm} = \epsilon^{\alpha + 1} \bigl( \hat{\matr{K}}_\mathrm{f} \invR^\epsilon_\rmf \nablaGamma \hat{p}^\epsilon_\mathrm{f} ,  \invR^\epsilon_\rmf \nablaGamma \phi_\mathrm{f} \bigr)_{\vct{L}^2 (\Gamma_a  ) }  + \epsilon^{\alpha } \bigl(  \hat{\matr{K}}_\mathrm{f}  \nablaN \hat{p}^\epsilon_\mathrm{f} ,  \invR^\epsilon_\rmf \nablaGamma \phi_\mathrm{f} \bigr)_{\vct{L}^2 (\Gamma_a  ) }  \\
&\hspace{1.45cm} + \epsilon^{\alpha} \bigl(  \hat{\matr{K}}_\mathrm{f}  \invR^\epsilon_\rmf \nablaGamma  \hat{p}^\epsilon_\mathrm{f}  , \nablaN \phi_\mathrm{f}  \bigr)_{\vct{L}^2 (\Gamma_a  ) }  + \epsilon^{\alpha -1} \bigl(\hat{\matr{K}}_\mathrm{f} \nablaN \hat{p}^\epsilon_\mathrm{f} , \nablaN \phi_\mathrm{f}  \bigr)_{\vct{L}^2 (\Gamma_a  ) } .
\end{split}
\end{align}
%\end{linenomath}

\section{A-Priori Estimates and Weak Convergence} \label{sec:sec3}
In this section, we obtain a-priori estimates for the solution~$( \hat{p}_+^\epsilon  , \hat{p}_-^\epsilon , \hat{p}_\rmf^\epsilon ) \in \Phi $ of the transformed weak formulation~\eqref{eq:rescaledweakdarcyeps} and, consequently, can identify a weakly convergent subsequence as~$\epsilon \rightarrow 0$. 
The main results are developed in \Cref{sec:sec33}.
They build on trace inequalities from \Cref{sec:sec31} and Poincaré-type inequalities from \Cref{sec:sec32}.

First, we introduce useful functions spaces on~$\Gamma$ and~$\Gamma_a$, as well as averaging operators on~$\smash{\Gamma_a}$.
Given a $\smash{\lambda_\Gamma}$-measurable, non-negative weight function $w \colon \Gamma \rightarrow \mathbb{R}$, we define the weighted Lebesgue space~$L^2_w (\Gamma )$ as the $L^2$-space on~$\Gamma$ with measure~$w\lambda_\Gamma$.
Further, we define the weighted Sobolev space~$H^1_a (\Gamma ) $ as the completion of 
%\begin{linenomath}
\begin{align}
\bigl\{ f\in \mathcal{C}^{0,1}(\Gamma )   \ \Big\vert\  \norm{f}_{H^1_a (\Gamma ) } < \infty   \bigr\}
\end{align}
%\end{linenomath}
with respect to the norm $\norm{f}_{H^1_a (\Gamma ) }^2 :=  \norm{f}_{L^2_a (\Gamma ) }^2 + \norm{\nablaGamma f}^2_{\vct{L}^2_a (\Gamma ) } $.
Besides, we define the space~$\smash{H^1_{\vct{N}} (\Gamma_a ) \subset L^2 (\Gamma_a )}$ as the closure of the space 
%\begin{linenomath}
\begin{align}
\bigl\{ f \in \mathcal{C}^{0,1} ( \Gamma_a )  \ \Big\vert\  \norm{f}_{H^1_\vct{N} (\Gamma_a) } < \infty \bigr\} 
\end{align}
%\end{linenomath}
with respect to the norm $\norm{f}_{H^1_\vct{N} (\Gamma_a) }^2 :=  \norm{f}_{L^2 (\Gamma_a ) }^2 + \norm{\nablaN f}_{\vct{L}^2 (\Gamma_a ) }^2   $.
Moreover, we introduce the averaging operators
%\begin{linenomath}
\begin{subequations}
\begin{alignat}{2}
\mathfrak{A}_\Gamma &\colon L^2 (\Gamma_a ) \rightarrow L^2_a (\Gamma ), \quad &&( \mathfrak{A}_\Gamma f ) ( \vct{p} ) :=  \frac{1}{a( \vct{p}  )} \int_{-a_-( \vct{p} ) }^{a_+( \vct{p}  )} f (\vct{p} , \theta_n ) \,\D \theta_n , \\
\mathfrak{A}_\mathrm{f} &\colon L^2 (\Gamma_a ) \rightarrow \mathbb{R}, \quad &&\mathfrak{A}_\mathrm{f} f := \frac{1}{\int_\Gamma a  \, \D\lambda_\Gamma } \int_{\Gamma_a }  f \,\D\lambda_{\Gamma_a} .
\end{alignat}
\end{subequations}
%\end{linenomath}

\subsection{Trace Inequalities} \label{sec:sec31}
We begin by introducing a trace operator~$\smash{\mathfrak{T}_\pm}$ on~$\smash{H^1_\vct{N} (\Gamma_a ) }$ for the lateral boundaries of~$\smash{\overline{\Gamma}_a^\smallperp}$.
\begin{lemma} \label{lem:Tpm}
There exists a uniquely defined bounded linear operator 
%\begin{linenomath}
\begin{align}
\mathfrak{T}_\pm \colon H^1_\vct{N} (\Gamma_a ) \rightarrow L^2_a (\Gamma ) 
\end{align}
%\end{linenomath}
such that, for all $\smash{f \in \mathcal{C}^{0,1} (\overline{\Gamma}_a^\smallperp )} $, we have
%\begin{linenomath}
\begin{align}
\bigl( \mathfrak{T}_\pm f \bigr) (\vct{p}  ) &= f \bigl( \vct{p} , \pm a_\pm (\vct{p} ) \bigr) .
\end{align}
%\end{linenomath}
\end{lemma}
\begin{proof}
W.l.o.g., we consider $\smash{\mathfrak{T}_+}$. The operator $\smash{\mathfrak{T}_-}$ can be treated analogously. 

Let $\smash{ f \in \mathcal{C}^{0,1} (\overline{\Gamma}_a^\smallperp ) }$.
Then, for all $\smash{( \vct{p} , \theta_n ) \in \Gamma_a }$, we have 
%\begin{linenomath}
\begin{align*}
f^2 \bigl( \vct{p} , a_+(\vct{p} ) \bigr) &= f^2 (\vct{p} , \theta_n  ) + 2 \int_{\theta_n}^{a_+ (\vct{p} ) }\! f(\vct{p} , \bar{\theta}_n ) \, \partial_{\theta_n } f(\vct{p} , \bar{\theta_n} ) \,\D \bar{\theta}_n  .
\end{align*}
%\end{linenomath}
An integration over~$\smash{\Gamma_a }$ yields 
%\begin{linenomath}
\begin{align*}
\int_\Gamma a f^2 \bigl( \cdot , a_+ (\cdot ) \bigr) \,\D \lambda_\Gamma &\le \int_{\Gamma_a} f^2 \, \D \lambda_{\Gamma_a }   + 2 \int_{\Gamma_a } a \abssize{ }{f} \abssize{ }{ \partial_{\theta_n } f } \,\D\lambda_{\Gamma_a } .
\end{align*}
%\end{linenomath}
By applying Hölder's inequality, we obtain
%\begin{linenomath}
\begin{align*}
\norm { \mathfrak{T}_+ f} ^2_{L^2_a (\Gamma ) }  &\le \norm{f}^2_{L^2 (\Gamma_a ) } + 2 \norm{a}_{L^\infty (\Gamma ) } \norm{f}_{L^2 (\Gamma_a ) } \norm{\nablaN f}_{\vct{L}^2 (\Gamma_a ) } \lesssim \norm{f}^2_{H^1_\vct{N} (\Gamma_a  ) } .
\end{align*}
%\end{linenomath}
The result now follows from the fact that $\smash{\mathcal{C}^{0,1} (\overline{ \Gamma}_a^\smallperp ) }$ is dense in $\smash{H^1_\vct{N}(\Gamma_a ) }$.
\end{proof}

Besides, we obtain the following characterization of the space~$\Phi$.
\begin{lemma} \label{lem:traceineq} 
We have 
%\begin{linenomath}
\begin{equation}
\begin{multlined}[c][0.875\displaywidth]
\Phi = \Bigl\{ ( \phi_+ , \phi_- ,\phi_\rmf ) \in   \Bigl[ \bigtimes\nolimits_{i = \pm } H^1_{0, \rho_{i, \mathrm{D}}^0} (\Omega_i^0 ) \Bigr] \times  H^1_{0, \rho_{a , \mathrm{D}}} (\Gamma_a ) \ \Big\vert\ \\
\phi_+\bigr\vert_{\Gamma_0^0} = \phi_-\bigr\vert_{\Gamma_0^0} , \ \mathfrak{T}_\pm \phi_\rmf = \phi_\pm \big\vert_\Gamma  \Bigr\} . \label{eq:Phi2}
\end{multlined}
\end{equation}
%\end{linenomath}
In particular, for $\smash{( \phi_+ , \phi_- ,\phi_\rmf ) \in \Phi}$, it is 
%\begin{linenomath}
\begin{align}
\normsize{\big}{\mathfrak{T}_\pm \phi_\rmf }_{L^2_a (\Gamma )  }   &\lesssim \normsize{}{ \phi_\pm }_{H^1 (\Omega_\pm^0 )  } . \label{eq:traceineq}
\end{align}
%\end{linenomath}
\end{lemma}
\begin{proof}
Using that $\smash{\mathcal{C}^{0,1} (\overline{\Omega_\pm^0 } ) \subset H^1 (\Omega_\pm^0 )}$ and $\smash{\mathcal{C}^{0,1} (\overline{\Gamma}_a^\smallperp ) \subset H^1_\vct{N} (\Gamma_a ) }$ are dense, we find 
%\begin{linenomath}
\begin{align*}
\phi_\pm \bigr\vert_{\gamma} = \mathcal{P}^{\partial G} \Bigl( \bigl[ \mathcal{Y}^\epsilon_\pm \phi_\pm \bigr] \bigr\vert_{\gamma_\pm^\epsilon } \Bigr), \quad\enspace \mathfrak{T}_\pm \phi_\rmf = \mathcal{P}^{\partial G} \Bigl( \bigl[ \mathcal{Y}^\epsilon_\rmf \phi_\rmf \bigr] \bigr\vert_{\Gamma_\pm^\epsilon } \Bigr)
\end{align*}
%\end{linenomath}
almost everywhere for any~$\smash{\epsilon \in ( 0, \bar{\epsilon}]}$ for all $\phi_\pm \in H^1 (\Omega_\pm^0) $ and $\phi_\rmf \in H^1 (\Gamma_a ) $.
Thus, it is easy to see that
%\begin{linenomath}
\begin{align*}
\Phi = \invY^\epsilon \bigl( \Phi^\epsilon \bigr) \subset \Phi^\prime , \\
\mathcal{Y}^\epsilon \bigl( \Phi^\prime \bigr) \subset \Phi^\epsilon = \mathcal{Y}^\epsilon ( \Phi ) \enspace\Rightarrow\enspace \Phi^\prime \subset \Phi,
\end{align*}
%\end{linenomath}
where $\smash{\Phi^\prime}$ denotes the space on the right-hand side of \cref{eq:Phi2}.
Besides, \cref{eq:traceineq} is a consequence of the trace inequality in~$\smash{\Omega_\pm^0}$. 
\end{proof}

Further, it is easy to see that the following lemma holds, which introduces a trace operator on the weighted Sobolev space~$H^1_a (\Gamma ) $. 
\begin{lemma} \label{lem:traceH1a}
Let ${\mathfrak{T}_{\overline{\Gamma}_a^\smallpar } \colon H^1 (\overline{\Gamma}_a^\smallpar ) \rightarrow L^2 (\partial \overline{\Gamma}_a^\smallpar ) }$ denote the trace operator on~$\smash{\overline{\Gamma}_a^\smallpar }$ from \Cref{lem:manifoldtrace}.
Further, we introduce the constant extension operator
%\begin{linenomath}
\begin{align}
\mathfrak{E}_a  \colon H^1_a ( \Gamma ) \rightarrow H^1 (\Gamma_a ) , \quad \bigl( \mathfrak{E}_a f \bigr) (\vct{p} , \theta_n ) := f ( \vct{p} ) .
\end{align}
%\end{linenomath}
Then, the trace operator $\smash{\mathfrak{T}_\smallpar^a \colon H^1_a (\Gamma ) \rightarrow L^2_a (\partial \overline{\Gamma } )}$ defined by
%\begin{linenomath}
\begin{align}
 \bigl( \mathfrak{T}_\smallpar^a f \bigr) (\vct{p}) := \begin{cases} 
  0 & \text{if } a(\vct{p}) = 0 , \\
\frac{1}{a(\vct{p})} \int_{-a_-(\vct{p})}^{a_+(\vct{p})} \mathfrak{T}_{\overline{\Gamma}^\smallpar_a } \bigl( \mathfrak{E}_a f \bigr) (\vct{p} , \theta_n ) \,\D \theta_n &\text{if } a(\vct{p}) \not= 0
\end{cases}  
\end{align}
%\end{linenomath}
is bounded and satisfies 
%\begin{linenomath}
\begin{align}
\normsize{\big }{ \mathfrak{T}_\smallpar^a f - f \bigr\vert_{\partial \overline{\Gamma}}}_{L^2_a ( \partial \overline{\Gamma} ) } &= 0 \qquad \text{for all } f \in \mathcal{C}^0 (\overline{\Gamma } ) . \label{eq:l2aeq}
\end{align}
%\end{linenomath}
\end{lemma}

\subsection{Poincaré-Type Inequalities} \label{sec:sec32}
We obtain two Poincaré-type inequalities for functions in~$\smash{H^1_\vct{N}(\Gamma_a ) }$.
\begin{lemma} \label{lem:apriori_l2_0} 
Let $i\in\{ +, -\}$ and $f  \in H^1_\vct{N} (\Gamma_a )$. 
Then, we have 
%\begin{linenomath}
\begin{align}
\normsize{\big}{ \mathfrak{T}_i f - \mathfrak{A}_\Gamma f  }_{L^2_a(\Gamma )} \lesssim \norm{\nablaN f}_{\vct{L}^2 (\Gamma_a )} . \label{eq:apriori_l2_0}
\end{align}
%\end{linenomath}
\end{lemma}
\begin{proof} 
We prove the inequality~\eqref{eq:apriori_l2_0} for $i=+$ and $\smash{f \in \mathcal{C}^{0,1} (\overline{\Gamma}_a^\smallperp )}$.
The case $i=-$ is analogous. The general case follows from a density argument. 
We now have 
%\begin{linenomath}
\begin{multline*}
\normsize{\big}{ \mathfrak{T}_+ f - \mathfrak{A}_\Gamma f  }_{L^2_a(\Gamma )}^2  
= \int_\Gamma a(\vct{p}) \left[ f\bigl( \vct{p}  , a_+(\vct{p}  )\bigr) - \frac{1}{a(\vct{p}  )} \int_{a_-(\vct{p}  ) }^{a_+ (\vct{p} ) } f ( \vct{p} , \theta_n ) \,\D \theta_n \right]^2\! \D\lambda_\Gamma (\vct{p} )  \\
 = \int_\Gamma  \frac{1}{a(\vct{p}  )} \left[ \int_{-a_-(\vct{p} )}^{a_+(\vct{p}  )} \int_{\theta_n}^{a_+(\vct{p}  )} \partial_{\theta_n} f(\vct{p} , \tau_n ) \,\D \tau_n \, \D \theta_n \right]^2 \D\lambda_\Gamma ( \vct{p})   \lesssim \norm{\nablaN f }^2_{L^2 (\Gamma_a ) } .  \tag*{\qedhere}
\end{multline*}
%\end{linenomath}
\end{proof}

\begin{lemma} \label{lem:apriori_l2_1}
Let $i\in\{ +, -\}$ and $f \in H^1_\vct{N} (\Gamma_a )$.
Then, we have
%\begin{linenomath}
\begin{align}
\norm{f}_{L^2 (\Gamma_a )} &\lesssim  \norm{\nablaN f}_{\vct{L}^2 (\Gamma_a )} + \norm{\mathfrak{T}_i f}_{L^2_a (\Gamma  )} . \label{eq:apriori_l2_1}
\end{align} 
%\end{linenomath}
\end{lemma}
\begin{proof} 
Subsequently, we prove the inequality~\eqref{eq:apriori_l2_1} for $i=+$ and $\smash{f \in \mathcal{C}^{0,1} (\overline{\Gamma}_a^\smallperp )}$.
Then, the desired inequality is obtained from a density argument.
The case $i=-$ follows by analogy.
Now, let $\smash{ (\vct{p}  , \theta_n ) \in \Gamma_a }$.
Then, we have
%\begin{linenomath}
\begin{align*}
f\bigl( \vct{p} ,   a_+(\vct{p} )\bigr) - f( \vct{p} , \theta_n )  &= \int_{\theta_n}^{ a_+(\vct{p} )} \partial_{\theta_n}  f (\vct{p}  , \tau_n ) \,\D\tau_n 
\end{align*}
%\end{linenomath}
and hence, by using Hölder's inequality,
%\begin{linenomath}
\begin{align*}
\begin{split}
&\abssize{\big}{ f\bigl( \vct{p} , a_+(\vct{p} )\bigr) - f(\vct{p} , \theta_n ) }^2 \le   a(\vct{p} ) \int_{- a_-(\vct{p} )}^{ a_+(\vct{p})} \abssize{\big}{\partial_{\theta_n} f(\vct{p} , \tau_n ) }^2 \D\tau_n .
\end{split}
\end{align*}
%\end{linenomath}
Consequently, we obtain
%\begin{linenomath}
\begin{align*}
\int_{- a_-(\vct{p} )}^{ a_+(\vct{p} )} \abssize{\big}{ f\bigl( \vct{p},  a_+(\vct{p} )\bigr) - f(\vct{p} , \theta_n ) }^2 \,\D \theta_n  \le   a^2(\vct{p} ) \int_{- a_-(\vct{p} )}^{ a_+(\vct{p} )} \abssize{\big}{\partial_{\theta_n} f (\vct{p} , \theta_n )  }^2 \,\D \theta_n .
\end{align*}
%\end{linenomath}
An additional integration on~$\Gamma$ yields 
%\begin{linenomath}
\begin{align}
\normsize{\big}{ f \bigl( \cdot ,  a_+(\cdot )\bigr) -  f }_{L^2 (\Gamma_a )}  \lesssim  \normsize{\big}{\nablaN f }_{\vct{L}^2 ( \Gamma_a   ) } . \label{eq:l2_apriori_proof1}
\end{align}
%\end{linenomath}
Further, we have $\smash{\norm{f ( \cdot ,  a_+(\cdot )) }_{L^2 (\Gamma_a )} = \norm{\mathfrak{T}_+ f}_{L^2_a (\Gamma ) }}$ so that the result follows by applying the reverse triangle inequality in \cref{eq:l2_apriori_proof1}.
\end{proof}

We can now combine Poincaré's inequality and the \Cref{lem:traceineq,lem:apriori_l2_1} to obtain the following estimate for function triples~$\smash{(\phi_+ , \phi_- , \phi_\rmf ) \in \Phi}$, which fits the setting of the coupled Darcy problem in \cref{eq:rescaledweakdarcyeps}.
\begin{lemma} \label{lem:apriori_l2_2} 
Let $\smash{ (\phi_+ , \phi_- , \phi_\rmf ) \in \Phi }$. 
\begin{enumerate}[label=(\roman*),wide,topsep=0pt]
\item There exists an $\smash{\epsilon^\ast > 0}$ such that, for all $\smash{\epsilon \in ( 0, \epsilon^\ast ]}$ and $\nu \ge 0$, we have 
%\begin{linenomath}
\begin{align}
\begin{split}
&\sum_{i=\pm } \norm{\phi_i}_{L^2 (\Omega_i^0 ) } + \epsilon^\nu \norm{\phi_\rmf}_{L^2 (\Gamma_a ) } \\
&\hspace{1.5cm} \lesssim  \sum_{i=\pm} \norm{\nabla \phi_i}_{\vct{L}^2 (\Omega_i^0 )} + \epsilon^{\frac{1}{2}} \norm{\nablaGamma \phi_\rmf}_{\vct{L}^2 ( \Gamma_a  ) } + \epsilon^{-\frac{1}{2}}\norm{\nablaN \phi_\rmf}_{\vct{L}^2 (\Gamma_a ) } . 
\end{split}%
\label{eq:apriori_l2_2_A}%
\end{align}%
%\end{linenomath}
\item Let $\nu \ge 0$ and $\epsilon \in [ 0, 1 ]$. Given additionally the assumption~\eqref{asm:rhoD}, we have
%\begin{linenomath}
\begin{align}
\sum_{i=\pm}\norm{\phi_i}_{L^2(\Omega_i^0 )} + \epsilon^\nu \norm{\phi_\rmf}_{L^2(\Gamma_a)} &\lesssim  \sum_{i=\pm} \norm{\nabla \phi_i}_{\vct{L}^2 (\Omega_i^0 )} +  \epsilon^\nu\norm{\nablaN \phi_\rmf}_{\vct{L}^2(\Gamma_a  ) } .
\end{align}
%\end{linenomath}
\end{enumerate}
\end{lemma}
\begin{proof}
\begin{enumerate}[label=(\roman*),wide,topsep=0pt]
\item 
Let $( \phi_+ , \phi_- , \phi_\rmf ) \in \Phi$ and, for $\epsilon \in ( 0, \bar{\epsilon} ]$, define $\smash{\phi^\epsilon \in H^1_{0 , \rho_\mathrm{D}^\epsilon } (\Omega )}$ by 
%\begin{linenomath}
\begin{align*}
\phi^\epsilon ( \vct{x} ) := \begin{cases}
\bigl[ \mathcal{Y}_i^\epsilon \phi_i \bigr] ( \vct{x } ) & \text{if } \vct{x} \in \Omega^\epsilon_i , \ i \in \{ + , - , \rmf \} .
\end{cases}
\end{align*}
%\end{linenomath}
Then, with \Cref{lem:trafo} and Poincaré's inequality, we have 
%\begin{linenomath}
\begin{align*}
\sum_{i=\pm} \norm{\phi_i}^2_{L^2 (\Omega_i^0 ) } = \sum_{i=\pm } \bigl[ 1 + \mathcal{O}(\epsilon ) \bigr]    \norm{\mathcal{Y}_i^\epsilon \phi_i}^2_{L^2 (\Omega_i^\epsilon ) }  \lesssim \norm{\phi^\epsilon}_{L^2 (\Omega ) }^2 \lesssim \norm{\nabla \phi^\epsilon}_{\vct{L}^2 (\Omega ) }^2
\end{align*}
%\end{linenomath}
if $\epsilon > 0$ is sufficiently small.
Moreover, \Cref{lem:trafo} yields
%\begin{linenomath}
\begin{multline*}
\norm{\nabla \phi^\epsilon}_{\vct{L}^2 (\Omega ) }^2 = \sum_{i\in \{ + , - ,\rmf \} } \norm{\nabla \bigl[ \mathcal{Y}_i^\epsilon \phi_i \bigr] }^2_{\vct{L}^2 (\Omega_i^\epsilon ) } \\ = \bigl[ 1 + \mathcal{O} ( \epsilon ) \bigr] \biggl[ \sum_{i=\pm } \norm{\nabla \phi_i}_{\vct{L}^2 ( \Omega_i^0 ) }^2 + \epsilon \norm{\nablaGamma \phi_\rmf }^2_{\vct{L}^2 (\Gamma_a ) } + \epsilon^{-1 } \norm{\nablaN \phi_\rmf}_{\vct{L}^2 (\Gamma_a )}^2 \biggr] .
\end{multline*}
%\end{linenomath}
By using Poincaré's inequality and the \Cref{lem:traceineq,lem:apriori_l2_1}, we obtain
%\begin{linenomath}
\begin{align*}
\norm{\phi_\rmf}_{L^2 (\Gamma_a ) } \lesssim \norm{\nablaN \phi_\rmf}_{\vct{L}^2 (\Gamma_a ) } + \norm{\mathfrak{T}_+ \phi_\rmf}_{L^2_a (\Gamma ) } \lesssim \norm{\nablaN \phi_\rmf}_{\vct{L}^2 (\Gamma_a ) } + \norm{\nabla \phi_+}_{\vct{L}^2 (\Omega^0_+ ) } .
\end{align*}
%\end{linenomath}
\item Follows directly from Poincaré's inequality and the \Cref{lem:traceineq,lem:apriori_l2_1}. \qedhere
\end{enumerate}
\end{proof}

\subsection{Results} \label{sec:sec33}
Using \Cref{lem:apriori_l2_2}, we can obtain the following a-priori estimates for the solution~$\smash{(\hat{p}_+^\epsilon , \hat{p}_-^\epsilon , \hat{p}_\rmf^\epsilon ) \in \Phi}$ of the transformed Darcy problem~\eqref{eq:weakdarcyeps}.
\begin{proposition} \label{prop:apriori_l2_3} 
Let $\beta \ge -1$. Besides, let either $\alpha \le 0$ or, given the assumption~\eqref{asm:rhoD}, let $\smash{ 2\beta \ge\alpha - 3}$. 
Further, let $2\nu \ge \max \{ 0 , \alpha - 1 \}$. 
Then, there exists~$\epsilon^\ast > 0$ such that, for all  $\epsilon \in (0,\epsilon^\ast]$, the solution $(\hat{p}^\epsilon_+ , \hat{p}^\epsilon_- , \hat{p}^\epsilon_\mathrm{f}) \in \Phi $ of the transformed Darcy problem~\eqref{eq:rescaledweakdarcyeps} satisfies 
%\begin{linenomath}
\begin{subequations}
\begin{align}
\sum\limits_{i=\pm}  \norm{\nabla \hat{p}_i^\epsilon}^2_{\vct{L}^2(\Omega_i^0 )}  + \epsilon^{\alpha + 1 }\norm{\nablaGamma \hat{p}_\mathrm{f}^\epsilon}^2_{\vct{L}^2(\Gamma_a )} + \epsilon^{\alpha - 1} \norm{\nablaN \hat{p}_\mathrm{f}^\epsilon}^2_{\vct{L}^2(\Gamma_a )} &\lesssim 1 , \label{eq:apriori_l2_3a} 
 \\
\sum_{i=\pm}  \norm{\hat{p}_i^\epsilon}_{L^2(\Omega^0_i )} + \epsilon^\nu \norm{\hat{p}^\epsilon_\mathrm{f}}_{L^2(\Gamma_a )} &\lesssim 1 .  \label{eq:apriori_l2_3b}
\end{align}
\end{subequations}
%\end{linenomath}
\end{proposition}
\begin{proof}
We use the solution~$\smash{ ( \hat{p}^\epsilon_+ , \hat{p}^\epsilon_- , \hat{p}^\epsilon_\mathrm{f} ) \in \Phi }$ as a test function in the transformed weak formulation~\cref{eq:rescaledweakdarcyeps}.
The uniform ellipticity of the hydraulic conductivity~$\smash{\matr{K}_\pm^0}$ yields 
%\begin{linenomath}
\begin{multline*}
\mathcal{A}_\pm^\epsilon \bigl( \hat{p}^\epsilon_\pm , \hat{p}^\epsilon_\pm \bigr) \gtrsim \norm{\nabla \hat{p}_\pm^\epsilon }^2_{\vct{L}^2 (\Omega_{\pm , \mathrm{out}}^0) } + \norm{\matr{M}^\epsilon_\pm \nabla \hat{p}_\pm^\epsilon }^2_{\vct{L}^2 (\Omega^0_{\pm , \mathrm{in}})} 
= \bigl[ 1  + \mathcal{O}(\epsilon ) \bigr] \normsize{\big }{\nabla \hat{p}^\epsilon_\pm }^2_{\vct{L}^2 (\Omega_\pm^0  ) } .
\end{multline*}
%\end{linenomath}
Here, we have used that, as a consequence of \Cref{lem:det} and \Cref{lem:trafo}~\ref{item:trafo4}, it is 
%\begin{linenomath}
\begin{align*}
&\norm{\matr{M}^\epsilon_\pm \nabla \hat{p}_\pm^\epsilon }^2_{\vct{L}^2 (\Omega^0_{\pm , \mathrm{in}})} = \sum_{j\in J} \int_{V_{\pm , j}} \bigl[ \chi_{\pm , j}^0 \abssize{\big}{ \matr{M}^\epsilon_\pm \nabla \hat{p}_\pm^\epsilon }^2 \bigr] \Big\vert_{\invpsi_{\pm , j}^0 (\vct{\theta })} \mu^0_{\pm , j} (\vct{\theta}) \,\D\lambda_n (\vct{\theta} ) \\
&\hspace{1.5cm} = \bigl[ 1+ \mathcal{O}(\epsilon ) \bigr] \sum_{j\in J} \int_{V_{\pm ,j}} \bigl[ \chi_{\pm , j}^\epsilon \abssize{\big}{ \nabla ( \mathcal{Y}_\pm^\epsilon \hat{p}_\pm^\epsilon )}^2 \bigr] \Big\vert_{\invpsi_{\pm , j}^\epsilon (\vct{\theta})} \mu_{\pm ,j}^\epsilon (\vct{\theta})\, \D \lambda_n (\vct{\theta}) \\
&\hspace{1.5cm} = \bigl[ 1+ \mathcal{O}(\epsilon )\bigr] \normsize{\big}{\nabla ( \mathcal{Y}_\pm^\epsilon \hat{p}_\pm^\epsilon )}^2_{\vct{L}^2 (\Omega_{\pm , \mathrm{in}}^\epsilon )} = \bigl[ 1+ \mathcal{O}(\epsilon ) \bigr] \normsize{\big}{\nabla \hat{p}_\pm^\epsilon }^2_{\vct{L}^2 (\Omega_{\pm ,\mathrm{in}}^0 )}.
\end{align*}
%\end{linenomath}
Besides, by using \Cref{lem:Roperators} and the uniform ellipticity of~$\smash{\hat{\matr{K}}_\rmf}$, we obtain
%\begin{linenomath}
\begin{align*}
\mathcal{A}_\rmf^\epsilon \bigl( \hat{p}_\rmf^\epsilon , \hat{p}_\rmf^\epsilon \bigr) &\gtrsim \epsilon^{\alpha +1 } \normsize{\big}{ \invR_\rmf^\epsilon \nablaGamma \hat{p}_\rmf^\epsilon + \epsilon^{-1} \nablaN \hat{p}_\rmf^\epsilon }^2_{\vct{L}^2 (\Gamma_a ) } \\
&= \epsilon^{\alpha + 1 } \bigl[ 1+ \mathcal{O}(\epsilon ) \bigr] \normsize{\big }{\nablaGamma \hat{p}_\rmf^\epsilon }^2_{\vct{L}^2 (\Gamma_a  ) } + \epsilon^{\alpha - 1 } \normsize{\big }{\nablaN \hat{p}_\rmf^\epsilon }^2_{\vct{L}^2 (\Gamma_a ) }  .
\end{align*}
%\end{linenomath}
By applying Hölder's inequality on the right-hand side of \cref{eq:rescaledweakdarcyeps}, we find
%\begin{linenomath}
\begin{align}
\begin{split}
&\sum_{i=\pm }  \normsize{\big }{\nabla \hat{p}^\epsilon_i }^2_{\vct{L}^2 (\Omega_i^0 ) }  + \epsilon^{\alpha + 1 }  \normsize{\big }{\nablaGamma \hat{p}_\rmf^\epsilon }^2_{\vct{L}^2 (\Gamma_a  ) } + \epsilon^{\alpha - 1 } \normsize{\big }{\nablaN \hat{p}_\rmf^\epsilon }^2_{\vct{L}^2 (\Gamma_a  ) } \\ 
&\hspace{5.5cm}\lesssim \sum_{i=\pm}\normsize{\big}{\hat{p}_i^\epsilon}_{L^2 (\Omega_\pm^0 )} + \epsilon^{\beta + 1 } \normsize{\big }{\hat{p}_\rmf^\epsilon}_{L^2 (\Gamma_a ) } 
\end{split}%
\label{eq:apriori_l2_3_proof1}
\end{align}
%\end{linenomath}
if $\epsilon > 0$ is sufficiently small.
Thus, the inequality~\eqref{eq:apriori_l2_3a} follows after applying \Cref{lem:apriori_l2_2} on the right-hand side of \cref{eq:apriori_l2_3_proof1}. 
Then, the inequality in \cref{eq:apriori_l2_3b} follows from \cref{eq:apriori_l2_3a} and \Cref{lem:apriori_l2_2}.
\end{proof}

As a consequence of \Cref{prop:apriori_l2_3}, the solution families~$\smash{\{ \hat{p}_i^\epsilon \}_{\epsilon \in ( 0 , \hat{\epsilon}]}}$, $i \in \{ +, - ,\rmf \}$, have weakly convergent subsequences in the following sense as $\smash{\epsilon \rightarrow 0}$.
\begin{proposition}  \label{prop:convergence}
Let $\beta \ge -1$. Besides, let either $\alpha \le 0$ or, given the assumption~\eqref{asm:rhoD}, let $\smash{ 2\beta \ge\alpha - 3}$. 
Then, there exists a sequence~$\{ \epsilon_k\}_{k\in\mathbb{N}} \subset (0, \hat{\epsilon}]$ with $\epsilon_k \searrow 0$ as $k \rightarrow \infty$ such that 
%\begin{linenomath}
\begin{subequations}
\begin{alignat}{3}
\hat{p}_\pm^{\epsilon_k} &\rightharpoonup \hat{p}^\ast_\pm \quad\enspace &&\text{in } H^1(\Omega_\pm^0 )  , \label{eq:conv_A} \\
\hat{p}_\pm^{\epsilon_k} &\rightarrow \hat{p}^\ast_\pm \quad\enspace &&\text{in } L^2(\Omega_\pm^0 ) ,  \label{eq:conv_Astrong} \\
\hat{p}_\mathrm{f}^{\epsilon_k} &\rightharpoonup \hat{p}^\ast_\mathrm{f} \quad\enspace &&\text{in } H^1(\Gamma_a ) \quad\enspace &&\text{if } \alpha \le -1, \label{eq:conv_B} \\
\hat{p}_\mathrm{f}^{\epsilon_k} &\rightharpoonup \hat{p}^\ast_\mathrm{f} \quad\enspace &&\text{in } H^1_{\vct{N}}(\Gamma_a ) \quad\enspace &&\text{if } \alpha \le 1, \label{eq:conv_C}  \\
\mathfrak{A}_\Gamma \hat{p}_\mathrm{f}^{\epsilon_k} &\rightharpoonup \mathfrak{A}_\Gamma \hat{p}^\ast_\mathrm{f} \quad\enspace &&\text{in } L^2_a(\Gamma ) \quad\enspace &&\text{if } \alpha \le 1.  \label{eq:conv_D}
\end{alignat}
\end{subequations}
%\end{linenomath}
In particular, we have $\smash{( \hat{p}_+^\ast , \hat{p}_-^\ast , \hat{p}^\ast_\rmf ) \in \Phi}$ if $\alpha \le -1 $ and $\smash{( \hat{p}_+^\ast , \hat{p}_-^\ast , \hat{p}^\ast_\rmf ) \in \Phi^\ast }$ if $\alpha \le 1$, where $\smash{\Phi^\ast}$ denotes the closure of~$\Phi$ in $\smash{H^1 (\Omega_+^0 ) \times H^1 (\Omega_-^0 ) \times H^1_\vct{N} (\Gamma_a )}$.
\end{proposition}
\begin{proof}
The weak convergence statements~\eqref{eq:conv_A}, \eqref{eq:conv_Astrong}, \eqref{eq:conv_B}, and~\eqref{eq:conv_C} are a direct consequence of the estimates in~\Cref{prop:apriori_l2_3} and the Rellich-Kondrachov theorem.
Further, the weak convergence~\eqref{eq:conv_D} follows from \Cref{prop:apriori_l2_3} and 
%\begin{linenomath}
\begin{align*}
\normsize{\big}{\mathfrak{A}_\Gamma \hat{p}^\epsilon_\mathrm{f}}_{L^2_a(\Gamma)}^2 &= \int_\Gamma \frac{1}{a(\vct{p}  )} \left[ \int_{-a_-(\vct{p} ) }^{a_+(\vct{p}  )} \hat{p}^\epsilon_\rmf ( \vct{p} , \theta_n ) \,\D \theta_n \right]^2 \D\lambda_\Gamma (\vct{p})  
\le \normsize{\big}{\hat{p}_\mathrm{f}^\epsilon}_{L^2 (\Gamma_a)}^2 .
\end{align*}
%\end{linenomath}
Besides, we have $\smash{( \hat{p}_+^\ast , \hat{p}_-^\ast , \hat{p}^\ast_\rmf ) \in \Phi}$  if $\alpha \le -1 $ since $\Phi$ is convex and closed in $\smash{H^1 (\Omega_+^0 )} \times \smash{H^1 (\Omega_-^0 )} \times \smash{H^1 (\Gamma_a) }$.
Further, $\smash{\Phi^\ast}$ is convex and closed in $\smash{H^1 (\Omega_+^0 )} \times \smash{H^1 (\Omega_-^0 )} \times \smash{H^1_\vct{N} (\Gamma_a) }$ and hence $\smash{( \hat{p}_+^\ast , \hat{p}_-^\ast , \hat{p}^\ast_\rmf ) \in \Phi^\ast}$ if $\alpha \le 1 $.
\end{proof}

Using \Cref{prop:apriori_l2_3}, we can conclude that the limit solution~$\smash{\hat{p}_\rmf^\ast}$ in~$\smash{\Gamma_a}$ is constant in $\theta_n$-direction if $\alpha < 1$ and completely constant if $\alpha < -1$.
\begin{proposition} \label{prop:pconst} 
Let $\beta \ge -1$. 
Besides, let either $\alpha \le 0$ or, given the assumption~\eqref{asm:rhoD}, let $\smash{ 2\beta \ge\alpha - 3}$. 
\begin{enumerate}[wide,label=(\roman*),topsep=0pt]
\item \label{item:pconst1} Let $\alpha < -1$. Then, for a.a.~$\smash{(\vct{p} , \theta_n ) \in \Gamma_a }$, the limit function~$\smash{\hat{p}_\mathrm{f}^\ast \in H^1 (\Gamma_a )}$ from \Cref{prop:convergence} satisfies
%\begin{linenomath}
\begin{align}
\nablaGamma \hat{p}_\mathrm{f}^\ast (\vct{p}, \theta_n ) = \nablaN \hat{p}_\rmf^\ast (\vct{p}, \theta_n ) = \vct{0} \quad\enspace \Rightarrow\enspace \hat{p}_\mathrm{f}^\ast (\vct{p}, \theta_n)  = \mathfrak{A}_\mathrm{f} \hat{p}_\mathrm{f}^\ast = \mathrm{const}.
\end{align}
%\end{linenomath}
\item \label{item:pconst2} Let $\alpha < 1$. Then, for a.a.~$\smash{(\vct{p}, \theta_n ) \in \Gamma_a }$, the limit function~$\smash{\hat{p}_\mathrm{f}^\ast \in H^1_{\vct{N}} (\Gamma_a )}$ from \Cref{prop:convergence} satisfies
%\begin{linenomath}
\begin{align}
\nablaN \hat{p}_\mathrm{f}^\ast (\vct{p} , \theta_n ) = \vct{0} \quad\enspace \Rightarrow\enspace \hat{p}_\mathrm{f}^\ast (\vct{p} , \theta_n) = (\mathfrak{A}_\Gamma \hat{p}_\mathrm{f}^\ast ) (\vct{p} ) .
\end{align}
%\end{linenomath}
\end{enumerate}
\end{proposition}
\begin{proof}
The results follow from the \Cref{prop:apriori_l2_3,prop:convergence}.
\end{proof}

If $\alpha < 1$, we obtain continuity of the limit solution across the interface~$\Gamma$.
\begin{proposition} \label{prop:pressurecontinuity} 
Let $\beta \ge -1$.
Besides, let either $\alpha \le 0$ or, given the assumption~\eqref{asm:rhoD}, let $\smash{ 2\beta \ge\alpha - 3}$. 
Then, if $\alpha < 1$, the limit functions~$\smash{\hat{p}^\ast_\pm \in H^1 (\Omega^0_\pm)}$ and $\smash{\hat{p}^\ast_\mathrm{f} \in H^1_{\vct{N}}(\Gamma_a )}$ from \Cref{prop:convergence} satisfy 
%\begin{linenomath}
\begin{align}
\hat{p}^\ast_\pm\bigr\vert_{\Gamma}  &= \mathfrak{A}_\Gamma \hat{p}^\ast_\mathrm{f}   \quad\enspace \text{a.e.\ on } \Gamma.
\end{align}
%\end{linenomath}
\end{proposition} 
\begin{proof}
Let $\zeta \in L^2_a (\Gamma )$.
Then, we have
%\begin{linenomath}
\begin{equation}
\begin{multlined}[c][0.875\displaywidth]
\abssize{\Big}{\bigl( \hat{p}_\pm^\ast - \mathfrak{A}_\Gamma \hat{p}_\mathrm{f}^\ast , \zeta \bigr)_{L^2_a(\Gamma )}} \le \abssize{\Big}{\bigl( \hat{p}_\pm^\ast - \hat{p}_\pm^{\epsilon_k} , \zeta \bigr)_{L^2_a(\Gamma )}}  \\
+ \abssize{\Big}{\bigl( \hat{p}_\pm^{\epsilon_k} - \mathfrak{A}_\Gamma \hat{p}_\mathrm{f}^{\epsilon_k} , \zeta \bigr)_{L^2_a(\Gamma )}}  + \abssize{\Big}{\bigl( \mathfrak{A}_\Gamma \hat{p}_\mathrm{f}^{\epsilon_k} - \mathfrak{A}_\Gamma\hat{p}_\mathrm{f}^\ast , \zeta \bigr)_{L^2_a(\Gamma )}}.
\end{multlined}%
\label{eq:T123}
\end{equation}
%\end{linenomath}
Using a version of the Sobolev trace inequality~\cite[Thm.\ II.4.1]{galdi11}, we obtain
%\begin{linenomath}
\begin{align*}
\norm{ \hat{p}_\pm^\ast - \hat{p}_\pm^{\epsilon_k} }_{L^2 ( \Gamma )}^2 &\lesssim \norm{ \hat{p}_\pm^\ast - \hat{p}_\pm^{\epsilon_k} }_{L^2 ( \Omega_\pm^0 )} \norm{ \hat{p}_\pm^\ast - \hat{p}_\pm^{\epsilon_k} }_{H^1 (\Omega_\pm^0 )} ,
\end{align*}
%\end{linenomath}
where, with \Cref{prop:convergence}, the first term vanishes as $k \rightarrow \infty$ and the second term is bounded. 
Besides, by using the \Cref{lem:traceineq,lem:apriori_l2_0} and \Cref{prop:apriori_l2_3}, we find
%\begin{linenomath}
\begin{align*}
\normsize{\big}{\hat{p}^{\epsilon_k}_\pm - \mathfrak{A}_\Gamma \hat{p}_\mathrm{f}^{\epsilon_k} }_{L^2_a ( \Gamma ) } = \normsize{\big}{\mathfrak{T}_\pm^{\epsilon_k} \hat{p}_\rmf^{\epsilon_k} - \mathfrak{A}_\Gamma \hat{p}_\mathrm{f}^{\epsilon_k} }_{L^2_a ( \Gamma ) } \lesssim \norm{\nablaN \hat{p}^{\epsilon_k}_\mathrm{f}}_{\vct{L}^2 (\Gamma_a )} \lesssim \epsilon_k^{\frac{1-\alpha}{2}} \rightarrow 0 .
\end{align*}
%\end{linenomath}
Further, the last term on the right-hand side of \cref{eq:T123} vanishes due to the weak convergence~\eqref{eq:conv_D} in \Cref{prop:convergence} as $k\rightarrow\infty$.
\end{proof}

\section{Limit Models}  \label{sec:sec4}
In the following, we present the convergence proofs and resulting limit models for vanishing~$\epsilon$. 
Depending on the value of the parameter~$\alpha \in \mathbb{R}$, we obtain five different limit models. 
We distinguish between the following cases that are discussed in separate subsections.
\begin{itemize}[wide]
\item \Cref{sec:sec41} discusses the case $\alpha < - 1$ of a highly conductive fracture, where the limit pressure head inside the fracture becomes completely constant. 
\item In \Cref{sec:sec42}, we discuss the case $\alpha = - 1$ of a conductive fracture, where the fracture pressure head in the limit model solves a PDE of effective Darcy flow on the interface~$\Gamma$.
\item In \Cref{sec:sec43}, we examine the case $\alpha \in (- 1, 1)$, where the fracture disappears in the limit model, i.e., we have both the continuity of pressure and normal velocity across the interface~$\Gamma$ without any effect of the fracture conductivity.
\item \Cref{sec:sec44} is concerned with the case $\alpha = 1$, where the fracture turns into a permeable barrier with a pressure jump across the interface~$\Gamma$ that scales with an effective conductivity. 
\item \Cref{sec:sec45} discusses the case $\alpha > 1$, where the fracture acts like a solid wall in the limit model.
\end{itemize}
Each subsection is structured as follows. 
First, we state the strong formulation of the respective limit model and introduce a corresponding weak formulation. 
Then, we prove weak convergence towards the limit model for the subsequence~$\smash{\{ \epsilon_k\}_{k\in\mathbb{N}}}$ as $k\rightarrow \infty$ and express the limit solution in terms of the limit functions from \Cref{prop:convergence}. 
In a second step, we show strong convergence for the whole sequence~$\smash{\{ \epsilon\}_{\epsilon \in ( 0, \hat{\epsilon}]}}$ as $\epsilon \rightarrow 0$ and  discuss the wellposedness of the limit model.

Further, regarding the convergence of the bulk solution, we obtain the following result that will be useful for all cases.
\begin{lemma} \label{lem:limApm}
Let $\beta \ge -1$. Besides, let either $\alpha \le 0$ or, given the assumption~\eqref{asm:rhoD}, let $\smash{ 2\beta \ge\alpha - 3}$. 
Then, for all $\smash{\phi_\pm \in H^1 (\Omega^0_\pm)}$, we have 
%\begin{linenomath}
\begin{align}
\mathcal{A}_\pm^{\epsilon_k} ( \hat{p}^{\epsilon_k}_\pm , \phi_\pm ) \rightarrow \bigl( \matr{K}_\pm^0 \nabla \hat{p}_\pm^\ast , \nabla \phi_\pm \bigr)_{\vct{L}^2 (\Omega_\pm^0 )}
\end{align}
%\end{linenomath}
as $k\rightarrow \infty$, where $\smash{\hat{p}_\pm^\ast \in H^1 (\Omega_\pm^0 )}$ denote the limit functions from \Cref{prop:convergence}.
\end{lemma}
\begin{proof}
For all $\smash{\phi_\pm \in H^1 (\Omega^0_\pm ) }$, we find 
%\begin{linenomath}
\begin{align*}
&\bigl( \matr{K}^0_\pm \matr{M}^{\epsilon_k}_\pm \nabla \hat{p}^{\epsilon_k}_\pm , \matr{M}^{\epsilon_k}_\pm \nabla \phi_\pm \bigr)_{\vct{L}^2 (\Omega^0_{\pm , \mathrm{in}})} - \bigl( \matr{K}^0_\pm \nabla \hat{p}^\ast_\pm , \nabla \phi_\pm \bigr)_{\vct{L}^2 (\Omega^0_{\pm , \mathrm{in}})} \\
&\enspace = \bigl( \matr{K}^0_\pm \bigl[ \matr{M}_\pm^{\epsilon_k} - \matr{I}_n \bigr] \nabla \hat{p}^{\epsilon_k}_\pm , \matr{M}_\pm^{\epsilon_k} \nabla \phi_\pm \bigr)_{\vct{L}^2 (\Omega^0_{\pm , \mathrm{in}})} \\
&\hspace{0.6cm} + \bigl( \matr{K}_\pm^0 \nabla \hat{p}^{\epsilon_k}_\pm , \bigl[ \matr{M}^{\epsilon_k}_\pm - \matr{I}_n \bigr] \nabla \phi_\pm  \bigr)_{\vct{L}^2 (\Omega^0_{\pm , \mathrm{in}})}
+ \bigl( \matr{K}^0_\pm \nabla \bigl[  \hat{p}_\pm^{\epsilon_k} - \hat{p}_\pm^\ast \bigr] , \nabla \phi_\pm \bigr)_{\vct{L}^2 (\Omega^0_{\pm , \mathrm{in}})} .
\end{align*}
%\end{linenomath}
As $k \rightarrow \infty$, the first two terms on the right-hand side vanish due to \Cref{lem:trafo}~\ref{item:trafo4}, the third term due to \Cref{prop:convergence}. 
Thus, the result follows with \Cref{prop:convergence}.
\end{proof}

Besides, for $\smash{f_\pm \colon \Omega_\pm^0 \rightarrow \mathbb{R}}$ and $\smash{\vct{F}_\pm \colon \Omega_\pm^0  \rightarrow \mathbb{R}^n}$ with well-defined (normal) trace on~$\Gamma$, we define the jump operators 
%\begin{linenomath}
\begin{align}
\jump{f}_\Gamma := {f_+}\bigr\vert_{\Gamma } - {f_-}\bigr\vert_{\Gamma } , \qquad 
\jump{\vct{F}}_\Gamma := \bigl[ \vct{F}_+ \cdot \vct{N} \bigr] \! \bigr\vert_\Gamma -\bigl[ \vct{F}_- \cdot \vct{N} \bigr]\! \bigr\vert_\Gamma\,  .
\end{align}
%\end{linenomath}

\subsection{Case I: \texorpdfstring{$\alpha < - 1$}{alpha < -1}} \label{sec:sec41}
If $\alpha < -1$, the fracture conductivity is much larger than the bulk conductivity.
As a result, the pressure head~$\smash{\hat{p}_\rmf^\epsilon}$ inside the fracture becomes constant as~$\epsilon \rightarrow 0$, i.e., pressure fluctuations in the fracture are instantaneously equilibrated.
The range of achievable constants for the fracture pressure head in the limit model may be constrained by the choice of Dirichlet conditions at the external fracture boundary.
For this reason, we define the set 
%\begin{linenomath}
\begin{align}
W := \bigl\{ \phi^\ast \in \mathbb{R} \ \big\vert\ \exists\, (\phi_+ , \phi_- , \phi_\rmf ) \in \Phi \text{ with } \phi_\rmf \equiv \phi^\ast \bigr\}
\end{align}
%\end{linenomath}
of admissible constants for the limit pressure head in the fracture.
Then, the set~$W$ can be characterized as follows.
\begin{remark} \label{rem:M}
\begin{enumerate}[wide,label=(\roman*)]
\item It is either $W = \mathbb{R}$ or $W  = \{ 0 \}$.
\item If $\smash{\lambda_{\partial\Omega } (\rho^{\hat{\epsilon}}_{\rmf , \mathrm{D}}) > 0}$, then we have $W = \{ 0 \} $.
\item If $\smash{\lambda_{\partial \Omega }( \rho_{\rmf ,\mathrm{D}}^{\hat{\epsilon}} ) = 0 }$ and $\smash{ \lambda_{\partial \Omega } ( \rho_{\mathrm{b} ,\mathrm{D}}^0 \cap U_\delta ( \Gamma ) ) = 0}$ for a constant~$\smash{\delta > 0}$, then we have $W = \mathbb{R}$.
\end{enumerate}
\end{remark}

The strong formulation of the limit problem for $\alpha < -1$ and $\beta \ge -1$ as $\epsilon \rightarrow 0$ now reads as follows.

Find $p_\pm \colon \Omega^0_\pm \rightarrow \mathbb{R}$ and $p_\Gamma \in W$ such that
%\begin{linenomath}
\begin{subequations}
\begin{alignat}{3}
-\nabla \cdot \bigl( \matr{K}_\pm^0 \nabla p_\pm \bigr) &= q_\pm^0 \qquad &&\text{in } \Omega_\pm^0 ,  \\
p_+ &= p_- \qquad &&\text{on } \Gamma_0^0 , \\
\matr{K}_+^0 \nabla p_+ \cdot \vct{N} &=  \matr{K}_-^0 \nabla p_- \cdot \vct{N} \qquad &&\text{on } \Gamma_0^0 , \\
 p_\pm &\equiv p_\Gamma  \qquad &&\text{on } \Gamma , \\
p_\pm &= 0 \qquad &&\text{on } \rho_{\pm,\mathrm{D}}^0 , \ \\
\matr{K}_\pm^0 \nabla p_\pm \cdot \vct{n} &= 0 \qquad &&\text{on } \rho_{\pm,\mathrm{N}}^0 .
\intertext{Moreover, if $W=\mathbb{R}$, the model is closed by the condition}
\label{eq:stronglimit_alm1_g} \int_\Gamma \jumpsize{\big}{\matr{K}^0\nabla p}_{\Gamma} \,\D \lambda_\Gamma  + A \,\overline{q_\Gamma} &= 0. 
\end{alignat}%
\label{eq:stronglimit_alm1}%
\end{subequations}%
%\end{linenomath}
Here, $A\in\mathbb{R}$ and $\smash{\overline{q_\Gamma}} \in \mathbb{R}$  are defined by 
\begin{align}
A := \int_\Gamma a \,\D \lambda_\Gamma , \qquad
\overline{q_\Gamma} &:= \begin{cases} 
 \mathfrak{A}_\mathrm{f} \hat{q}_\mathrm{f}  , & \text{if } \beta = -1, \\
0 , &\text{if } \beta > - 1 .
\end{cases}
\end{align}
A weak formulation of the system in \cref{eq:stronglimit_alm1} is given by the following problem.

Find $ ( p_+ , p_-, p_\Gamma ) \in \Phi^0_\mathrm{I}$ such that, for all $ ( \phi_+ , \phi_- , \phi_\Gamma) \in \Phi^0_\mathrm{I}$,
%\begin{linenomath}
\begin{align}
&\sum_{i = \pm} \bigl( \matr{K}_i^0 \nabla p_i , \nabla \phi_i )_{\vct{L}^2 (\Omega_i^0  ) } = \sum_{i = \pm} \bigl( q_i^0 , \phi_i \bigr)_{L^2 (\Omega_i^0 )} + A\, \overline{q_\Gamma} \phi_\Gamma .
\label{eq:weaklimit_alm1}
\end{align}
%\end{linenomath}
Here, the space~$\Phi^0_\mathrm{I}$ is given by
%\begin{linenomath}
\begin{align}
\begin{split}
\Phi^0_\mathrm{I} &:= \Bigl\{ ( \phi_+ , \phi_- , \phi_\Gamma) \in \bigl[ \bigtimes\nolimits_{i=\pm} H_{0,\rho_{i,\mathrm{D}}^0}^1 (\Omega_i^0 ) \bigr] \times W \ \Big\vert\  \\
&\hspace{5.5cm} \phi_+\bigr\vert_{\Gamma_0^0} = \phi_-\bigr\vert_{\Gamma_0^0 } , \enspace   \phi_\pm\bigr\vert_{\Gamma } \equiv \phi_\Gamma  \Bigr\}  \\
&\hphantom{:}\cong \bigl\{ \phi \in H^1_{0,\rho^0_{\mathrm{b}, \mathrm{D}}} (\Omega ) \ \big\vert\ \phi\vert_{\Gamma} \equiv \mathrm{const.} \in W \bigr\} . 
\end{split} 
\end{align}
%\end{linenomath}
Further, we obtain the following weak convergence result.
\begin{theorem} \label{thm:alm1}
Let $\alpha < - 1$ and $\beta \ge -1$.
Then, $\smash{( \hat{p}^\ast_+ , \hat{p}^\ast_- , \mathfrak{A}_\mathrm{f}  \hat{p}_\mathrm{f}^\ast ) \in \Phi_\mathrm{I}^0}$ is a weak solution of problem~\eqref{eq:weaklimit_alm1}, where $\smash{\hat{p}^\ast_\pm\in H^1(\Omega^0_\pm )}$, $\smash{\hat{p}^\ast_\mathrm{f} \in H^1 (\Gamma_a )}$ denote the limit functions from \Cref{prop:convergence}.
Moreover, we have $\smash{\hat{p}_\mathrm{f}^\ast (\vct{p}, \theta_n ) = \mathfrak{A}_\mathrm{f} \hat{p}^\ast_\mathrm{f} \in W}$ for a.a.~$\smash{ (\vct{p}, \theta_n )  \in \Gamma_a }$.
\end{theorem}
\begin{proof}
Take a test function triple $\smash{ (\phi_+ , \phi_- , \phi_\rmf) \in \Phi}$  with $\smash{\phi_\rmf \equiv \phi_\Gamma \in W}$.
By inserting $\smash{(\phi_+ , \phi_- , \phi_\rmf) }$ into the transformed weak formulation~\eqref{eq:rescaledweakdarcyeps}, we obtain
%\begin{linenomath}
\begin{align*}
&\sum_{i = \pm} \mathcal{A}_i^{\epsilon_k} (\hat{p}_i^{\epsilon_k} , \phi_i ) = \bigl[ 1+ \mathcal{O}(\epsilon_k ) \bigr] \biggl[ \sum_{i = \pm } \bigl( q_i^0 , \phi_i \bigr)_{L^2 (\Omega_i^0 )} + \epsilon_k^{\beta + 1 } A ( \mathfrak{A}_\rmf \hat{q}_\rmf  ) \phi_\Gamma \biggr]  .
\end{align*}
%\end{linenomath}
Thus, by letting~$k  \rightarrow  \infty$ and using  \Cref{lem:limApm}, we find that the limit solution~$\smash{  ( \hat{p}^\ast_+ , \hat{p}^\ast_- , \mathfrak{A}_\mathrm{f} \hat{p}_\mathrm{f}^\ast )} $ satisfies~\cref{eq:weaklimit_alm1}. 
Besides, with the \Cref{prop:convergence,prop:pconst}, it is $\smash{(\hat{p}^\ast_+ , \hat{p}_-^\ast , \hat{p}_\rmf^\ast ) \in \Phi}$ with $\smash{ \hat{p}_\rmf^\ast \equiv \mathfrak{A}_\rmf \hat{p}_\rmf^\ast }$ and hence $\smash{(\hat{p}_+^\ast , \hat{p}_-^\ast , \mathfrak{A}_\rmf \hat{p}_\rmf^\ast ) \in \Phi_\mathrm{I}^0}$.
\end{proof}
Moreover, we obtain strong convergence in the following sense.
\begin{theorem} 
Let $\alpha < -1$ and $\beta \ge -1$. 
Then, for the whole sequence~$\smash{ \{ \hat{p}_i^\epsilon \}_{\epsilon\in (0,\hat{\epsilon}]}}$, $\smash{i\in\{ + , - , \mathrm{f}\}}$, we have strong convergence 
%\begin{linenomath}
\begin{subequations}
\begin{alignat}{2}
\hat{p}_\pm^\epsilon &\rightarrow \hat{p}^\ast_\pm \quad\enspace &&\text{in } H^1 (\Omega_\pm^0 ) , \\
\hat{p}_\mathrm{f}^\epsilon &\rightarrow \hat{p}^\ast_\mathrm{f} \quad\enspace &&\text{in } H^1 (\Gamma_a) 
\end{alignat}%
\label{eq:strongconv_alm1}%
\end{subequations}%
%\end{linenomath}
as $\epsilon \rightarrow 0$.
Further, $\smash{ ( \hat{p}^\ast_+ , \hat{p}^\ast_- , \mathfrak{A}_\mathrm{f}  \hat{p}_\mathrm{f}^\ast ) \in \Phi_\mathrm{I}^0}$ is the unique weak solution of problem ~\eqref{eq:weaklimit_alm1}.
\end{theorem}
\begin{proof}
The solution of \cref{eq:weaklimit_alm1} is unique as a consequence of the Lax-Milgram theorem. 
Thus, the weak convergence~\eqref{eq:conv_A} and~\eqref{eq:conv_B} in \Cref{prop:convergence} hold for the whole sequence~$\smash{ \{ \hat{p}_i^\epsilon \}_{\epsilon\in (0,\hat{\epsilon}]}}$, $\smash{i\in\{ + , - , \mathrm{f}\}}$.
This follows from \Cref{prop:apriori_l2_3} and the fact that every weakly convergent subsequence has the same limit.

Now, in order to show the strong convergence~\eqref{eq:strongconv_alm1}, we define the norm~$\smash{\normiii{\cdot }}$ on~$\Phi \subset \smash{ H^1 (\Omega_+^0 ) \times H^1 (\Omega_-^0 ) \times H^1 (\Gamma_a )}$ by
%\begin{linenomath}
\begin{align}
&\normsizeiii{\big}{(\phi_+ , \phi_- , \phi_\mathrm{f} )}^2 :=  \sum_{i=\pm } \bigl( \matr{K}_i^0 \nabla \phi_i , \nabla \phi_i \bigr)_{\vct{L}^2 (\Omega_i^0  )} + \bigl( \hat{\matr{K}}_\mathrm{f} \nabla_{\!\Gamma_a} \phi_\mathrm{f} , \nabla_{\! \Gamma_a } \phi_\mathrm{f} \bigr)_{\vct{L}^2 (\Gamma_a ) } .
\label{eq:normiii}
\end{align}
%\end{linenomath}
Then, with \Cref{lem:apriori_l2_2}, it is easy to see that the norm~$\smash{\normiii{\cdot }}$ on the space~$\Phi$ is equivalent to the natural product norm of $\smash{ H^1 (\Omega_+^0 ) \times H^1 (\Omega_-^0 ) \times H^1 (\Gamma_a )}$.
Moreover, with analogous arguments as in \Cref{lem:limApm}, we find 
%\begin{linenomath}
\begin{align}
\bigl( \matr{K}^0_\pm \nabla \hat{p}_\pm^\epsilon , \nabla \hat{p}^\epsilon_\pm \bigr)_{\vct{L}^2 (\Omega^0_{\pm , \mathrm{in}})} &= \bigl( \matr{K}^0_\pm \matr{M}^\epsilon_\pm \nabla \hat{p}_\pm^\epsilon , \matr{M}^\epsilon_\pm \nabla \hat{p}_\pm^\epsilon \bigr)_{\vct{L}^2 (\Omega_{\pm , \mathrm{in}}^0)} + \mathcal{O} (\epsilon ) . \label{eq:limApm2}
\end{align}
%\end{linenomath}
The uniform ellipticity of~$\smash{\hat{\matr{K}}_\rmf}$ and \Cref{prop:apriori_l2_3} yield
%\begin{linenomath}
\begin{align*}
\bigl( \hat{\matr{K}}_\rmf \nabla_{\!\Gamma_a } \hat{p}_\rmf^\epsilon , \nabla_{\! \Gamma_a } \hat{p}_\rmf^\epsilon \bigr)_{\vct{L}^2 (\Gamma_a )} \lesssim \norm{\nablaGamma \hat{p}^\epsilon_\rmf }^2_{\vct{L}^2 (\Gamma_a ) } + \norm{\nablaN \hat{p}_\rmf^\epsilon }^2_{\vct{L}^2 (\Gamma_a ) } = \mathcal{O}(\epsilon^{-\alpha - 1} ) .
\end{align*}
%\end{linenomath}
Thus, with $\smash{\mathcal{A}_\mathrm{f}^\epsilon (\hat{p}_\mathrm{f}^\epsilon , \hat{p}_\mathrm{f}^\epsilon ) \ge 0}$ and \cref{eq:rescaledweakdarcyeps}, we have 
%\begin{linenomath}
\begin{align}
\label{eq:estimateiii}
\begin{split}
& \normsizeiii{\big}{( \hat{p}^\epsilon_+ , \hat{p}_-^\epsilon , \hat{p}^\epsilon_\mathrm{f}) }^2 \le \sum_{i=\pm} \mathcal{A}_i^\epsilon (\hat{p}^\epsilon_i , \hat{p}_i^\epsilon ) + \mathcal{A}_\mathrm{f}^\epsilon (\hat{p}_\mathrm{f}^\epsilon , \hat{p}_\mathrm{f}^\epsilon ) + \smallo (\epsilon )  \\
&\hspace{1.5cm } = \bigl[1 + \mathcal{O}(\epsilon ) \bigr] \biggl[ \sum_{i=\pm } \bigl( q_i^0 , \hat{p}_i^\epsilon \bigr)_{L^2 (\Omega_i^0 ) } + \epsilon^{\beta + 1 } \bigl( \hat{q}_\mathrm{f} , \hat{p}_\mathrm{f}^\epsilon \bigr)_{L^2 (\Gamma_a )}  \biggr] + \smallo (\epsilon ) .
\end{split}%
\end{align}%
%\end{linenomath}
With \Cref{prop:convergence} and \Cref{thm:alm1}, we find
%\begin{linenomath}
\begin{align*}
\limsup_{\epsilon\rightarrow 0 } \normsizeiii{\big}{( \hat{p}^\epsilon_+ , \hat{p}_-^\epsilon , \hat{p}^\epsilon_\mathrm{f}) }^2 &\le  \sum_{i = \pm} \bigl( q_i^0 , \hat{p}_i^\ast \bigr)_{L^2 (\Omega_i^0 )} + A\, \overline{q_\Gamma} \, \mathfrak{A}_\mathrm{f} \hat{p}_\mathrm{f}^\ast \\
&= \sum_{i = \pm} \bigl( \matr{K}_i^0 \nabla \hat{p}_i^\ast , \nabla \hat{p}_i^\ast )_{\vct{L}^2 (\Omega_i^0  ) }  \le \normsizeiii{\big}{( \hat{p}^\ast_+ , \hat{p}_-^\ast , \hat{p}^\ast_\mathrm{f}) }^2. 
\end{align*}
%\end{linenomath}
Consequently, with the weak lower semicontinuity of the norm, we obtain
%\begin{linenomath}
\begin{align*}
\lim_{\epsilon \rightarrow 0 } \normsizeiii{\big}{( \hat{p}^\epsilon_+ , \hat{p}_-^\epsilon , \hat{p}^\epsilon_\mathrm{f}) } = \normsizeiii{\big}{( \hat{p}^\ast_+ , \hat{p}_-^\ast , \hat{p}^\ast_\mathrm{f}) } . \tag*{\qedhere}
\end{align*} 
%\end{linenomath}
\end{proof}

\subsection{Case II: \texorpdfstring{$\alpha = -1$}{alpha = -1}}
\label{sec:sec42}
For $\alpha = -1$ and $\beta \ge -1$, the fracture pressure head in the limit models fulfills a Darcy-like PDE on the interface~$\Gamma$ with an effective hydraulic conductivity matrix~$\smash{\matr{K}_\Gamma}$. 
The inflow from the bulk domains into the fracture is modeled by an additional source term on the right-hand side of the interfacial PDE.
The bulk and interface solution are coupled by the continuity of the pressure heads across the interface~$\Gamma$, which corresponds to the case of a conductive fracture in accordance with the choice of the parameter~$\alpha = -1$. 
We remark that the effective conductivity matrix~$\smash{\matr{K}_\Gamma}$ for the limit fracture in \cref{eq:KGamma} below explicitly depends on the off-diagonal entries of the full-dimensional conductivity matrix~$\smash{\hat{\matr{K}}_\rmf}$, which is not accounted for in previous works with equivalent scaling of bulk and fracture conductivities~\cite{huy74,list20,morales10}. 

The resulting limit model for $\alpha = -1$ resembles discrete fracture models for Darcy flow that are derived by averaging methods~\cite{burbulla23,martin05}.
The averaging approach leads to a Darcy-like PDE on the fracture interface~$\Gamma$ as in \cref{eq:stronglimit_aeqm1_b} below. 
However, the choice of coupling conditions between bulk and interface solution does not occur naturally in this case, especially if the averaged model aspires to describe both conductive and blocking fractures. 
Therefore, coupling conditions in averaged models are typically obtained by making formal assumptions on the flow profile inside the fracture.

The strong formulation of the limit problem for $\alpha = -1$ and $\beta \ge -1$ now reads as follows.

Find $\smash{p_\pm \colon \Omega_\pm^0 \rightarrow \mathbb{R}}$ and $\smash{p_\Gamma \colon \Gamma \rightarrow \mathbb{R}}$ such that 
%\begin{linenomath}
\begin{subequations}
\begin{alignat}{3}
-\nabla \cdot \bigl( \matr{K}_\pm^0 \nabla p_\pm \bigr) &= q_\pm^0 \qquad &&\text{in } \Omega_\pm^0 , \\
\label{eq:stronglimit_aeqm1_b} -\nablaGamma \cdot \bigl( a \matr{K}_\Gamma \nablaGamma p_\Gamma  \bigr) &= aq_\Gamma + \jumpsize{\big}{\matr{K}^0 \nabla p }_\Gamma \qquad &&\text{in } \Gamma , \\
{p}_+ = {p}_- &= p_\Gamma \qquad &&\text{on } \Gamma , \\
{p}_+ &= {p}_- \qquad &&\text{on } \Gamma_0^0 , \\
\matr{K}_+^0 \nabla p_+ \cdot \vct{N} &= \matr{K}_-^0 \nabla p_- \cdot \vct{N} \qquad &&\text{on } \Gamma_0^0 , \\
{p}_\pm &= 0 \qquad &&\text{on } \rho_{\pm ,\mathrm{D}}^0 ,  \\
\matr{K}_\pm^0 \nabla {p}_\pm \cdot \vct{n} &= 0 \qquad &&\text{on } \rho_{\pm ,\mathrm{N}}^0  ,  \\
p_\Gamma &= 0 \qquad &&\text{on } \partial\Gamma_\mathrm{D}  , \label{eq:stronglimit_aeqm1_h} \\
\matr{K}_\Gamma \nablaGamma  p_\Gamma \cdot \vct{n} &= 0 \qquad &&\text{on } \partial\Gamma_\mathrm{N}  . \label{eq:stronglimit_aeqm1_i}
\end{alignat}%
\label{eq:stronglimit_aeqm1}%
\end{subequations}%
%\end{linenomath}
In \cref{eq:stronglimit_aeqm1_b}, $\smash{q_\Gamma \in L^2_a (\Gamma )}$ and $\smash{\matr{K}_\Gamma \in L^\infty (\Gamma ; \mathbb{R}^{n\times n})}$ are given by 
%\begin{linenomath}
\begin{align}
\label{eq:qgamma} q_\Gamma (\vct{p}  ) &:= \begin{cases} ( \mathfrak{A}_\Gamma \hat{q}_\rmf ) (\vct{p} )   &\text{if } \beta = -1 , \\
0 &\text{if } \beta > - 1 ,  \end{cases} \\
\label{eq:KGamma} \matr{K}_\Gamma (\vct{p} ) &:= \bigl( \mathfrak{A}_\Gamma \bigl[ \hat{\matr{K}}_\rmf - \bigl[ \hat{\matr{K}}_\rmf \vct{N} \cdot \vct{N} \bigr]^{-1}  \hat{\matr{K}}_\rmf \vct{N} \otimes \hat{\matr{K}}_\rmf \vct{N}  \bigr] \bigr) (\vct{p}) .
\end{align} 
%\end{linenomath}
Here, the application of the operator~$\smash{\mathfrak{A}_\Gamma}$ is to be understood componentwise.
The boundary parts~$\smash{\partial\Gamma_\mathrm{D}}$, $\smash{\partial\Gamma_\mathrm{N}}$ in the~\cref{eq:stronglimit_aeqm1_i,eq:stronglimit_aeqm1_h} are given by%
%\begin{linenomath}
\begin{subequations}
\begin{align}
\partial \Gamma_{\mathrm{D}} &:= \bigl\{ \vct{p} \in \partial\Gamma \ \big\vert\ \exists Z_{\vct{p}} \subset \mathbb{R}, \abssize{}{Z_{\vct{p}}} > 0 ,\ \forall \theta_n \in Z_{\vct{p}} \colon \vct{p} + \hat{\epsilon} \theta_n \vct{N} (\vct{p} )  \in \rho_{\mathrm{f},\mathrm{D}}^{\hat{\epsilon}} \bigr\} , \\
\partial \Gamma_\mathrm{N} &:= \partial \Gamma \setminus \partial \Gamma_\mathrm{D} .
\end{align}
\end{subequations}
%\end{linenomath}
Generally, in particular, we have $\smash{\partial \Gamma_\mathrm{N} \setminus \partial \Omega \not= \emptyset}$, i.e., we also have a homogeneous Neumann condition at closing points of the fracture inside the domain.

A weak formulation of the system in \cref{eq:stronglimit_aeqm1} is given by the following problem.

Find $( {p}_+ , p_- , p_\Gamma) \in \Phi^0_\mathrm{II}$ such that, for all $ ( \phi_+ , \phi_- , \phi_\Gamma ) \in \Phi^0_\mathrm{II}$,
%\begin{linenomath}
\begin{align}
\begin{split}
&\sum_{i = \pm} \bigl( \matr{K}_i^0 \nabla p_i, \nabla \phi_i \bigr)_{\vct{L}^2 (\Omega_i^0  ) } +  \bigl( a \matr{K}_\Gamma \nablaGamma  p_\Gamma   ,  \nablaGamma \phi_\Gamma \bigr)_{\vct{L}^2 ( \Gamma )} \\
&\hspace{5cm}  = \sum_{i = \pm } \bigl( q_i^0 , \phi_i \bigr)_{L^2 (\Omega_i^0 )} + \bigl( aq_\Gamma , \phi_\Gamma \bigr)_{L^2(\Gamma )} .
\end{split}%
\label{eq:weaklimit_aeqm1}%
\end{align}%
%\end{linenomath}
Here, the space~$\Phi^0_\mathrm{II}$ is defined by 
%\begin{linenomath}
\begin{align}
\begin{split}
\Phi^0_\mathrm{II} &:= \Bigl\{ ( \phi_+ , \phi_- , \phi_\Gamma) \in \Bigl[\bigtimes\nolimits_{i=\pm} H_{0,\rho_{i,\mathrm{D}}^0}^1 (\Omega_i^0 )\Bigr] \times H^1_a (\Gamma )  \ \Big\vert\ \\
&\hspace{2cm}  \phi_+\bigr\vert_{\Gamma_0^0} = \phi_-\bigr\vert_{\Gamma_0^0} , \enspace \phi_\pm\bigr\vert_{\Gamma } = \phi_\Gamma , \enspace  \mathfrak{T}^a_\parallel \phi_\Gamma \bigr\vert_{\partial \Gamma_\mathrm{D}} \equiv 0 \Bigr\} \\
&\hphantom{:}\cong \Bigl\{ \phi \in H^1_{0,\rho^0_{\mathrm{b},\mathrm{D}}} (\Omega^0 ) \ \Big\vert\ \phi_\Gamma := \phi\bigr\vert_{\Gamma} \in H^1_a (\Gamma ), \ \mathfrak{T}^a_\parallel \phi_\Gamma \bigr\vert_{\partial \Gamma_\mathrm{D}} \equiv 0   \Bigr\}.
\end{split} 
\end{align}
%\end{linenomath}
We now have the following weak convergence result.
\begin{theorem} \label{thm:aeqm1}
Let $\alpha = -1$ and $\beta \ge - 1$.
Then, $\smash{( \hat{p}^\ast_+ , \hat{p}^\ast_- , \mathfrak{A}_\Gamma \hat{p}_\mathrm{f}^\ast ) \in \Phi_\mathrm{II}^0}$ is a weak solution of problem~\eqref{eq:weaklimit_aeqm1}, where $\smash{\hat{p}^\ast_\pm\in H^1(\Omega^0_\pm )}$, $\smash{\hat{p}^\ast_\mathrm{f} \in H^1 (\Gamma_a )}$ denote the limit functions from \Cref{prop:convergence}.
Further, for a.a.\ $\smash{(\vct{p} , \theta_n ) \in \Gamma_a }$, we have $\smash{\hat{p}_\mathrm{f}^\ast (\vct{p} , \theta_n) = (\mathfrak{A}_\Gamma \hat{p}^\ast_\mathrm{f} ) (\vct{p} )}$.
\end{theorem}
\begin{proof}
According to \Cref{prop:apriori_l2_3}, we have $\smash{\norm{\epsilon^{-1} \nablaN \hat{p}_\mathrm{f}^\epsilon}_{\vct{L}^2 (\Gamma_a )} \lesssim 1}$.
Thus, there exists~$\smash{\zeta^\ast \in L^2 (\Gamma_a)}$ such that
%\begin{linenomath}
\begin{align}
\epsilon_k^{-1} \nablaN \hat{p}_\mathrm{f}^{\epsilon_k} \rightharpoonup \zeta^\ast \vct{N} \quad\enspace\text{in } \vct{L}^2 (\Gamma_a ) \label{eq:zetalimit2}
\end{align}
%\end{linenomath}
as $k\rightarrow \infty$.
By multiplying the transformed weak formulation~\eqref{eq:rescaledweakdarcyeps} by~$\smash{\epsilon_k}$ and taking the limit~$k\rightarrow\infty$, we find
%\begin{linenomath}
\begin{align}
\bigl( \hat{\matr{K}}_\rmf \nablaGamma \hat{p}_\rmf^\ast , \nablaN \phi_\rmf \bigr)_{\vct{L}_2 (\Gamma_a )} + \bigl( \zeta^\ast \hat{\matr{K}}_\rmf \vct{N } , \nablaN \phi_\rmf \bigr)_{\vct{L}^2 (\Gamma_a ) } =  0 \label{eq:zetalimit4}
\end{align}
%\end{linenomath}
for any test function triple~$\smash{ (\phi_+ , \phi_- , \phi_\rmf ) \in \Phi}$, where we have used \Cref{lem:Roperators}.
A solution for $\smash{\zeta^\ast \in L^2 (\Gamma_a ) }$ is now clearly given by 
%\begin{linenomath}
\begin{align}
\zeta^\ast = - \bigl[ \hat{\matr{K}}_\rmf \vct{N}  \cdot \vct{N} \bigr]^{-1}\, \hat{\matr{K}}_\rmf \nablaGamma \hat{p}_\rmf^\ast \cdot \vct{N} . \label{eq:zetalimit3}
\end{align}
%\end{linenomath}
Moreover, suppose that $\smash{\bar{\zeta}^\ast \in L^2 (\Gamma_a )}$ is another solution of \cref{eq:zetalimit4}. 
Then, with \cref{eq:zetalimit4}, we find
%\begin{linenomath}
\begin{align*}
\bigl( \bigl[ \hat{\matr{K}}_\rmf \vct{N} \cdot \vct{N } \bigr] \bigl( \zeta^\ast - \bar{\zeta}^\ast \bigr) , \partial_{\theta_n} \phi_\rmf \bigr)_{L^2 (\Gamma_a ) } &= 0 \quad\enspace \text{for all } (\phi_+ , \phi_- , \phi_\rmf ) \in \Phi^\ast .
\end{align*}
%\end{linenomath}
Thus, by choosing $\smash{\phi_\rmf \in H^1_\vct{N}(\Gamma_a )}$ as
%\begin{linenomath}
\begin{align*}
\phi_\rmf ( \vct{p} ,\theta_n ) := \int_{-a_-(\vct{p})}^{\theta_n} \bigl( \zeta^\ast - \bar{\zeta}^\ast \bigr) (\vct{p} , \bar{\theta}_n ) \,\D \bar{\theta}_n ,
\end{align*}
%\end{linenomath}
we obtain $\smash{\zeta^\ast = \bar{\zeta}^\ast}$ a.e.\ in~$\Gamma_a$, i.e., $\zeta^\ast$ is uniquely determined by \cref{eq:zetalimit3}. 

Next, we define the space 
%\begin{linenomath}
\begin{align*}
\Phi_\rmf := \bigl\{ \phi_\mathrm{f} \in H^1(\Gamma_a )  \ \big\vert \ \nablaN \phi_\mathrm{f} \equiv \vct{0} \bigr\} \cong H^1_a (\Gamma ) 
\end{align*}
%\end{linenomath}
and take a test function triple $\smash{ (\phi_+ , \phi_- , \phi_\mathrm{f}) \in \Phi}$ with $\phi_\mathrm{f} \in \Phi_\mathrm{f}$.
Then, there is a function $\smash{\phi_\Gamma \in H^1_a (\Gamma )}$ with  $\smash{\phi_\mathrm{f} (\vct{p} , \theta_n )} = \smash{\phi_\Gamma ( \vct{p} )}$ a.e.\ in~$\smash{\Gamma_a }$. 
With \Cref{prop:convergence}, \Cref{lem:limApm}, and \cref{eq:zetalimit2}, we obtain 
%\begin{linenomath}
\begin{align*}
\mathcal{A}_\pm^{\epsilon_k} ( \hat{p}^{\epsilon_k}_\pm , \phi_\pm ) \enspace &\rightarrow\enspace \bigl( \matr{K}_\pm^0 \nabla \hat{p}_\pm^\ast, \nabla \phi_\pm^0 \bigr)_{\vct{L}^2 (\Omega_i^0 ) }  , \\
\mathcal{A}^{\epsilon_k}_\rmf ( \hat{p}_\rmf^{\epsilon_k}  , \phi_\rmf ) \enspace &\rightarrow \enspace \bigl( \hat{\matr{K}}_\rmf \nablaGamma \hat{p}_\rmf^\ast , \nablaGamma \phi_\rmf \bigr)_{\vct{L}^2 (\Gamma_a ) } + \bigl( \zeta^\ast \hat{\matr{K}}_\rmf \vct{N} , \nablaGamma \phi_\rmf \bigr)_{\vct{L}^2 (\Gamma_a ) } 
\end{align*}
%\end{linenomath}
as $k\rightarrow\infty$. 
Here, we have used that
%\begin{linenomath}
\begin{align*}
&\bigl( \hat{\matr{K}}_\rmf \invR^{\epsilon_k}_\rmf \nablaGamma \hat{p}_\rmf^{\epsilon_k} ,  \invR^{\epsilon_k}_\rmf \nablaGamma \phi_\rmf \bigr)_{\vct{L}^2 (\Gamma_a ) } - \bigl( \hat{\matr{K}}_\rmf \nablaGamma \hat{p}_\rmf^\ast , \nablaGamma \phi_\rmf \bigr)_{\vct{L}^2 (\Gamma_a ) } \\ 
&\hspace{0.25cm} = \bigl( \hat{\matr{K}}_\rmf \bigl[ \invR^{\epsilon_k}_\rmf - \operatorname{id}_{\rmT\Gamma}  \bigr] \nablaGamma \hat{p}^{\epsilon_k}_\rmf , \invR^{\epsilon_k}_\rmf \nablaGamma \phi_\rmf \bigr)_{\vct{L}^2 (\Gamma_a ) } \\
&\hspace{1cm} +
\bigl( \hat{\matr{K}}_\rmf \nablaGamma \hat{p}_\rmf^{\epsilon_k} , \bigl[ \invR^{\epsilon_k}_\rmf - \operatorname{id}_{\rmT\Gamma}  \bigr] \nablaGamma \phi_\rmf \bigr)_{\vct{L}^2 (\Gamma_a ) }  + 
\bigl( \hat{\matr{K}}_\rmf \nablaGamma \bigl[ \hat{p}_\rmf^{\epsilon_k} - \hat{p}_\rmf^\ast  \bigr] , \nablaGamma \phi_\rmf \bigr)_{\vct{L}^2 (\Gamma_a ) }
\end{align*}
%\end{linenomath}
for all $\smash{\phi_\rmf \in H^1 (\Gamma_a )}$, where the first two terms on the right-hand side vanish according to \Cref{lem:Roperators} as $k\rightarrow \infty$ and the third terms tends to zero with \Cref{prop:convergence}.
Moreover, with \cref{eq:zetalimit3} and \Cref{prop:pconst}~\ref{item:pconst2}, we have 
%\begin{linenomath}
\begin{equation*}
 \bigl( \hat{\matr{K}}_\rmf \nablaGamma \hat{p}_\rmf^\ast , \nablaGamma \phi_\rmf \bigr)_{\vct{L}^2 (\Gamma_a ) } + \bigl( \zeta^\ast \hat{\matr{K}}_\rmf \vct{N} , \nablaGamma \phi_\rmf \bigr)_{\vct{L}^2 (\Gamma_a ) }  
=  \bigl(  a \matr{K}_\Gamma \nablaGamma ( \mathfrak{A}_\Gamma \hat{p}_\mathrm{f}^\ast ) , \nablaGamma \phi_\Gamma \bigr)_{\vct{L}^2 (\Gamma ) } ,
\end{equation*}
%\end{linenomath}
where $\smash{\matr{K}_\Gamma}$ is defined by \cref{eq:KGamma}.
Thus, by inserting $\smash{(\phi_+ , \phi_- , \phi_\rmf )}$ into the transformed weak formulation~\eqref{eq:rescaledweakdarcyeps} and letting $k\rightarrow \infty$, it follows that the limit solution~$\smash{( \hat{p}^\ast_+ , \hat{p}^\ast_- , \mathfrak{A}_\Gamma \hat{p}^\ast_\mathrm{f} ) }$ satisfies \cref{eq:weaklimit_aeqm1}. 
Besides, with \Cref{lem:traceH1a} and \Cref{prop:pressurecontinuity}, we have $\smash{( \hat{p}^\ast_+ , \hat{p}^\ast_- , \mathfrak{A}_\Gamma \hat{p}^\ast_\mathrm{f} )  \in \Phi_\mathrm{II}^0}$. 
\end{proof}

The effective hydraulic conductivity matrix~$\smash{\matr{K}_\Gamma}$ has the following properties.
\begin{lemma} 
\begin{enumerate}[label=(\roman*),wide,topsep=0pt]
\item The effective hydraulic conductivity matrix~$\smash{\matr{K}_\Gamma}$ from \cref{eq:KGamma} is symmetric and positive semidefinite.
In addition, for all $\vct{p} \in \Gamma $ and $\vct{\xi} \in \rmT_\vct{p}\Gamma$, we have 
$\smash{\vct{\xi} \cdot \matr{K}_\Gamma (\vct{p} )  \,\vct{\xi} > 0}$.
\item If $\smash{\hat{\matr{K}}_\rmf \in \mathcal{C}^0 ( \overline{\Gamma}_a^\smallpar ; \mathbb{R}^{n\times n }) }$, then $\smash{\matr{K}_\Gamma }$ is uniformly elliptic on~$\rmT \Gamma$, i.e., for all~$\vct{p }\in \Gamma$ and $\vct{\xi}\in \rmT_\vct{p}\Gamma$, we have 
$\smash{\vct{\xi} \cdot \matr{K}_\Gamma (\vct{p} )  \,\vct{\xi} \gtrsim \abs{\vct{\xi}}^2 }$.
\end{enumerate}
\end{lemma}
\begin{proof}
\begin{enumerate}[label=(\roman*),wide,topsep=0pt]
\item $\smash{\matr{K}_\Gamma}$ is symmetric by definition.
Moreover, for $\vct{\xi} \in \mathbb{R}^n$, we have 
%\begin{linenomath}
\begin{align*}
\vct{\xi} \cdot \matr{K}_\Gamma \vct{\xi} = \mathcal{A}_\Gamma \bigl( \vct{\xi} \hat{\matr{K}}_\rmf \cdot \vct{\xi} - \bigl[ \hat{\matr{K}}_\rmf \vct{N} \cdot \vct{N} \bigr]^{-1} \bigl[ \hat{\matr{K}}_\rmf \vct{N} \cdot \vct{\xi} \bigr]^2 \bigr) . 
\end{align*}
%\end{linenomath}
With the Cauchy-Schwarz inequality, we obtain 
%\begin{linenomath}
\begin{align*}
\bigl[ \hat{\matr{K}}_\rmf \vct{N} \cdot \vct{\xi} \bigr]^2 = \bigl[ \hat{\matr{K}}_\rmf^{\frac{1}{2}} \vct{N} \cdot \hat{\matr{K}}_\rmf^{\frac{1}{2}} \vct{\xi} \bigr]^2 \le \bigl(  \hat{\matr{K}}_\rmf \vct{N} \cdot \vct{N}\bigr) \bigl(  \hat{\matr{K}}_\rmf \vct{\xi} \cdot \vct{\xi}\bigr) 
\end{align*}
%\end{linenomath}
with strict inequality if $\vct{N} \perp \vct{\xi}$. 
\item Suppose that, for all $k\in \mathbb{N}$, there exist $\vct{p}_k \in \Gamma$ and $\smash{ \vct{\xi}_k \in \rmT_{\vct{p}_k}\Gamma }$ such that 
%\begin{linenomath}
\begin{align*}
\vct{\xi}_k \cdot \matr{K}_\Gamma (\vct{p}_k )  \,\vct{\xi}_k \le \frac{1}{k} \abs{\vct{\xi}_k}^2 .
\end{align*}
%\end{linenomath}
W.l.o.g., we assume $\smash{\abs{\vct{\xi}_k} = 1}$ for all $k \in \mathbb{N}$.
Then, with the Bolzano-Weierstraß theorem, there exists a subsequence such that 
%\begin{linenomath}
\begin{align*}
\vct{p}_{k_l} \rightarrow \vct{p} \in \overline{\Gamma} , \qquad \vct{\xi}_{k_l} \rightarrow \vct{\xi} \in \rmT_\vct{p}\partial G  
\end{align*}
%\end{linenomath}
as $l \rightarrow \infty$. In particular, we have 
%\begin{linenomath}
\begin{align*}
\vct{\xi}_{k_l} \cdot \matr{K}_\Gamma (\vct{p}_{k_l} )  \,\vct{\xi}_{k_l} \rightarrow  \vct{\xi} \cdot \matr{K}_\Gamma (\vct{p} )  \,\vct{\xi} = 0
\end{align*}
%\end{linenomath}
as $l \rightarrow \infty$, which is a contradiction to~(i). \qedhere
\end{enumerate}
\end{proof}

Further, the following strong convergence result holds true.
\begin{theorem} 
Let $\alpha = -1 $ and $\beta \ge -1$. Then, we have strong convergence 
%\begin{linenomath}
\begin{subequations}
\begin{alignat}{2}
\hat{p}_\pm^{\epsilon_k} &\rightarrow \hat{p}^\ast_\pm \quad\enspace &&\text{in } H^1 (\Omega_\pm^0 ) , \\
\hat{p}_\mathrm{f}^{\epsilon_k} &\rightarrow \hat{p}^\ast_\mathrm{f} \quad\enspace &&\text{in } H^1 (\Gamma_a ) , \\
\epsilon_k^{-1} \nablaN \hat{p}_\rmf^{\epsilon_k} &\rightarrow \zeta^\ast \vct{N} \quad\enspace &&\text{in } \vct{L}^2 (\Gamma_a ) 
\end{alignat}%
\label{eq:strongconv_aeqm1}%
\end{subequations}%
%\end{linenomath}
as $k\rightarrow \infty$, where $\smash{\zeta^\ast \in L^2 (\Gamma_a ) }$ is given by \cref{eq:zetalimit3}.
Moreover, if $\smash{\matr{K}_\Gamma}$ is uniformly elliptic on~$\rmT\Gamma$,  $\smash{( \hat{p}^\ast_+ , \hat{p}^\ast_- , \mathfrak{A}_\Gamma \hat{p}_\mathrm{f}^\ast ) \in \Phi_\mathrm{II}^0}$ is the unique weak solution of the problem in \cref{eq:weaklimit_aeqm1} and the strong convergence in \cref{eq:strongconv_aeqm1} holds for the whole sequence~$\smash{ \{ \hat{p}_i^\epsilon \}_{\epsilon\in (0,\hat{\epsilon}]}}$, $\smash{i\in\{ + , - , \mathrm{f}\}}$.
\end{theorem}
\begin{proof}
First, we define the norm 
%\begin{linenomath}
\begin{multline*}
\normsizeiii{\big}{( \phi_+ , \phi_- , \phi_\rmf ,\, \zeta ) }^2 := \sum_{i=\pm} \bigl( \matr{K}_i^0 \nabla \phi_i , \nabla \phi_i \bigr)_{\vct{L}^2 (\Omega_i^0 ) } \\
+ \bigl( \hat{\matr{K}}_\rmf \bigl[ \nablaGamma \phi_\rmf + \zeta \vct{N} \bigr] , \bigl[ \nablaGamma \phi_\rmf  + \zeta \vct{N} \bigr] \bigr)_{\vct{L}^2 (\Gamma_a ) } + \bigl( \hat{\matr{K}}_\rmf \nablaN \phi_\rmf , \nablaN \phi_\rmf \bigr)_{L^2 (\Gamma_a ) } 
\end{multline*}
%\end{linenomath}
on $\smash{\Phi \times L^2 (\Gamma_a )}$. 
Then, with \Cref{lem:apriori_l2_2}, it is easy to see that the norm~$\smash{\normiii{\cdot}}$ is equivalent to the product norm on $\smash{\Phi \times L^2 (\Gamma_a ) \subset H^1 (\Omega_+^0 ) \times H^1 (\Omega_-^0 ) \times H^1 (\Gamma_a ) \times L^2 (\Gamma_a ) }$.
With \Cref{lem:Roperators}, \Cref{prop:apriori_l2_3}, and the \cref{eq:rescaledweakdarcyeps,eq:limApm2}, we find 
%\begin{linenomath}
\begin{multline*}
\normsizeiii{\big}{\bigl( \hat{p}_+^{\epsilon_k}  , \hat{p}_-^{\epsilon_k}  , \hat{p}_\rmf^{\epsilon_k}  , \epsilon_k^{-1} \partial_{\theta_n} \hat{p}_\rmf^{\epsilon_k }   \bigr) }^2 = \sum_{i=\pm} \mathcal{A}_i^{\epsilon_k} (\hat{p}_i^{\epsilon_k} , \hat{p}_i^{\epsilon_k} ) + \mathcal{A}_\rmf ( \hat{p}_\rmf^{\epsilon_k} , \hat{p}_\rmf^{\epsilon_k } ) + \smallo (\epsilon_k ) \\
= \bigl[1 + \mathcal{O}(\epsilon_k ) \bigr] \biggl[ \sum_{i=\pm } \bigl( q_i^0 , \hat{p}_i^{\epsilon_k} \bigr)_{L^2 (\Omega_i^0 ) } + \epsilon_k^{\beta + 1 } \bigl( \hat{q}_\mathrm{f} , \hat{p}_\mathrm{f}^{\epsilon_k} \bigr)_{L^2 (\Gamma_a )}  \biggr] + \smallo (\epsilon_k ) .
\end{multline*}
%\end{linenomath}
Thus, with the \Cref{prop:convergence} and \Cref{thm:aeqm1}, we obtain 
%\begin{linenomath}
\begin{multline*}
\lim_{k \rightarrow \infty } \normsizeiii{\big}{\bigl( \hat{p}_+^{\epsilon_k}  , \hat{p}_-^{\epsilon_k}  , \hat{p}_\rmf^{\epsilon_k}  , \epsilon_k^{-1} \partial_{\theta_n} \hat{p}_\rmf^{\epsilon_k }   \bigr) }^2 = \sum_{i = \pm} \bigl( q_i^0 , \hat{p}_i^\ast \bigr)_{L^2 (\Omega_i^0 )}  + \bigl( aq_\Gamma , \mathfrak{A}_\Gamma \hat{p}_\mathrm{f}^\ast \bigr)_{L^2(\Gamma )} \\
= \sum_{i = \pm} \bigl( \matr{K}_i^0 \nabla \hat{p}_i^\ast, \nabla \hat{p}_i^\ast \bigr)_{\vct{L}^2 (\Omega_i^0 ) } +  \bigl( a \matr{K}_\Gamma \nablaGamma \bigl[ \mathfrak{A}_\Gamma p_\mathrm{f}^\ast \bigr]   ,  \nablaGamma \bigl[ \mathfrak{A}_\Gamma p_\mathrm{f}^\ast \bigr]   \bigr)_{\vct{L}^2 ( \Gamma  )} .
\end{multline*}
%\end{linenomath}
Additionally, with the \cref{eq:KGamma,eq:zetalimit3} and \Cref{prop:pconst}, it is
%\begin{linenomath}
\begin{align*}
\bigl( a \matr{K}_\Gamma \nablaGamma \bigl[ \mathfrak{A}_\Gamma p_\mathrm{f}^\ast \bigr]   ,  \nablaGamma \bigl[ \mathfrak{A}_\Gamma p_\mathrm{f}^\ast \bigr]   \bigr)_{\vct{L}^2 ( \Gamma  )}  &= \bigl( \hat{\matr{K}}_\rmf \bigl[ \nablaGamma \hat{p}_\rmf^\ast  + \zeta^\ast \vct{N }\bigr] , \bigl[ \nablaGamma \hat{p}_\rmf^\ast + \zeta^\ast \vct{N} \bigr] \bigr)_{\vct{L}^2 (\Gamma_a ) } .
\end{align*}
%\end{linenomath}
Thus, with \Cref{prop:pconst}, we have 
%\begin{linenomath}
\begin{align*}
\lim_{k \rightarrow \infty } \normsizeiii{\big}{( \hat{p}^{\epsilon_k}_+ , \hat{p}_-^{\epsilon_k} , \hat{p}^{\epsilon_k}_\mathrm{f} , \epsilon_k^{-1} \partial_{\theta_n} \hat{p}_\rmf^{\epsilon_k } )  } &= \normsizeiii{\big}{( \hat{p}^\ast_+ , \hat{p}_-^\ast , \hat{p}^\ast_\mathrm{f}, \zeta^\ast ) }.
\end{align*}
%\end{linenomath}

Now, let $\smash{\matr{K}_\Gamma}$ be uniformly elliptic on $\rmT \Gamma$ and $\smash{ (\phi_+ , \phi_- , \phi_\Gamma ) \in \Phi_\mathrm{II}^0}$.
Then, we have
%\begin{linenomath}
\begin{multline*}
\sum_{i = \pm} \bigl( \matr{K}_i^0 \nabla \phi_i, \nabla \phi_i \bigr)_{\vct{L}^2 (\Omega_i^0  ) } +  \bigl( a \matr{K}_\Gamma \nablaGamma  \phi_\Gamma   ,  \nablaGamma \phi_\Gamma \bigr)_{\vct{L}^2 ( \Gamma  )} \\
\gtrsim \sum_{i=\pm} \norm{\nabla \phi_i}^2_{\vct{L}^2 (\Omega_i^0 )} + \norm{\nablaGamma \phi_\Gamma}^2_{\vct{L}^2 (\Gamma_a )} .
\end{multline*}
%\end{linenomath}
Hence, we obtain coercivity on~$\smash{\Phi_\mathrm{II}^0}$ by applying \Cref{lem:apriori_l2_2}. 
Thus, as a consequence of the Lax-Milgram theorem, $\smash{ ( \hat{p}^\ast_+ , \hat{p}^\ast_- , \mathfrak{A}_\Gamma \hat{p}_\mathrm{f}^\ast ) \in \Phi_\mathrm{II}^0}$ is the unique weak solution of the problem in \cref{eq:weaklimit_aeqm1}. 
Further, this implies the convergence of the whole sequence~$\smash{ \{ \hat{p}_i^\epsilon \}_{\epsilon\in (0,\hat{\epsilon}]}}$, $\smash{i\in\{ + , - , \mathrm{f}\}}$, as $\epsilon\rightarrow 0$ since every convergent subsequence has the same limit.
\end{proof}

\subsection{Case III: \texorpdfstring{$\alpha \in (-1, 1)$}{alpha in (-1,1)}}
\label{sec:sec43}
For $\alpha \in (-1, 1)$ and $\beta \ge -1$, the hydraulic conductivities in bulk and fracture are of similar magnitude such that the fracture disappears in the limit $\smash{\epsilon \rightarrow 0}$.
No effect of the fracture conductivity is visible in the limit model and pressure and normal velocity are continuous across the interface~$\Gamma$ (except for source terms if $\beta = -1$).
The strong formulation of the limit problem reads as follows. 

Find $\smash{p_\pm \colon \Omega^0_\pm \rightarrow \mathbb{R}}$ such that
%\begin{linenomath}
\begin{subequations}
\begin{alignat}{3}
-\nabla \cdot \bigl( \matr{K}_\pm^0 \nabla p_\pm \bigr) &= q_\pm^0 \qquad &&\text{in } \Omega_\pm^0 ,  \\
p_+ &= p_-  \qquad &&\text{on } \gamma , \\
 \jumpsize{\big}{\matr{K}^0\nabla p}_\Gamma  + aq_\Gamma &= 0 \qquad &&\text{on } \Gamma ,\\
\matr{K}_+^0 \nabla p_+ \cdot \vct{N} &= \matr{K}_-^0 \nabla p_- \cdot \vct{N} \qquad &&\text{on } \Gamma_0^0   , \\
p_\pm &= 0 \qquad &&\text{on } \rho_{\pm,\mathrm{D}}^0 ,  \\
\matr{K}_\pm^0 \nabla p_\pm \cdot \vct{n} &= 0 \qquad &&\text{on } \rho_{\pm,\mathrm{N}}^0 , 
\end{alignat}%
\label{eq:stronglimit_ainm1p1}%
\end{subequations}%
%\end{linenomath}
where $\smash{ q_\Gamma \in L^2_a (\Gamma ) }$ is defined as in \cref{eq:qgamma}. 

A weak formulation of the system in \cref{eq:stronglimit_ainm1p1} is given by the following problem.

Find $\smash{(p_+ , p_- ) \in \Phi^0_\mathrm{III}}$ such that, for all $\smash{ ( \phi_- , \phi_+ ) \in \Phi_\mathrm{III}^0}$ with $\smash{\phi_\Gamma := \phi_\pm\bigr\vert_{\Gamma}}$,
%\begin{linenomath}
\begin{align}
&\sum_{i = \pm} \bigl( \matr{K}_i^0 \nabla p_i , \nabla \phi_i )_{\vct{L}^2 (\Omega_i^0 ) } = \sum_{i = \pm } \bigl( q_i^0 , \phi_i \bigr)_{L^2 (\Omega_i^0 )} +  \bigl( aq_\Gamma , \phi_\Gamma \bigr)_{L^2(\Gamma ) }. \label{eq:weaklimit_ainm1p1}
\end{align}
%\end{linenomath}
Here, the space~$\Phi^0_\mathrm{III}$ is given by 
%\begin{linenomath}
\begin{align}
\begin{split}
\Phi_\mathrm{III}^0 &:= \Bigl\{ ( \phi_+ , \phi_-  ) \in \bigtimes\nolimits_{i=\pm} H_{0,\rho_{i,\mathrm{D}}^0}^1 (\Omega_i^0 )   \ \Big\vert\  \phi_+\bigr\vert_{\gamma} = \phi_-\bigr\vert_{\gamma} \Bigr\} \cong H^1_{0,\rho_{\mathrm{b},\mathrm{D}}^0} (\Omega ) .
\end{split}
\end{align}
%\end{linenomath}
We now obtain the following convergence results.
\begin{theorem} \label{thm:ainm1p1}
Let $\alpha \in (-1, 1)$ and $\beta \ge -1$. 
Besides, let either $\alpha \le 0$ or assume that \eqref{asm:rhoD} holds.
Then, given the limit functions $\smash{\hat{p}^\ast_\pm\in H^1(\Omega^0_\pm )}$ and $\smash{\hat{p}^\ast_\mathrm{f} \in H^1_{\vct{N}} (\Gamma_a )}$ from \Cref{prop:convergence}, we find that $\smash{ ( \hat{p}^\ast_+ , \hat{p}^\ast_-  ) \in \Phi_\mathrm{III}^0}$ is a weak solution of \cref{eq:weaklimit_ainm1p1}. 
Moreover, we have $\smash{\hat{p}^\ast_\pm = \mathfrak{A}_\Gamma \hat{p}_\mathrm{f}^\ast}$ on~$\Gamma$
and $\smash{\hat{p}_\mathrm{f}^\ast (\vct{p} , \theta_n) = (\mathfrak{A}_\Gamma \hat{p}^\ast_\mathrm{f} ) (\vct{p} )}$ for a.a.~$\smash{(\vct{p} , \theta_n ) \in \Gamma_a }$.
\end{theorem}
\begin{proof}
Take a test function triple $\smash{(\phi_+ , \phi_- , \phi_\mathrm{f}) \in \Phi}$  such that $\smash{\phi_\mathrm{f} (\vct{p} , \theta_n ) = \phi_\Gamma ( \vct{p} )}$ a.e.\ in~$\smash{\Gamma_a }$. 
Then, by inserting $\smash{(\phi_+ , \phi_- , \phi_\mathrm{f})}$ into the transformed weak formulation~\eqref{eq:rescaledweakdarcyeps}, we obtain
%\begin{linenomath}
\begin{align}
\begin{split}
&\sum_{i=\pm} \mathcal{A}_i^{\epsilon_k} ( \hat{p}_i^{\epsilon_k} , \phi_i ) + \epsilon_k^{\alpha + 1 } \bigl( \hat{\matr{K}}_\rmf \invR_\rmf^{\epsilon_k} \nablaGamma \hat{p}_\rmf^{\epsilon_k } , \invR_\rmf^{\epsilon_k } \nablaGamma \phi_\rmf \bigr)_{\vct{L}^2 (\Gamma_a ) } \\
&\hspace{0.5cm}+ \epsilon_k^\alpha \bigl( \hat{\matr{K}}_\rmf \nablaN \hat{p}_\rmf^{\epsilon_k } , \invR_\rmf^{\epsilon_k} \nablaGamma \phi_\rmf  \bigr)_{\vct{L}^2 (\Gamma_a ) } \\
&\hspace{2.25cm}= \bigl[ 1 + \mathcal{O}(\epsilon ) \bigr] \biggl[ \sum_{i = \pm} \bigl( q_i^0 , \phi_i \bigr)_{L^2 (\Omega_i^0 )} + \epsilon_k^{\beta + 1 } \bigl( a q_\Gamma , \phi_\Gamma \bigr)_{L^2 (\Gamma )} \biggr] . 
\end{split}%
\label{eq:proof_ainm1p1_1}%
\end{align}%
%\end{linenomath}
Further, with \Cref{lem:Roperators} and \Cref{prop:apriori_l2_3}, we have 
%\begin{linenomath}
\begin{align*}
\epsilon^{\alpha + 1 } \abssize{\Big}{ \bigl( \hat{\matr{K}}_\rmf \invR_\rmf^{\epsilon} \nablaGamma \hat{p}_\rmf^{\epsilon } , \invR_\rmf^{\epsilon } \nablaGamma \phi_\rmf \bigr)_{\vct{L}^2 (\Gamma_a ) } } 
&\lesssim \epsilon^{\frac{\alpha + 1}{2}} \normsize{\big}{\nablaGamma \phi_\rmf }_{\vct{L}^2(\Gamma_a )}, \\
\epsilon^\alpha \abssize{\Big}{ \bigl( \hat{\matr{K}}_\rmf \nablaN \hat{p}_\rmf^{\epsilon } , \invR_\rmf^{\epsilon } \nablaGamma \phi_\rmf  \bigr)_{\vct{L}^2 (\Gamma_a ) }} 
&\lesssim \epsilon^\frac{\alpha + 1}{2} \normsize{\big}{\nablaGamma \phi_\mathrm{f} }_{\vct{L}^2(\Gamma_a )}
\end{align*}
%\end{linenomath}
if $\epsilon$ is sufficiently small.
Thus, by using  \Cref{prop:convergence} and \Cref{lem:limApm} and letting $k\rightarrow \infty$ in \cref{eq:proof_ainm1p1_1}, it follows that the limit solution pair~$ ( \hat{p}^\ast_+ , \hat{p}^\ast_-  )$ satisfies the weak formulation~\eqref{eq:weaklimit_ainm1p1}. 
Besides, with \Cref{prop:pressurecontinuity}, we have $\smash{(\hat{p}_+^\ast , \hat{p}_-^\ast ) \in \Phi^0_\mathrm{III}}$.  
\end{proof}

\begin{theorem} \label{thm:strong_ainl1p1}
Let $\alpha \in (-1,1 )$ and $\beta \ge -1$. 
Then, given the assumption~\eqref{asm:rhoD}, we have strong convergence 
%\begin{linenomath}
\begin{subequations}
\begin{alignat}{2}
\hat{p}_\pm^{\epsilon } &\rightarrow \hat{p}_\pm^\ast \quad\enspace &&\text{in } H^1 (\Omega_\pm^0 ) , \\
\hat{p}_\mathrm{f}^\epsilon &\rightarrow \hat{p}_\mathrm{f}^\ast \quad\enspace &&\text{in } H^1_{\vct{N}} (\Gamma_a ) 
\end{alignat}%
\label{eq:strongconv_ainm1p1}%
\end{subequations}%
%\end{linenomath}
as $\epsilon \rightarrow 0$ for the whole sequence~$\smash{\{ \hat{p}_i^\epsilon \}_{\epsilon\in (0,\hat{\epsilon}]}}$, $i \in \{ +,-,\mathrm{f}\}$ .
Besides, $\smash{( \hat{p}_+^\ast , \hat{p}_-^\ast ) \in \Phi_\mathrm{III}^0}$ is the unique weak solution of \cref{eq:weaklimit_ainm1p1}.
\end{theorem}
\begin{proof}
As consequence of the Lax-Milgram theorem, the problem in \cref{eq:weaklimit_ainm1p1} has a unique weak solution. 
Thus, the weak convergence~\eqref{eq:conv_A} holds for the whole sequence~$\smash{\{ \hat{p}^\epsilon_\pm \}_{\epsilon \in (0, \hat{\epsilon}]}}$.
This follows from \Cref{prop:apriori_l2_3} and the fact that every weakly convergent subsequent has the same limit. 
Besides, with the \Cref{prop:apriori_l2_3,prop:pconst,prop:pressurecontinuity},  the weak convergence~\eqref{eq:conv_C} is satisfied for the whole sequence~$\smash{\{ \hat{p}_\mathrm{f}^\ast \}_{\epsilon \in (0, \hat{\epsilon}]} }$. 

Next, we equip the space~$\smash{\Phi^\ast}$ with the norm~$\smash{\normiii{\cdot}}$ defined by 
%\begin{linenomath}
\begin{align}
\normsizeiii{\big}{( \phi_+ , \phi_- , \phi_\mathrm{f} ) }^2 &:= \sum_{i=\pm } \bigl( \matr{K}_i^0 \nabla \phi_i , \nabla \phi_i \bigr)_{\vct{L}^2 (\Omega_i^0 )} + \bigl( \hat{\matr{K}}_\rmf \nablaN \phi_\rmf , \nablaN \phi_\rmf \bigr)_{\vct{L}^2 (\Gamma_a ) } , \label{eq:normast}
\end{align}
%\end{linenomath}
which, as a consequence of \Cref{lem:apriori_l2_2}, is equivalent to the usual product norm on~$\smash{ \Phi^\ast \subset H^1(\Omega_+^0 ) \times H^1 (\Omega_-^0 ) \times H^1_{\vct{N}} (\Gamma_a )}$.
Besides, with \Cref{prop:apriori_l2_3}, we have 
%\begin{linenomath}
\begin{align*}
\bigl( \hat{\matr{K}}_\rmf \nablaN \hat{p}_\rmf^\epsilon , \nablaN \hat{p}_\rmf^\epsilon  \bigr)_{\vct{L}^2 (\Gamma_a ) } \lesssim \normsize{\big }{\nablaN \hat{p}_\rmf^\epsilon }^2_{\vct{L}^2 (\Gamma_a ) } = \mathcal{O} (\epsilon^{1-\alpha} ) .
\end{align*}
%\end{linenomath}
Thus, using the \cref{eq:rescaledweakdarcyeps,eq:limApm2} and $\smash{\mathcal{A}_\rmf^\epsilon ( \hat{p}^\epsilon_\rmf , \hat{p}_\rmf^\epsilon ) \ge 0}$, we find 
%\begin{linenomath}
\begin{align*}
\normsizeiii{\big}{( \hat{p}_+^\epsilon , \hat{p}_-^\epsilon , \hat{p}_\mathrm{f}^\epsilon )}^2 &\le \sum_{i=\pm} \mathcal{A}_i^\epsilon (\hat{p}_i^\epsilon , \hat{p}_i^\epsilon ) + \mathcal{A}_\mathrm{f}^\epsilon (\hat{p}_\mathrm{f}^\epsilon , \hat{p}_\mathrm{f}^\epsilon ) + \smallo (\epsilon ) \\
&= \bigl[ 1 + \mathcal{O}(\epsilon ) \bigr] \biggl[ \sum_{i=\pm} \bigl( q_i^0 , \hat{p}_i^\epsilon \bigr)_{L^2 (\Omega_i^0 ) } + \epsilon^{\beta + 1} \bigl( \hat{q}_\mathrm{f} , \hat{p}_\mathrm{f}^\epsilon \bigr)_{L^2 (\Gamma_a ) } \biggr] +\smallo (\epsilon )  .
\end{align*}
%\end{linenomath}
Further, \Cref{thm:ainm1p1} yields
%\begin{linenomath}
\begin{align*}
\limsup_{\epsilon \rightarrow 0 } \normsizeiii{\big}{( \hat{p}_+^\epsilon , \hat{p}_-^\epsilon , \hat{p}_\mathrm{f}^\epsilon )}^2 &\le \sum_{i = \pm } \bigl( q_i^0 , \hat{p}_i^\ast \bigr)_{L^2 (\Omega_i^0 )} +  \bigl( aq_\Gamma , \mathfrak{A}_\Gamma p_\mathrm{f}^\ast \bigr)_{L^2(\Gamma ) } \\
&= \sum_{i = \pm} \bigl( \matr{K}_i^0 \nabla \hat{p}_i^\ast , \nabla \hat{p}^\ast_i )_{\vct{L}^2 (\Omega_i^0 ) } = \normsizeiii{\big}{( \hat{p}_+^\ast , \hat{p}_-^\ast , \hat{p}_\mathrm{f}^\ast )}^2 .
\end{align*}
%\end{linenomath}
With the weak lower semicontinuity of the norm, we now have
%\begin{linenomath}
\begin{align*}
\lim_{\epsilon \rightarrow 0 }\normsizeiii{\big}{( \hat{p}_+^\epsilon , \hat{p}_-^\epsilon , \hat{p}_\mathrm{f}^\epsilon )} = \normsizeiii{\big}{( \hat{p}_+^\ast , \hat{p}_-^\ast , \hat{p}_\mathrm{f}^\ast )} . \tag*{\qedhere}
\end{align*}
%\end{linenomath}
\end{proof}

\subsection{Case IV: \texorpdfstring{$\alpha = 1$}{alpha = 1}}
\label{sec:sec44}
For $\alpha = 1$ and $\beta \ge -1$, the fracture becomes a permeable barrier in limit~$\epsilon \rightarrow 0$ with a jump of pressure heads across the interface~$\Gamma$ but continuous normal velocity (except for source terms).

In the following, we will derive two different limit models for $\alpha = 1$ and $\beta \ge -1$.
First, in \Cref{sec:aeqp1_coupled}, we obtain a coupled limit problem, where the pressure head~$\smash{\hat{p}_\rmf^\ast}$ in the fracture satisfies a parameter-dependent Darcy-type ODE inside the full-dimensional fracture domain~$\Gamma_a$. 
The ODE is formulated with respect to the normal coordinate~$\theta_n$, while the tangential coordinate~$\vct{p}$ acts as a parameter. 
This resembles the limit problem in~\cite{kumar20} for Richards equation with the respective scaling of hydraulic conductivities. 
However, in \Cref{sec:aeqp1_decoupled}, it then turns out that the bulk problem can be solved independently from the fracture problem. 
In the decoupled bulk limit problem, the jump of pressure heads across the interface~$\Gamma$ scales with an effective hydraulic conductivity, that is defined as a non-trivial mean value of the fracture conductivity in normal direction and reminds of a result from homogenization theory. 
In particular, if one is still interested in the fracture solution, it is possible to first solve the decoupled bulk limit problem in \Cref{sec:aeqp1_decoupled}, which will then provide the boundary conditions to solve the ODE for the fracture pressure head in \Cref{sec:aeqp1_coupled}.

\subsubsection{Coupled Limit Problem} \label{sec:aeqp1_coupled}
The strong formulation of the coupled limit problem for $\alpha = 1$ and $\beta \ge -1$ reads as follows. 

Find $\smash{p_\pm \colon \Omega^0_\pm \rightarrow \mathbb{R}}$ and $\smash{p_\rmf \colon \Gamma_a \rightarrow \mathbb{R}}$ such that 
%\begin{linenomath}
\begin{subequations}
\begin{alignat}{2}
-\nabla \cdot \bigl( \matr{K}_\pm^0 \nabla p_\pm \bigr) &= q_\pm^0 \qquad &&\text{in } \Omega_i^0 , \ \\
- \partial_{\theta_n }\bigl( \hat{K}_\rmf^\perp  \partial_{\theta_n } p_\rmf \bigr) &= q_\rmf^\ast   \qquad &&\text{in } \Gamma_a  , \\ 
p_\pm &= \mathfrak{T}_\pm p_\rmf \qquad &&\text{on } \Gamma , \\
 \matr{K}_\pm^0 \nabla p_\pm^0 \cdot \vct{N} &= \mathfrak{T}_\pm \bigl( \hat{K}_\rmf^\perp \partial_{\theta_n} p_\rmf \bigr) \qquad &&\text{on } \Gamma , \\
p_+ &= p_- \qquad &&\text{on } \Gamma_0^0 , \\ 
\matr{K}_+^0 \nabla p_+ \cdot \vct{N} &= \matr{K}_-^0 \nabla p_- \cdot \vct{N} \qquad &&\text{on } \Gamma_0^0 , \\
p_\pm &= 0 \qquad &&\text{on } \rho_{\pm,\mathrm{D}}^0 ,  \\
\matr{K}_\pm^0 \nabla p_\pm \cdot \vct{n} &= 0 \qquad &&\text{on } \rho_{\pm,\mathrm{N}}^0 ,
\end{alignat}%
\label{eq:stronglimit_aeqp1_coupled}%
\end{subequations}%
%\end{linenomath}
where $\smash{\hat{q}^\ast_\rmf \in L^2 (\Gamma_a ) }$ and $\smash{\hat{K}_\rmf^\perp \in L^\infty (\Gamma_a ) }$ are defined by
%\begin{linenomath}
\begin{align}
\hat{q}_\rmf^\ast (\vct{p} , \theta_n ) := \begin{cases} 
\hat{q}_\rmf (\vct{p} , \theta_n ) &\text{if } \beta = -1 ,\\
0 &\text{if } \beta > -1 ,
\end{cases} \\
\hat{K}_\rmf^\perp (\vct{p} , \theta_n ) := \matr{\hat{K}}_\rmf (\vct{p} , \theta_n ) \vct{N} ( \vct{p} ) \cdot \vct{N} (\vct{p} ) .
\end{align}
%\end{linenomath}
A weak formulation of the system in \cref{eq:stronglimit_aeqp1_coupled} is given by the following problem.

Find $\smash{ ( p_+ , p_- , p_\rmf ) \in \Phi^\ast}$ such that, for all $\smash{ ( \phi_+ , \phi_- , \phi_\mathrm{f} ) \in \Phi^\ast}$, 
%\begin{linenomath}
\begin{align}
\begin{split}
&\sum_{i = \pm } \bigl( \matr{K}_i^0 \nabla p_i , \nabla \phi_i )_{\vct{L}^2 (\Omega_i^0  ) } + \bigl( \hat{\matr{K}}_\rmf \nablaN p_\rmf , \nablaN \phi_\rmf \bigr)_{\vct{L}^2 (\Gamma_a )} \\
&\hspace{5cm} =  \sum_{i = \pm }\bigl( q_i^0 , \phi_i \bigr)_{L^2 (\Omega_i^0 )} + \bigl( \hat{q}^\ast_\rmf , \phi_\rmf \bigr)_{L^2 (\Gamma_a ) } .
\end{split}%
\label{eq:weaklimit_aeqp1_coupled}%
\end{align}%
%\end{linenomath}
We obtain the following convergence results.
\begin{theorem} \label{thm:aeqp1_coupled}
Let $\alpha = 1$ and $\beta \ge -1$. 
Then, given the assumption~\eqref{asm:rhoD}, the triple $\smash{( \hat{p}_+^\ast , \hat{p}_-^\ast , \hat{p}_\rmf^\ast ) \in \Phi^\ast }$ is a weak solution of problem~\eqref{eq:weaklimit_aeqp1_coupled}, where $\smash{\hat{p}_\pm^\ast \in H^1 (\Omega_\pm^0 ) }$ and $\smash{\hat{p}_\rmf^\ast \in H^1_\vct{N} (\Gamma_a ) }$ denote the limit functions from \Cref{prop:convergence}.  
\end{theorem}
\begin{proof} 
According to \Cref{prop:apriori_l2_3}, we have 
%\begin{linenomath}
\begin{align*}
\norm{\hat{p}^\epsilon_\mathrm{f}}_{H^1_\vct{N}(\Gamma_a )} + \epsilon \norm{\nablaGamma \hat{p}^\epsilon_\mathrm{f}}_{\vct{L}^2(\Gamma_a )} \lesssim 1
\end{align*}
%\end{linenomath}
and hence 
%\begin{linenomath}
\begin{align}
\epsilon_k \hat{p}_\rmf^{\epsilon_k }  \rightarrow 0 \enspace\ \text{in } H^1_\vct{N} (\Gamma_a ) , \qquad \epsilon_k \nablaGamma \hat{p}_\rmf^{\epsilon_k } \rightharpoonup \vct{0} \enspace\ \text{in } \vct{L}^2 (\Gamma_a ) \label{eq:epspf_conv}
\end{align}
%\end{linenomath}
as $k\rightarrow \infty$.  
As a result, we have 
%\begin{linenomath}
\begin{multline*}
\epsilon_k \bigl( \hat{\matr{K}}_\rmf \invR^{\epsilon_k}_\rmf \nablaGamma \hat{p}_\rmf^{\epsilon_k} , \nablaN \phi_\rmf \bigr)_{L^2 (\Gamma_a ) } \\
= \epsilon_k \bigl( \hat{\matr{K}}_\rmf \bigl[ \invR^{\epsilon_k}_\rmf - \operatorname{id}_{\rmT \Gamma } \bigr] \nablaGamma \hat{p}_\rmf^{\epsilon_k} , \nablaN \phi_\rmf \bigr)_{L^2 (\Gamma_a ) } + \epsilon_k \bigl( \hat{\matr{K}}_\rmf \nablaGamma \hat{p}_\rmf^{\epsilon_k} , \nablaN \phi_\rmf \bigr)_{L^2 (\Gamma_a ) } ,
\end{multline*}
%\end{linenomath}
where, as $k\rightarrow \infty$, the first term vanishes with \Cref{lem:Roperators} and the second term with \cref{eq:epspf_conv}.
Thus, with the \Cref{prop:apriori_l2_3,prop:convergence} and the \Cref{lem:Roperators,lem:limApm}, we conclude that $\smash{(\hat{p}^\ast_+ , \hat{p}_-^\ast , \hat{p}_\mathrm{f}^\ast) \in \Phi^\ast}$ solves \cref{eq:weaklimit_aeqp1_coupled} by taking the limit $\smash{k \rightarrow \infty}$ in the transformed weak formulation~\eqref{eq:rescaledweakdarcyeps}.
\end{proof}

\begin{theorem}
Let $\alpha =-1$ and $\beta \ge -1$. 
Then, given the assumption~\eqref{asm:rhoD}, we have strong convergence 
%\begin{linenomath}
\begin{subequations}
\begin{alignat}{2}
\hat{p}_\pm^{\epsilon } &\rightarrow \hat{p}_\pm^\ast \quad\enspace &&\text{in } H^1 (\Omega_\pm^0 ) , \\
\hat{p}_\mathrm{f}^\epsilon &\rightarrow \hat{p}_\mathrm{f}^\ast \quad\enspace &&\text{in } H^1_{\vct{N}} (\Gamma_a ) ,\\
\epsilon \nablaGamma \hat{p}_\rmf^\epsilon &\rightarrow \vct{0} \quad\enspace &&\text{in } \vct{L}^2 (\Gamma_a ) 
\end{alignat}%
\label{eq:strong_aeqp1}%
\end{subequations}%
%\end{linenomath}
as $\epsilon \rightarrow 0$ for the whole sequence~$\smash{\{ \hat{p}_i^\epsilon \}_{\epsilon\in (0,\hat{\epsilon}]}}$, $i \in \{ +,-,\mathrm{f}\}$.
Besides, we find that $\smash{( \hat{p}_+^\ast , \hat{p}_-^\ast , \hat{p}_\rmf^\ast ) \in \Phi^\ast}$ is the unique weak solution of the problem in \cref{eq:weaklimit_aeqp1_coupled}.
\end{theorem}
\begin{proof} 
Clearly, the bilinear form of the weak formulation~\eqref{eq:weaklimit_aeqp1_coupled} is continuous and coercive with respect to the norm  defined by~\cref{eq:normast}. 
Thus, with the Lax-Milgram theorem, we obtain that $\smash{( \hat{p}^\ast_+ , \hat{p}^\ast_- , \hat{p}^\ast_\mathrm{f} ) \in \Phi^\ast}$ is the unique solution of \cref{eq:weaklimit_aeqp1_coupled}. 
As a result, every weakly convergent subsequence has the same limit and hence, with \Cref{prop:apriori_l2_3}, the weak convergence statements~\eqref{eq:conv_A} and~\eqref{eq:conv_C} in \Cref{prop:convergence} hold for the whole sequence~$\smash{\{ \hat{p}_i^\epsilon\}_{\epsilon\in (0, \hat{\epsilon}]}}$, $i\in\{ +,-,\mathrm{f}\}$.

Further, we define the space 
%\begin{linenomath}
\begin{align*}
\vct{L}^2_\Gamma ( \Gamma_a ) := \bigl\{ \vct{\xi}\in \vct{L}^2 (\Gamma_a ) \ \big\vert \  \vct{\xi} (\vct{p} , \theta_n ) \cdot \vct{N}(\vct{p} ) = 0 \ \text{ for a.a. } (\vct{p}, \theta_n ) \in \Gamma_a  \bigr\} 
\end{align*}
%\end{linenomath}
and equip the product space $\smash{\Phi^\ast \times \vct{L}^2_\Gamma ( \Gamma_a ) }$ with the norm 
%\begin{linenomath}
\begin{multline*}
\normsizeiii{\big}{ \bigl( \phi_+ , \phi_- , \phi_\rmf ,\, \vct{\xi} \bigr) }^2 \\ := \sum_{i=\pm} \bigl( \matr{K}_i^0 \nabla \phi_i ,  \nabla \phi_i \bigr)_{\vct{L}^2(\Omega_i^0 ) } + \bigl( \hat{\matr{K}}_\rmf \bigl[ \nablaN \phi_\rmf + \vct{\xi} \bigr] , \bigl[ \nablaN \phi_\rmf + \vct{\xi} \bigr] \bigr)_{\vct{L}^2 (\Gamma_a ) } .
\end{multline*}
%\end{linenomath}
Then, with \Cref{lem:apriori_l2_2}, it is easy to see that the norm~$\smash{\normiii{ \cdot } }$ is equivalent to the standard product norm 
on $\smash{\Phi^\ast \times \vct{L}^2_\Gamma ( \Gamma_a )  }$.
Moreover, with \Cref{lem:Roperators} and the \cref{eq:rescaledweakdarcyeps,eq:limApm2}, we have 
%\begin{linenomath}
\begin{multline*}
\normsizeiii{\big}{ \bigl( \hat{p}_+^\epsilon  , \hat{p}_-^\epsilon  , \hat{p}_\rmf^\epsilon ,\, \epsilon\nabla_\Gamma \hat{p}_\rmf^\epsilon  \bigr) }^2 = \sum_{i=\pm } \mathcal{A}_i^\epsilon ( \hat{p}_i^\epsilon , \hat{p}_i^\epsilon \bigr) + \mathcal{A}_\rmf ( \hat{p}_\rmf^\epsilon , \hat{p}_\rmf^\epsilon ) + \smallo (\epsilon ) \\
= \bigl[1 + \mathcal{O}(\epsilon ) \bigr] \biggl[ \sum_{i=\pm } \bigl( q_i^0 , \hat{p}_i^{\epsilon} \bigr)_{L^2 (\Omega_i^0 ) } + \epsilon^{\beta + 1 } \bigl( \hat{q}_\mathrm{f} , \hat{p}_\mathrm{f}^{\epsilon} \bigr)_{L^2 (\Gamma_a )}  \biggr] + \smallo (\epsilon ) .
\end{multline*}
%\end{linenomath}
Thus, with \Cref{prop:convergence} and \Cref{thm:aeqp1_coupled}, we find 
%\begin{linenomath}
\begin{align*}
\lim_{\epsilon \rightarrow 0 } \normsizeiii{\big}{ \bigl( \hat{p}_+^\epsilon  , \hat{p}_-^\epsilon  , \hat{p}_\rmf^\epsilon ,\, \epsilon\nabla_\Gamma \hat{p}_\rmf^\epsilon  \bigr) }^2 &= \sum_{i = \pm }\bigl( q_i^0 , \hat{p}^\ast_i \bigr)_{L^2 (\Omega_i^0 )} + \bigl( \hat{q}^\ast_\rmf , \hat{p}^\ast_\rmf \bigr)_{L^2 (\Gamma_a ) } \\
= \normsizeiii{\big}{ \bigl( \hat{p}_+^\ast  , \hat{p}_-^\ast  , \hat{p}_\rmf^\ast , \vct{0}  \bigr) }^2 . \tag*{\qedhere}
\end{align*}
%\end{linenomath}
\end{proof}

\subsubsection{Decoupled Limit Problem} \label{sec:aeqp1_decoupled}
Starting from the coupled limit problem~\eqref{eq:weaklimit_aeqp1_coupled}, we will subsequently derive a decoupled limit problem for the bulk solution only.
The strong formulation of the decoupled bulk limit problem reads as follows.

Find $\smash{p_\pm \colon \Omega^0_\pm \rightarrow \mathbb{R}}$ such that 
%\begin{linenomath}
\begin{subequations}
\begin{alignat}{2}
-\nabla \cdot \bigl( \matr{K}_\pm^0 \nabla p_\pm \bigr) &= q_\pm^0 \qquad &&\text{in } \Omega_i^0 , \ \\
p_+ &= p_- \qquad &&\text{on } \Gamma_0^0 , \\ 
\matr{K}_+^0 \nabla p_+ \cdot \vct{N} &= \matr{K}_-^0 \nabla p_- \cdot \vct{N} \qquad &&\text{on } \Gamma_0^0 , \\
\jumpsize{\big}{\matr{K}^0 \nabla p}_\Gamma + aq_\Gamma &= 0 \qquad &&\text{on } \Gamma , \\
\matr{K}_+^0 \nabla p_+ \cdot \vct{N} &= K_\Gamma^\perp \bigl(  \jump{p}_\Gamma - a Q_\Gamma \bigr) \qquad &&\text{on } \Gamma , \\
p_\pm &= 0 \qquad &&\text{on } \rho_{\pm,\mathrm{D}}^0 ,  \\
\matr{K}_\pm^0 \nabla p_\pm \cdot \vct{n} &= 0 \qquad &&\text{on } \rho_{\pm,\mathrm{N}}^0 ,
\end{alignat}%
\label{eq:stronglimit_aeqp1}%
\end{subequations}%
%\end{linenomath}
where $\smash{q_\Gamma \in L^2_a (\Gamma )}$ is given by \cref{eq:qgamma}.
$\smash{Q_\Gamma \in L^2 (\Gamma )}$ 
and the effective hydraulic conductivity $\smash{K_\Gamma^\perp \colon \Gamma \rightarrow \mathbb{R}}$ with $\smash{a K_\Gamma^\perp  \in L^\infty(\Gamma )}$  are defined by
%\begin{linenomath}
\begin{subequations}
\begin{align}
Q_\Gamma (\vct{p}  ) &:= \begin{cases} \bigl( \mathfrak{A}_\Gamma \hat{Q}_\mathrm{f} \bigr)  (\vct{p} )   &\text{if } \beta = -1 , \\
0 &\text{if } \beta > - 1 ,  \end{cases} \\
\hat{Q}_\rmf (\vct{p} , \theta_n ) &:=  \hat{q}_\rmf ( \vct{p} , \theta_n ) \int_{-a_-(\vct{p} )}^{ \theta_n }  \bigl[ \hat{K}_\rmf^\perp \bigr]^{-1} (\vct{p}, \bar{\theta}_n) \,\D \bar{\theta}_n ,
\end{align}
\end{subequations}
%\end{linenomath}
\vspace{-0.25cm}
\begin{align}
\label{eq:Kperp} K_\Gamma^\perp (\vct{p} ) &:= \bigl[ a(\vct{p} ) \mathfrak{A}_\Gamma \bigl( [\hat{K}_\rmf^\perp ]^{-1} \bigr) (\vct{p}) \bigr]^{-1} .
\end{align}
A weak formulation of the system in \cref{eq:stronglimit_aeqp1} is given by the following problem.
\vspace{8pt}

Find $ (p_+ , p_-  ) \in \Phi^0_\mathrm{IV}$ such that, for all $\smash{(\phi_+ , \phi_- ) \in \Phi_\mathrm{IV}^0}$,
%\begin{linenomath}
\begin{align}
\begin{split}
&\sum_{i = \pm } \bigl( \matr{K}_i^0 \nabla p_i , \nabla \phi_i )_{\vct{L}^2 (\Omega_i^0  ) } + \bigl( K_\Gamma^\perp \jump{p}_\Gamma , \jump{\phi}_\Gamma \bigr)_{L^2 (\Gamma )}  \\
&\hspace{1.6cm} =  \sum_{i = \pm }\bigl( q_i^0 , \phi_i \bigr)_{L^2 (\Omega_i^0 )} + \bigl( aq_\Gamma , \phi_- \bigr)_{L^2(\Gamma ) } + \bigl( aK_\Gamma^\perp Q_\Gamma , \jump{\phi}_\Gamma \bigr)_{L^2 (\Gamma ) }.
\end{split}%
\label{eq:weaklimit_aeqp1}%
\end{align}%
%\end{linenomath}
Here, the space~$\Phi^0_\mathrm{IV}$ is given by 
%\begin{linenomath}
\begin{align}
\Phi_\mathrm{IV}^0 &:= \Bigl\{ (\phi_+ , \phi_- ) \in \bigtimes\nolimits_{i=\pm } H^1_{0,\rho_{i,\mathrm{D}}^0} (\Omega_i^0 ) \ \Big\vert \ \jump{\phi}_\Gamma \in L^2_{a^{-1}} (\Gamma ) , \ \phi_+\bigr\vert_{\Gamma_0^0} = \phi_-\bigr\vert_{\Gamma_0^0}  \Bigr\} .
\end{align} 
%\end{linenomath}
We require the following auxiliary result.
\begin{lemma} \label{lem:PhiIV}
The map 
%\begin{linenomath}
\begin{align}
( \phi_+ , \phi_- , \phi_\rmf )  \mapsto ( \phi_+ , \phi_- )
\end{align}
%\end{linenomath}
defines a continuous embedding~$\smash{\Phi^\ast \hookrightarrow \Phi_\mathrm{IV}^0}$.
\end{lemma}
\begin{proof}
With \Cref{lem:traceineq}, we have 
%\begin{linenomath}
\begin{align*}
\jump{\phi}_\Gamma^2 (\vct{p} )&= \biggl[ \int_{a_-(\vct{p})}^{a_+(\vct{p} ) } \partial_{\theta_n} \phi_\mathrm{f} (\vct{p}, \theta_n )  \,\D \theta_n \biggr]^2  \le a(\vct{p}) \int_{a_-(\vct{p} )}^{a_+(\vct{p} ) } \bigl[ \partial_{\theta_n }  \phi_\mathrm{f}(\vct{p}, \theta_n )   \bigr]^2  \,\D \theta_n 
\end{align*}
%\end{linenomath}
for a.a.\ $\smash{(\vct{p}, \theta_n ) \in \Gamma_a }$.
Thus, an additional integration on~$\Gamma$ yields
%\begin{linenomath}
\begin{align*}
\normsize{\big}{\jump{\phi}_\Gamma}_{L^2_{a^{-1}}(\Gamma ) } \le \norm{\phi_\mathrm{f}}_{H^1_{\vct{N}} (\Gamma_a ) } . \tag*{\qedhere}
\end{align*}
%\end{linenomath}
\end{proof}

We now obtain the following convergence result.
\begin{theorem} \label{thm:aeqp1}
Let $\alpha = 1$ and $\beta \ge -1$.
Then, given that the assumption~\eqref{asm:rhoD} holds true, $\smash{ (\hat{p}_+^\ast , \hat{p}_-^\ast ) \in \Phi^0_\mathrm{IV}}$ is the unique solution of problem~\eqref{eq:weaklimit_aeqp1}, where $\smash{\hat{p}_\pm^\ast \in H^1(\Omega^0_\pm )}$ denote the limit functions from \Cref{prop:convergence}. 
\end{theorem}
\begin{proof}
Let $\smash{(\phi_+ , \phi_- ) \in \Phi^0_\mathrm{IV}}$.
We define~$\phi_\mathrm{f} \in H_{\vct{N}}^1 (\Gamma_a )$ by 
%\begin{linenomath}
\begin{align*}
\phi_\mathrm{f} ( \vct{p} , \theta_n ) &:= \phi_-\bigr\vert_{\Gamma } (\vct{p} ) + \jump{\phi}_\Gamma (\vct{p} ) K_\Gamma^\perp (\vct{p} )  \int_{-a_-(\vct{p})}^{\theta_n} [\hat{K}_\rmf^\perp]^{-1} (\vct{p}, \bar{\theta}_n ) \,\D \bar{\theta}_n ,
\end{align*}
%\end{linenomath}
where $\smash{K_\Gamma^\perp \in L^\infty_a (\Gamma )}$ is given by \cref{eq:Kperp}.
It is easy to check that $\smash{(\phi_+ , \phi_- , \phi_\rmf ) \in \Phi^\ast}$.
In particular, we have 
%\begin{linenomath}
\begin{align*}
\partial_{\theta_n } \phi_\mathrm{f} (\vct{p} , \theta_n ) &= \jump{\phi}_\Gamma (\vct{p} ) K_\Gamma^\perp (\vct{p} )  [\hat{K}_\rmf^\perp ]^{-1} (\vct{p}, \theta_n ) .
\end{align*}
%\end{linenomath}
Thus, by inserting the test function triple $\smash{(\phi_+ , \phi_- , \phi_\rmf ) }$ into the weak formulation~\eqref{eq:weaklimit_aeqp1_coupled} and by using that 
%\begin{linenomath}
\begin{align*}
\bigl( \hat{\matr{K}}_\rmf \nablaN \hat{p}_\rmf^\ast , \nablaN \phi_\rmf \bigr)_{\vct{L}^2 (\Gamma_a )} &=  \bigl( K_\Gamma^\perp \partial_{\theta_n} \hat{p}_\rmf^\ast , \jump{\phi}_\Gamma \bigr)_{L^2 (\Gamma_a ) }  = \bigl(K_\Gamma^\perp \jump{\hat{p}^\ast}_\Gamma , \jump{\phi}_\Gamma \bigr)_{L^2 (\Gamma )},
\end{align*}
%\end{linenomath}
we find that $\smash{ (\hat{p}_+^\ast , \hat{p}_-^\ast ) }$ satisfies \cref{eq:weaklimit_aeqp1}.
With \Cref{lem:PhiIV}, we have $\smash{(\hat{p}_+^\ast , \hat{p}_-^\ast ) \in \Phi_\mathrm{IV}^0}$.
The uniqueness of the solution follows from the Lax-Milgram theorem.
\end{proof}

\subsection{Case V: \texorpdfstring{$\alpha > 1$}{alpha > 1} }
\label{sec:sec45}
For $\alpha > 1$ and $2 \beta \ge \alpha - 3$, the fracture becomes a solid wall as $\epsilon \rightarrow 0$, i.e., the interface~$\Gamma$ is an impermeable barrier with zero flux across~$\Gamma$. 
The strong formulation of the limit problem reads as follows.

Find $\smash{p_\pm \colon \Omega_\pm^0 \rightarrow \mathbb{R}}$ such that
%\begin{linenomath}
\begin{subequations}
\begin{alignat}{3}
-\nabla \cdot \bigl( \matr{K}_\pm^0 \nabla p_\pm \bigr) &= q_\pm^0 \qquad &&\text{in } \Omega_\pm^0 ,  \\
\matr{K}_\pm^0 \nabla p_\pm \cdot \vct{N} &= 0 \qquad &&\text{on } \Gamma , \label{eq:stronglimit_ag1_b}  \\
p_+ &= p_- \qquad &&\text{on } \Gamma_0^0 , \\
\matr{K}_+^0 \nabla p_+ \cdot \vct{N} &= \matr{K}_-^0 \nabla p_- \cdot \vct{N} \qquad &&\text{on } \Gamma_0^0 , \\
p_\pm &= 0 \qquad &&\text{on } \rho_{\pm ,\mathrm{D}}^0 , \\
\matr{K}_\pm^0 \nabla p_\pm \cdot \vct{n} &= 0 \qquad &&\text{on } \rho_{\pm ,\mathrm{N}}^0 .
\end{alignat}%
\label{eq:stronglimit_agp1}%
\end{subequations}%
%\end{linenomath}
A weak formulation of the system in \cref{eq:stronglimit_agp1} is given by the following problem.

Find $ (p_+ , p_- ) \in \Phi^0_\mathrm{V}$ such that, for all $ ( \phi_+ , \phi_-  ) \in \Phi_\mathrm{V}^0$,
%\begin{linenomath}
\begin{align}
\label{eq:weaklimit_agp1} &\sum_{i = \pm } \bigl( \matr{K}_i^0 \nabla p_i , \nabla \phi_i )_{\vct{L}^2 (\Omega_i^0  ) } =  \sum_{i = \pm }\bigl( q_i^0 , \phi_i \bigr)_{L^2 (\Omega_i^0 )} .
\end{align}
%\end{linenomath}
Here, the space~$\Phi^0_\mathrm{V}$ is given by 
%\begin{linenomath}
\begin{align}
\Phi_\mathrm{V}^0 &:= \Bigl\{ (\phi_+ , \phi_- ) \in \bigtimes\nolimits_{i=\pm } H^1_{0,\rho_{i,\mathrm{D}}^0} (\Omega_i^0 ) \ \Big\vert \ \phi_+\bigr\vert_{\Gamma_0^0} = \phi_-\bigr\vert_{\Gamma_0^0}  \Bigr\} \cong H^1_{0, \rho^0_{\mathrm{b},\mathrm{D}} } (\Omega^0 \setminus \Gamma ) .
\end{align}
%\end{linenomath}
We now have the following convergence results.
\begin{theorem}
Let $\alpha > 1$ and $2 \beta \ge \alpha - 3$. 
Then, given the assumption~\eqref{asm:rhoD}, $\smash{ ( \hat{p}^\ast_+ , \hat{p}^\ast_- ) \in \Phi_\mathrm{V}^0}$ is a weak solution of problem~\eqref{eq:weaklimit_agp1}, where $\smash{\hat{p}^\ast_\pm\in H^1(\Omega^0_\pm )}$ denote the limit functions from \Cref{prop:convergence}.
\end{theorem}
\begin{proof}
With \Cref{prop:apriori_l2_3}, we have 
%\begin{linenomath}
\begin{align*}
\epsilon^\alpha \normsize{\big }{\nablaN \hat{p}_\mathrm{f}^\epsilon}_{\vct{L}^2 (\Gamma_a )} \lesssim \epsilon^{\alpha - 1} \normsize{\big }{\nablaN \hat{p}_\mathrm{f}^\epsilon}_{\vct{L}^2 (\Gamma_a )} &\lesssim \epsilon^\frac{\alpha - 1}{2} , \\
\epsilon^{\alpha + 1} \normsize{\big }{\nablaGamma \hat{p}_\mathrm{f}^\epsilon }_{\vct{L}^2 (\Gamma_a )} \lesssim \epsilon^{\alpha} \normsize{\big }{\nablaGamma \hat{p}_\mathrm{f}^\epsilon }_{\vct{L}^2 (\Gamma_a )} &\lesssim \epsilon^\frac{\alpha - 1}{2} .
\end{align*}
%\end{linenomath}
Thus, with the \Cref{lem:Roperators,lem:limApm}, the result follows by letting~$k\rightarrow \infty$ in the transformed weak formulation~\eqref{eq:rescaledweakdarcyeps}.
\end{proof}

\begin{theorem} \label{thm:agp1}
Let $\alpha > -1 $ and $2 \beta \ge \alpha - 3 $. 
Then, given the assumption~\eqref{asm:rhoD}, we have strong convergence 
%\begin{linenomath}
\begin{alignat}{2}
\hat{p}_\pm^{\epsilon } &\rightarrow \hat{p}_\pm^\ast \quad\enspace &&\text{in } H^1 (\Omega_\pm^0 ) \label{eq:strongconv_agp1}
\end{alignat}
%\end{linenomath}
as $\epsilon \rightarrow 0$ for the whole sequence~$\smash{\{ \hat{p}_\pm^\epsilon \}_{\epsilon\in (0,\hat{\epsilon}]}}$.
Moreover, $\smash{ ( \hat{p}_+^\ast , \hat{p}_-^\ast ) \in \Phi_\mathrm{V}^0}$ is the unique weak solution of the problem in \cref{eq:weaklimit_agp1}.
\end{theorem}
\begin{proof}
The result follows with analogous arguments as in the cases above.
\end{proof}

\appendixtitleon
\begin{appendices}
\section{Geometric Background} \label{sec:secA}
In the following, we summarize useful definitions and results related to the geometry of Euclidean submanifolds. 

\subsection{Orthogonal Projection and Signed Distance Function} \label{sec:secA1}
We introduce the orthogonal projection and (signed) distance function of a set and state selected properties and regularity results. 
For details, we refer to \cite{leobacher21}.
\begin{definition} \label{def:unpp}
Let $\emptyset \not= M \subset \mathbb{R}^n$.
\begin{enumerate}[wide,topsep=0pt,label=(\roman*)]
\item We write $\smash{d^M \colon \mathbb{R}^n \rightarrow [0, \infty )}$, $\smash{d^M ( \vct{x} ) := \inf_{\vct{p}\in M} \abs{\vct{x} - \vct{p}}}$ for the distance function of~$M$.
If $M = \partial A \not= \emptyset$ for a set $A \subset \mathbb{R}^n$, we can define the signed distance function of~$M$ by  
%%\begin{linenomath}
\begin{align}
d_\leftrightarrow^M \colon \mathbb{R}^n \rightarrow \mathbb{R} , \quad d^M_\leftrightarrow ( \vct{x} ) := \begin{cases}
d^M ( \vct{x} ) &\text{if } \vct{x} \in A , \\
-d^M ( \vct{x} ) &\text{if } \vct{x} \in \mathbb{R}^n \setminus A .
\end{cases}
\end{align}
%%\end{linenomath}
\item A set~$A \subset \mathbb{R}^n$ is said to have the unique nearest point property with respect to~$M$ if, for all $\vct{x}\in A$, there exists a unique~$\vct{p} \in M$ such that $\smash{d}^M (\vct{x} ) = \abs{\vct{x} - \vct{p}}$. 
We write $\mathrm{unpp}(M)$ for the maximal set with this property.
\item We define the orthogonal projection onto~$M$ by
%%\begin{linenomath}
\begin{align}
\mathcal{P}^M \colon \mathrm{unpp}(M) \rightarrow M , \quad \vct{x} \mapsto \argmin_{\vct{p}\in M} \abs{\vct{x} - \vct{p}} .
\end{align}
%%\end{linenomath}
\item Let $\delta > 0$. 
Then, we define the $\delta$-neighborhood of~$M$ by
%%\begin{linenomath}
\begin{align}
U_\delta (M) &:= \bigl\{ \vct{x} \in \mathbb{R}^n \ \big\vert\ d^M(\vct{x} ) < \delta  \bigr\} . 
\end{align}
%%\end{linenomath}
For $\vct{x} \in \mathbb{R}^n$, we also write $U_\delta ( \vct{x} ) := U_\delta ( \{ \vct{x} \} )$.
\item We define the reach of $M$ by 
%%\begin{linenomath}
\begin{align}
\mathrm{reach}(M) &:= \sup \bigl\{ \delta > 0 \ \big\vert \ U_\delta (M) \subset \mathrm{unpp} (M) \bigr\} .
\end{align}
%%\end{linenomath}
\end{enumerate}
\end{definition}
Let $M \subset \mathbb{R}^n$ be a $\mathcal{C}^k$-submanifold, $k\in \mathbb{N}$.
Then, the orthogonal projection~$\mathcal{P}^M$ is $\mathcal{C}^{k-1}$-differentiable on $\mathrm{unpp}(M)^\circ$~\cite[Thm.\ 2]{leobacher21}. 
If $k\ge 2$, we have $\mathcal{P}^M (\vct{p} + \vct{n} ) = \vct{p}$ for $\vct{p}\in M$ and $\vct{n} \perp \rmT_\vct{p}M$ with $\vct{p} +  \vct{n} \in \mathrm{unpp}(M)^\circ$~\cite[Prop.\ 2]{leobacher21}.
Besides, if $M$ is compact and $k\ge 2$, we have $\mathrm{reach} (M) > 0$~\cite[Prop.\ 6]{leobacher21}.
Moreover, if $M = \partial A $ for a set $A \subset \mathbb{R}^n$ of class $\mathcal{C}^k$, $k \ge 2$, the signed distance function $d_\leftrightarrow^{\partial A}$ is $\mathcal{C}^k$-differentiable on~$\mathrm{unpp}(\partial A)^\circ$ (cf. \cite[Thm.\ 7.8.2]{delfour11} and \cite[Thm.\ 2]{leobacher21}).

\subsection{Shape Operator} \label{sec:secA2}
Let $2 \le k \in \mathbb{N}$ and $M \subset \mathbb{R}^n$ be an $(n-1)$-dimensional $\mathcal{C}^k$-sub\-ma\-ni\-fold with a global unit normal vector field~$\vct{N}\in \mathcal{C}^{k-1} (M ; \mathbb{R}^n )$. 
We define the shape operator~$\mathcal{S}_\vct{p}$ of~$M$ at~$\vct{p}\in M$ for each~$\vct{v} \in \rmT_\vct{p}M$ as the negative directional derivative $\smash{\mathcal{S}_\vct{p} (\vct{v} ) := -\nabla_{\!\vct{v}} \vct{N}(\vct{p})}$.
Then, for each $\vct{p} \in M$, the shape operator~$\mathcal{S}_\vct{p}$ is a self-adjoint linear operator $\smash{\mathcal{S}_\vct{p} \colon \rmT_\vct{p} M \rightarrow \rmT_\vct{p} M}$.
The eigenvalues~$\kappa_1 (\vct{p} ), \dots , \kappa_{n-1} (\vct{p} ) $ of the shape operator~$\mathcal{S}_\vct{p}$ are called the principal curvatures of~$M$ at~$\vct{p} \in M$. 
In particular, we have $\kappa_1 , \dots , \kappa_{n-1} \in \mathcal{C}^{k-2} (M)$. 

\subsection{Function Spaces on Manifolds} \label{sec:secA3}
Let $M \subset \mathbb{R}^n$ be an $m$-dimensional $\smash{\mathcal{C}^{0,1} }$-submanifold with boundary~$\partial M$. 
We denote charts for~$M$ as triples~$\smash{( U , \vct{\psi}, V) }$, i.e., $U \subset M$ and $\smash{V \subset \mathbb{R}^m}$ (or $V \subset \mathbb{R}^{m-1} \times [0, \infty )$ for charts with boundary) are open and~$\vct{\psi }\colon U \rightarrow V$ is bi-Lipschitz.
For the inverse chart~$\smash{\vct{\psi}^{-1}}$, we also use the symbol~$\invpsi$.
Besides, we write $\smash{\matr{g}\vert^{\vct{\psi}}}$ for the metric tensor in coordinates of the chart~$\smash{\vct{\psi} }$, i.e., $\smash{\matr{g}\vert^{\vct{\psi} } (\vct{\theta} ) = [ \matr{D} \invpsi (\vct{\theta }) ]^\trp \matr{D} \invpsi (\vct{\theta }) \in \mathbb{R}^{m\times m }}$.
For $p \in [1,\infty ]$, we write $L^p (M)$ for the Lebesgue space on~$M$ with respect to the Riemannian measure~$\lambda_M$. 
Moreover, we define $\smash{\vct{L}^p (M) := L^p (M)^m}$. 
Following \cite{hebey00}, we define the first-order Sobolev space~$H^1 (M)$ as the completion of 
%%\begin{linenomath}
\begin{align}
 \bigl\{ f \in \mathcal{C}^{0,1} (M)  \ \big\vert\ \norm{f}_{H^1(M)} < \infty   \bigr\} 
\end{align}
%%\end{linenomath}
with respect to the norm $\smash{\norm{f}_{H^1(M)}^2 :=  \norm{f}_{L^2 (M ) }^2 +  \norm{\nabla_{\! M} f}^2_{\vct{L}^2 (M) }}$, where $\smash{\nabla_{\! M} f}$ denotes the gradient of~$f$. 
In local coordinates, we have
%%\begin{linenomath}
\begin{align}
\nabla_{\! M} f \bigl( \invpsi  (\vct{\theta } ) \bigr)  &=  \matr{D} \invpsi (\vct{\theta} )  \, \matr{g}^{-1} \big\vert^{\vct{\psi}} (\vct{\theta} ) \nabla (f \circ \invpsi ) (\vct{\theta }) .
\end{align}%
%%\end{linenomath}
Besides, $\smash{H^1 (M)}$ is a reflexive Hilbert space. 
For the more general case of Sobolev spaces~$\smash{W^{k,p}(M)}$ of arbitrary order~$k\in\mathbb{N}$ and $1 \le p < \infty$ on Riemannian manifolds, we refer to \cite{hebey00}. 
Further, if $M$ is compact, we can alternatively define the Sobolev space~$\smash{H^1 (M)}$ by using local coordinates~\cite{wloka87}.
Given a finite atlas~$\smash{\{( U_i , \vct{\psi}_i , V_i ) \}_{i\in I}}$ of~$M$ and a subordinate partition of unity~$\smash{\{ \chi_i \}_{i\in I }} \subset \mathcal{C}^{0,1} (M)$, we define the space
%%\begin{linenomath}
\begin{align}
H^1 (M) &:= \bigl\{ f \in L^2 (M) \ \big\vert \ ( \chi_i f ) \circ \invpsi_i \in H^1 (V_i ) \bigr\} 
\end{align}
%%\end{linenomath}
with the norm $\smash{\norm{f}_{H^1(M) }^2 :=  \sum_{i \in I } \norm{( \chi_i f ) \circ \invpsi_i }_{H^1 (V_i) }^2 }$.
It is easy to check that the two definitions for $\smash{H^1 (M)}$ are equivalent.
Consequently, it is  $\smash{ H^1 (M) = H^1 (\operatorname{Int} (M) )}$, where $\operatorname{Int}(M)$ denotes the interior of~$M$.
Moreover, with analogous arguments as in \cite[\S 11]{booss93}, one can prove the following trace theorem.
\begin{lemma}  \label{lem:manifoldtrace}
Let $\partial M $ be compact.
Then, there exists a unique bounded linear operator $\smash{\mathfrak{T}_M \colon H^1(M) \rightarrow L^2 (\partial M ) }$ such that $\smash{\mathfrak{T}_M f = f\vert_{\partial M } }$ for all $\smash{f \in  H^1 (M) \cap \mathcal{C}^0 (M)}$.
\end{lemma}

\end{appendices}

\section*{Acknowledgments}
This work was funded by the Deutsche Forschungsgemeinschaft (DFG, German Research Foundation) -- project number~327154368 -- SFB 1313, and under Germany's Excellence Strategy -- EXC 2075 -- 390740016. The authors also were supported by the Stuttgart Center for Simulation Science (SimTech).

%%%%%%%%%%%%%%%%%%%%%%%%%%%%%%%%%%%%%%%%%%%%%%%%%%%%%%
%          7. REFERENCES SECTION
%%%%%%%%%%%%%%%%%%%%%%%%%%%%%%%%%%%%%%%%%%%%%%%%%%%%%%

\printbibliography[heading=none]

%\medskip
% The information below will be filled in by AIMS production staff.
%Received xxxx 20xx; revised xxxx 20xx; early access xxxx 20xx.
%\medskip

\end{document}